%% file: arrow2.tex
\documentclass[11pt]{amsart}

\pdfoutput=1

\input{preamble}

\def\nullhomologous{null\-ho\-mol\-o\-gous}
\def\homologous{ho\-mol\-o\-gous}

\begin{document}

\title{The homological arrow polynomial for virtual links}

\author{Kyle A. Miller}
\address{Department of Mathematics\\ University of California, Santa Cruz\\ California 95064, USA}
%\curraddr{}
\email{kymiller@ucsc.edu}
% \thanks{}
% ORCID 0000-0001-7400-5304


\date{July 2022, revised February 2023}

%\dedicatory{}

\maketitle

\begin{abstract}
  The arrow polynomial is an invariant of framed oriented virtual links that generalizes the virtual Kauffman bracket.
  In this paper we define the \emph{homological arrow polynomial}, which generalizes the arrow polynomial to framed oriented virtual links with labeled components.
  The key observation is that, given a link in a thickened surface, the homology class of the link defines a functional on the surface's skein module, and by applying it to the image of the link in the skein module this gives a virtual link invariant.

  We give a graphical calculus for the homological arrow polynomial by taking the usual diagrams for the Kauffman bracket and including labeled ``whiskers'' that record intersection numbers with each labeled component of the link.

  We use the homological arrow polynomial to study $(\Z/n\Z)$-\nullhomologous{} virtual links and checkerboard colorability, giving a new way to complete Imabeppu's characterization of checkerboard colorability of virtual links with up to four crossings.
  We also prove a version of the Kauffman--Murasugi--Thistlethwaite theorem that the breadth of an evaluation of the homological arrow polynomial for an ``h-reduced'' diagram $D$ is $4(c(D)-g(D)+1)$.
\end{abstract}

\tableofcontents

\newpage

\section{Introduction}
\label{sec:introduction}

\subsection{Virtual links}

A \emph{surface link} is a smooth oriented link $L$ in the interior of a thickened surface $\Sigma\times I$, where $\Sigma$ is a compact oriented surface possibly with nonempty boundary.
The usual notion of equivalence for a surface link is isotopy, but we say surface links $L\subset\Sigma\times I$ and $L'\subset\Sigma'\times I$ are \emph{weakly equivalent} if there exists an orientation-preserving diffeomorphism $\Sigma\times I\to \Sigma'\times I$ carrying $L$ to $L'$, $\Sigma\times\{0\}$ to $\Sigma'\times\{0\}$, and $\Sigma\times\{1\}$ to $\Sigma'\times\{1\}$, each respecting orientations.
Isotopic surface links are weakly equivalent, and furthermore weak equivalence of links in a fixed thickened surface $\Sigma\times I$ is the equivalence relation generated by both link isotopy and the induced action of the mapping class group of $\Sigma$.

Given a compact subsurface $\Sigma'\subseteq\Sigma$ with $L\subset\Sigma'\times I$, then the surface link $L\subset \Sigma'\times I$ is called a \emph{destabilization} of $L\subset \Sigma\times I$.
\emph{Virtual equivalence} is the equivalence relation\footnote{A set theory consideration: there is no set of all compact oriented surfaces, but it is sufficient to consider a minimal set that contains one closed connected surface of every genus and that is closed under taking compact subsurfaces and finite disjoint unions. Then the equivalence relation is on the set of surface links for these surfaces.} generated by weak equivalence and destabilization, and a \emph{virtual link} is an equivalence class with respect to virtual equivalence\cite{Carter2002a,Kuperberg2003}.
The \emph{virtual genus} of a virtual link is the minimal genus of $\Sigma$ over all representative surface links $L\subset \Sigma\times I$ with $\Sigma$ closed.
A virtual link is \emph{classical} if its virtual genus is $0$, and by \cite{Kuperberg2003} classical virtual links are in one-to-one correspondence with links in $S^3$ up to isotopy.

The combinatorial data of a virtual link can be given as a \emph{virtual link diagram}, which is a ribbon graph (i.e., a combinatorial map) with oriented edges whose vertices each have degree $2$ or $4$ and whose degree-$4$ vertices have a marked pair of opposite half-edges corresponding to the overstrand of a crossing, furthermore edges opposite each other at each vertex have compatible orientations.
The degree-$2$ vertices correspond to arbitrary arc subdivisions, which allow for unknots to be represented in this way.
(This is the data of a PD code with P, Xp, and Xm terms --- see \Cref{sec:computing-arrow}.)
Virtual links are equivalent if and only if their virtual link diagrams are related by the Reidemeister moves and arc subdivision.
Each virtual link diagram has a \emph{cellular embedding} in a closed oriented surface by gluing disks to the boundary components of the geometric realization of the underlying ribbon graph.
The \emph{genus} of a virtual link diagram is the genus of this surface, and the minimal genus over all virtual link diagrams corresponds to the virtual genus of the virtual link, and so a virtual link is classical if there exists a genus-$0$ (or \emph{planar}) virtual link diagram for it.

Virtual link diagrams are usually immersed in the plane with normal crossings.
The double points --- artefacts of non-planarity --- are called \emph{virtual crossings} and are usually marked with a small circle to make them obvious.
Such diagrams up to the classical Reidemeister moves and \emph{detour moves} (re-immersions of the ribbon graph) generate virtual equivalence\cite{Kauffman1999}.

\begin{figure}[tb]
  \centering
  \hspace*{-1cm}
  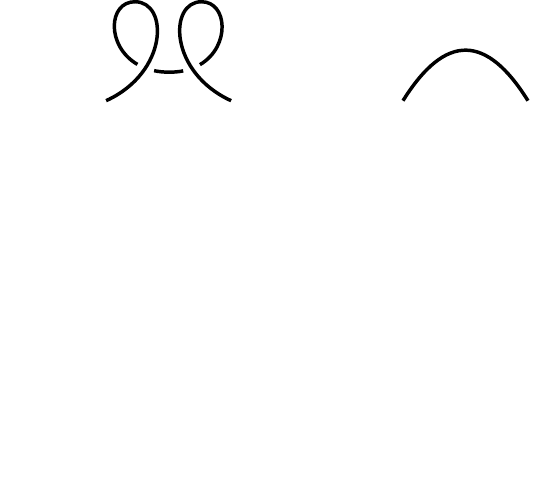
  \caption{Reidemeister moves for framed surface links.}
  \label{fig:framed-reidemeister}
\end{figure}

A \emph{framed} surface link is a surface link $L\subset\Sigma\times I$ along with an extension to an embedding of $L\times I$ in the interior of $\Sigma\times I$, up to isotopy.
Note that the embedding does not usually respect the product structure of $\Sigma\times I$.
By extension, a \emph{framed} (or \emph{flat}) virtual link is a framed surface link up to the equivalence relation generated by weak equivalence and destabilization.
Diagrams for framed virtual links are virtual link diagrams with the convention that the framing annuli lie parallel to the plane.
The only change to the Reidemeister theorem for framed virtual links is that the Reidemeister I move is replaced by the Reidemeister I\textquotesingle{} move for regular isotopy (see \Cref{fig:framed-reidemeister}).

\subsection{The arrow polynomial}

Discovered independently by Dye--Kauffman\cite{Dye2008} and Miyazawa \cite{Miyazawa2006,Miyazawa2008}, the arrow polynomial is an invariant for oriented framed virtual links of at least one component that takes values in the polynomial ring $\Z[A^{\pm 1},K_1,K_2,\dots]$.
In the Dye--Kauffman formulation, the Kauffman bracket is modified to make use of the orientation of the virtual link by introducing cusps in the $B$-smoothing (see \Cref{fig:arrow-bracket-1}).
The evaluation of a fully resolved state as an element of the polynomial ring is from counting cusps after certain cusp cancelation rules.
To explain this process we use the Miyazawa formulation.
First, replace the cusps with \emph{vertex orientations} according to \Cref{fig:arrow-bracket-cusps}; these are local orientations at an integer-labeled degree-$2$ vertex.
Second, reduce the vertex orientations according to \Cref{fig:arrow-bracket-cusps}, leaving a single vertex orientation along each component of the state.
Third, evaluate the state as $(-A^2-A^{-2})^{b_0(S)-1} \prod_{C\subseteq S}K_{\abs{n(C)}/2}$, where, setting $K_0=1$,
\begin{enumerate}
\item $b_0(S)$ is the number of components in the state $S$,
\item $C$ ranges over components of the state, and
\item $n(C)$ is the label of the sole vertex orientation along that component (which is always even, as we will come to see).
\end{enumerate}
Hence, the arrow polynomial of $L$ given a virtual link diagram $D$ has the formula
\begin{equation*}
  \langle D \rangle_{\mathrm{A}} = \sum_{S} A^{a(S)-b(S)} (-A^2-A^{-2})^{b_0(S)-1} \prod_{C\subseteq S} K_{\abs{n(C)}/2},
\end{equation*}
where $a(S)$ and $b(S)$ respectively denote the numbers of $A$-smoothings and $B$-smoothings in state $S$.
This is independent of the diagram, so we are justified in defining $\langle L \rangle_{\mathrm{A}} = \langle D \rangle_{\mathrm{A}}$.

\begin{figure}[tb]
  \centering
  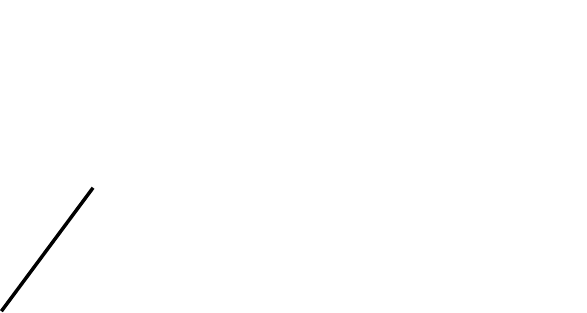
  \caption{Crossing resolutions for the Dye--Kauffman arrow polynomial, which introduce cusps.}
  \label{fig:arrow-bracket-1}
\end{figure}

\begin{figure}[tb]
  \centering
  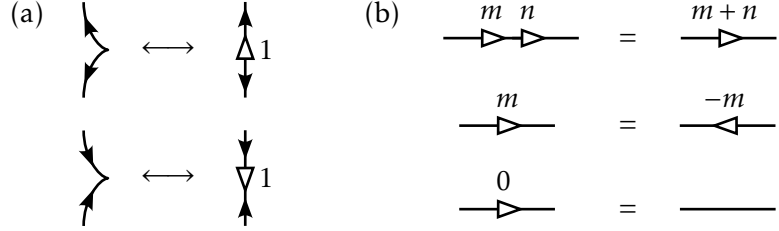
  \caption{(a) Graphical notation for arrow polynomial cusps as integer-labeled vertex orientations, similar to the decorated magnetic virtual graphs of \cite{Miyazawa2008}. (b)~Rules for integer-labeled vertex orientations after erasing arc orientations.}
  \label{fig:arrow-bracket-cusps}
\end{figure}

For virtual links in general, Dye and Kauffman define the writhe-normalized arrow polynomial $\langle D\rangle_{\mathrm{NA}} = (-A^3)^{-\writhe(D)}\langle D\rangle_{\mathrm{A}}$, where the \emph{writhe} is the sum of the signs of the crossings of the virtual link diagram, which makes the polynomial be invariant under the usual Reidemeister I move.
Said another way, the writhe-normalized arrow polynomial is the arrow polynomial of a writhe-$0$ representative in the equivalence class of framed virtual links modulo the usual Reidemeister I move.
One might define the arrow version of the virtual Jones polynomial by substituting $t=A^{-4}$ into $\langle L\rangle_{\mathrm{NA}}$.
The usual virtual Jones polynomial, then, is from additionally setting $K_n=1$ for all $n$.

\begin{figure}[tb]
  \centering
  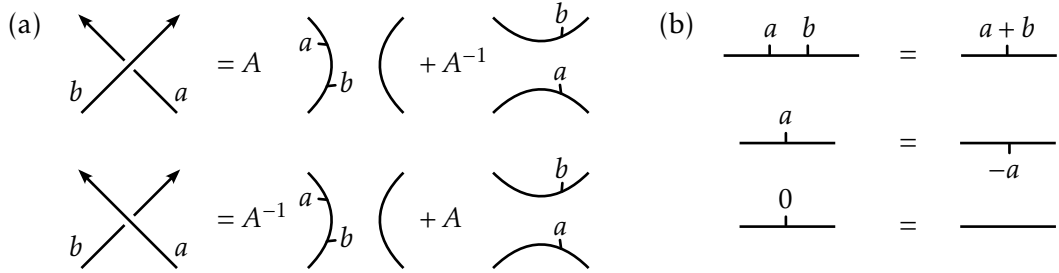
  \caption{(a) Expansion for the homological arrow polynomial with components labeled by $a,b\in G$, using the left-pushoff convention. Whiskers indicate local contributions to algebraic intersection numbers and are labeled by an element of the abelian group $G$. (b) Whisker rules. Loops with a single whisker labeled by $a\in G$ evaluate to $(-A^2-A^{-2})X_{\pm a}$ with $X_0=1$. }
  \label{fig:whisker-expansion}
\end{figure}

\subsection{Overview}

In this paper, we give a formulation of the arrow polynomial from a homological point of view, similar in flavor to Turaev's two-variable extension of the Jones polynomial to knotoids\cite{Turaev2010}.
This makes many facts about the arrow polynomial transparent while also giving a straightforward extension to virtual links with labeled components.

Taking a surface link representative $L\subset\Sigma$ of a virtual link, we may consider its image in the Kauffman bracket skein module $\Sk(\Sigma)$ for the thickened surface.
This is not in itself an invariant of the virtual link for the simple reason that it depends on the surface, but, even if the minimal-genus closed surface representative were used, the image is not invariant under the action of the mapping class group, which would be essential for a virtual link invariant.
The main idea is to choose a functional on the Kauffman bracket skein module that is derived from the link itself, and this is done in such a way that when it is evaluated on the image $[L]$ in $\Sk(\Sigma)$ the value is independent of the representative surface link.

We give an overview of how this works, with a detailed description in \Cref{sec:homological-arrow-poly}.
Let $L\subset \Sigma\times I$ be a surface link, let $G$ be an abelian group, and let $\lambda:H_0(L)\to G$ be a group homomorphism, which we regard as a labeling of the oriented components of $L$ by elements of $G$
(the choice of $G=H_0(L)$ and $\lambda=\id$ is most general).
The homological arrow polynomial for $L$ is an invariant in the polynomial ring $R_G=\Z[A^{\pm 1}][X_{\pm g}\mid g\in G]/(X_0-1)$.
We define a homomorphism $h_{L,\lambda}:H_1(\Sigma)\to G$ using algebraic intersection numbers in $\Sigma$ via the following.
Recall that the algebraic intersection number in $\Sigma$ defines a homomorphism $H_1(\Sigma)\otimes H_1(\Sigma)\to \Z$, which for $\alpha,\beta\in H_1(\Sigma)$ we denote the image of $\alpha\otimes\beta$ by $\alpha\cdot\beta$.
Letting $\pi:\Sigma\times I\to\Sigma$ be the canonical projection, we may define $H_1(L)\otimes H_1(\Sigma)\to\Z$ by $\alpha\otimes\beta \mapsto \pi_*(\alpha)\cdot \beta$, and then, using the hom-tensor adjunction and the fact that there are canonical isomorphisms $\Hom(H_1(L),\Z)\cong H^1(L)\cong H_0(L)$, we get a homomorphism $H_1(\Sigma)\to H_0(L)$, which we may then compose with $\lambda$ to obtain $h_{L,\lambda}$.
Next, $\overline{h}_{L,\lambda}:\Sk(\Sigma)\to R_G$ is the $\Z[A^{\pm 1}]$-linear homomorphism defined on simple skeins, where on simple skeins $\overline{h}_{L,\lambda}$ is multiplicative with respect to disjoint unions and, for simple closed curves $C\subseteq \Sigma\times\{1/2\}$, we define $\overline{h}_{L,\lambda}([C])=(-A^2-A^{-2})X_{\pm h_{L,\lambda}([C])}$ (this $\pm$ is due to the fact that skeins in $\Sk(\Sigma)$ are unoriented).
With $h_{L,\lambda}$ in hand, the \emph{homological arrow polynomial} is $\mathcal{A}(L,\lambda):= \overline{h}_{L,\lambda}([L])$.
The special case of $\lambda:H_0(L)\to \Z$ sending each component to $1$ yields a polynomial that is equal to $-A^{2}-A^{-2}$ times the usual arrow polynomial with $X_{\pm n}=K_{\abs{n}/2}$ (the arrow polynomial is normalized such that the value of the unknot is $1$).

The homological arrow polynomial has a graphical calculus (see \Cref{fig:whisker-expansion}) similar to the ones for the arrow polynomial, and it lets us calculate $h_{L,\lambda}$ in a local way.
Since $L$ is oriented, then in a diagram for $L$ on $\Sigma$ there is a well-defined normal bundle for $\pi(L)$ that we can use to push off $\pi(L)$, the effect of which is that when we do the usual Kauffman bracket expansion we may assume that each curve intersects the pushed-off $\pi(L)$ only near crossings.
We indicate such intersection points using a \emph{whisker} that records the orientation of $L$ at that point along with the value of $\lambda$ for the corresponding component of $L$.
A whisker is essentially the same data as the vertex orientations from \Cref{fig:arrow-bracket-1} but using general abelian groups beyond $\Z$.

After defining the polynomial in \Cref{sec:homological-arrow-poly}, the main part of this paper is \Cref{sec:properties}, where basic properties are proved.
The graphical calculus of \Cref{fig:whisker-expansion} is used to reprove and generalize known properties of the arrow polynomial, where the homological point of view helps simplify arguments.
We investigate properties of the homological arrow polynomial for links that are \nullhomologous{} in $H_1(\Sigma\times I;R)$ for varying rings $R$.
This includes $\Z$-\nullhomologous{} links (i.e., \emph{almost classical links}) and $\Z/2\Z$-\nullhomologous{} links (i.e., \emph{checkerboard-colorable links}).

\begin{figure}[tb]
  \centering
  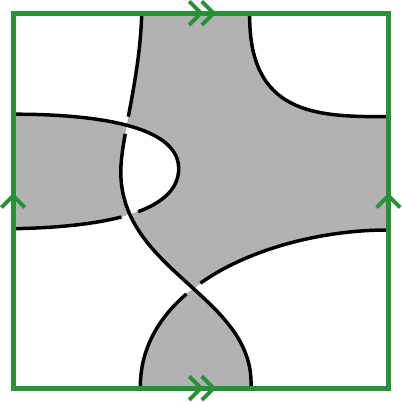
  \caption{The virtual knot $3.5$ is checkerboard colorable and has virtual genus $1$.
    Here it is drawn on an identification diagram for a torus with a shaded checkerboard surface.
  Its arrow polynomial is $\langle 3.5\rangle_{\mathrm{A}} = -A^{-5}+A^{-1}-A^3+(-A^{-1}+A^7)K_1^2$.}
  \label{fig:vknot_3_5}
\end{figure}

Incidentally, we clarify \cite[Theorem 10.8]{Dye2016}, which is incorrect as stated.
The \emph{arrow set} $\operatorname{AS}(L)$ of a framed virtual link $L$ is the set
\[
  \operatorname{AS}(L) = \{k_1+\dots+k_n \mid pK_{k_1}K_{k_2}\cdots K_{k_n}\text{ is a nonzero monomial in } \langle L \rangle_{\mathrm{A}} \},
\]
and $\operatorname{AS}(L)=\{0\}$ if and only if the arrow polynomial of $L$ is equal to its virtual Kauffman bracket.
What \cite[Theorem 10.8]{Dye2016} says is that $\Z/2\Z$-\nullhomologous{} links have $\operatorname{AS}(L)=\{0\}$, however the checkerboard colorable virtual knot $3.5$ has $-A^{-5}+A^{-1}-A^3+(-A^{-1}+A^7)K_1^2$ as its arrow polynomial (see \Cref{fig:vknot_3_5}), so $\operatorname{AS}(L)=\{0,2\}$.
One way to get the same conclusion is by using \Cref{thm:z-nullhomologous-poly}, which implies that if $L$ is $\Z$-\nullhomologous{} then $\operatorname{AS}(L)=\{0\}$.
Or, using the same hypothesis, \Cref{thm:deng2020} implies that if $L$ is $\Z/2\Z$-\nullhomologous{} then $\operatorname{AS}(L)\subset 2\N$.

In \Cref{sec:checkerboard}, we show how cabled homological arrow polynomials can be used to complete Imabeppu's conjectured characterization\cite{Imabeppu2016} of checkerboard colorability of all virtual knots with up to four crossings.
Note that this was already settled in \cite{Boden2017} for virtual knots with up to six crossings using mod-$2$ Alexander numberings, but their results require a minimal-crossing diagram whereas obstructions from cabled homological arrow polynomials do not.

We also consider in \Cref{sec:alternating-links} a Kauffman--Murasugi--Thistlethwaite type theorem, where if an alternating virtual link is ``$h$-reduced'' then the $A$-breadth of a small variation of the homological arrow polynomial is $4(c(D)-g(D))$, where $c(D)$ is the number of crossings in a diagram $D$ and $g(D)$ is its genus.
This strengthens a previous generalization in \cite{Kamada2004}.
While it's not as strong as the generalization in \cite{Boden2019}, their polynomial depends on finding the minimal-genus representative for the virtual link.
Some recent related work is \cite{Boden2020b}, which proves the first and second Tait conjectures for adequate links in thickened surfaces, and \cite{Boden2020}, which proves that minimal crossing number diagrams can be realized with a diagram realizing the virtual genus.

Finally, in \Cref{sec:computational-investigations} we compute cabled arrow polynomials to distinguish all virtual knots with up to four crossings.
There are nine pairs of five-crossing virtual knots that the cabled arrow polynomials up to the third do not distinguish (see \Cref{tab:arrow-non-unique-5-3}).
These computations were corroborated by writing the program from scratch in two different programming languages, JavaScript and Lean~4.
We plan to use the proof assistant capabilities of Lean~4 to have a machine-checked proof that it correctly computes a virtual link invariant.

{We leave it for future work to further elaborate the connections between the homological arrow polynomial and other virtual knot invariants.
  For example, index-style polynomials such as the writhe polynomial can be written in terms of intersection numbers of the surface link diagram with the \emph{distinguished half} associated to $A$-smoothings of single crossings\cite{Boden2017a,Cheng2020}.
}

When referring to virtual knots, we use Jeremy Green's census\cite{Green2004}.
For identification of virtual knots, we use the software KnotFolio\cite{KnotFolio} {at \url{https://kmill.github.io/knotfolio/}}.

\section{The homological arrow polynomial}
\label{sec:homological-arrow-poly}

We will first consider framed oriented surface links in a fixed thickened surface $\Sigma\times I$, define a functional on the skein module of $\Sigma\times I$, and then show the result is invariant under stabilization of the surface, and hence is an invariant of framed virtual links.

The \emph{Kauffman bracket skein module} $\Sk(M)$ for a $3$-manifold $M$ consists of $\Z[A^{\pm 1}]$-linear combinations of framed unoriented links in $M$, called \emph{skeins}, modulo the relations for the Kauffman bracket:
\begin{equation*}
  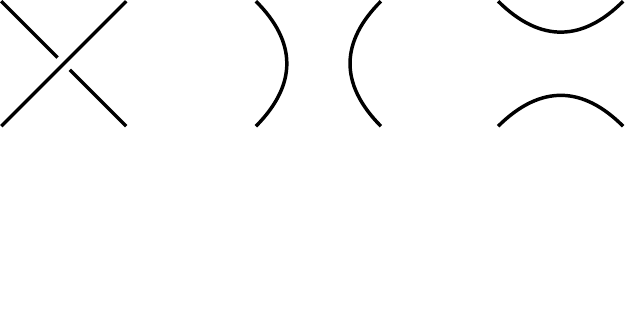
\end{equation*}
In the second relation, the unknot bounds a disk in $M$ disjoint from the rest of the link.
The skein module $\Sk(\Sigma)\defeq\Sk(\Sigma\times I)$ has a $\Z[A^{\pm 1}]$-algebra structure, where the multiplication operation is from including the first skein in $\Sigma\times[\frac{1}{2},1]$ and the second in $\Sigma\times[0,\frac{1}{2}]$.
See \cite{Przytycki1998}.

Using the inclusion $\Sigma\incl\Sigma\times I$ that sends $\Sigma$ to $\Sigma\times\{1/2\}$, every simple closed multicurve in $\Sigma$, when included in $\Sigma\times I$, yields a skein in $\Sk(\Sigma)$ called a \emph{simple skein}.
Evidently, every skein in $\Sk(\Sigma)$ is a $\Z[A^{\pm 1}]$-linear combination of simple skeins.

Consider the ring $\mathcal{H}(\Sigma)=\Z[A^{\pm 1}][X_{\pm\gamma}\mid\gamma\in H_1(\Sigma)]/(X_0-1)$, where the $\pm\gamma$ indicates that the variables are associated to unoriented homology classes.
In the following lemma, we define a preliminary invariant of framed oriented surface links.
This invariant is very similar to the surface bracket polynomial in \cite{Dye2004}, however we expand homology classes multiplicatively using the components of each state.

\begin{theorem}
  \label{thm:sfc-poly-1}
  Let $\Sigma$ be a compact oriented surface and let $\mathcal{S}(\Sigma)$ be the set of simple closed multicurves in $\Sigma$.
  Consider the mapping $\mathcal{S}(\Sigma)\to \mathcal{H}(\Sigma)$ defined by
  \[ S \mapsto \prod_{C\in S}(-A^2-A^{-2})X_{\pm[C]},
  \]
  where $C$ ranges over the components of $S$ and $\pm[C]$ is the unoriented homology class of $C$ in $H_1(\Sigma)$.
  Thinking of $\mathcal{S}$ as simple skeins in $\Sk(\Sigma)$, this mapping extends to a $\Z[A^{\pm 1}]$-linear map $\Sk(\Sigma)\to \mathcal{H}(\Sigma)$.
\end{theorem}
\begin{proof}
  Given a framed link diagram $L$ in $\Sigma$, then by applying the Kauffman bracket skein relation one can represent its value in $\Sk(\Sigma)$ as a linear combination of simple skeins, and as usual the decomposition into simple skeins is given by a state sum.
  The hypothesized extension of the map when applied to $[L]\in\Sk(\Sigma)$ is given by the following state sum:
  \begin{equation}\label{eq:state-sum}
    [L] \mapsto \sum_{S} A^{a(S)-b(S)}(-A^2-A^{-2})^{b_0(S)}\prod_{C\in S}X_{\pm [C]},
  \end{equation}
  where the sum ranges over Kauffman states for $L$, with $a(S)$ and $b(S)$ respectively denoting the numbers of $A$-type and $B$-type smoothings and $b_0(S)$ is the number of connected components of $S$.
  What remains is to show that this is invariant under the framed Reidemeister moves (\Cref{fig:framed-reidemeister}).
  Note that in the above state sum, \nullhomologous{} $C\in S$ have $\pm[C]=0$ and hence contribute $-A^2-A^{-2}$ without any additional $X_{\pm\gamma}$ variables.
  Thus, the usual check that the Kauffman bracket is invariant under framed Reidemeister moves carries over to this new state sum, and we are done.
\end{proof}

The next step is to use the framed surface link $L\subset \Sigma\times I$ itself to construct an invariant in a ring that does not depend on $\Sigma$.
Suppose $G$ is an abelian group and $\lambda:H_0(L)\to G$ is a group homomorphism, called a \emph{$G$-labeling} of $L$.
We may regard this homomorphism as being a labeling of the components of $L$ since, if $L$ has $n$ components and $L=L_1\sqcup\dots\sqcup L_n$, then $\lambda$ is determined by the values $\lambda_1,\dots,\lambda_n$ of its evaluation at the respective connected components $L_1,\dots,L_n$ of $L$.

Recall that for a compact oriented $n$-manifold $M$ there is a bilinear form $H_k(M)\otimes H_{n-k}(M)\to H_0(M)$ from applying Poincar\'e duals to the cup product $H^k(M,\partial M)\otimes H^{n-k}(M,\partial M)\to H^{n}(M,\partial M)$.
The map ${M}\to {*}$ induces a homomorphism $H_0(M)\to \Z$, and the composition $H_k(M)\otimes H_{n-k}(M)\to \Z$ computes the \emph{algebraic intersection number}, whose operation is denoted by $\alpha\cdot\beta$ for $\alpha\in H_k(M)$ and $\beta\in H_{n-k}(M)$.
In the special case of $H_1(\Sigma)\otimes H_1(\Sigma)\to \Z$, then $\alpha\cdot\beta$ may be computed by realizing $\alpha$ and $\beta$ as transversely intersecting smooth multicurves then summing over all $p\in\alpha\cap\beta$ the sign of $\omega(d_p\alpha,d_p\beta)$, where $\omega$ is an area form for $\Sigma$ and $d_p\alpha$ and $d_p\beta$ denote tangent vectors at $p$ of some oriented parameterizations of $\alpha$ and $\beta$.

Let $\pi:\Sigma\times I\to \Sigma$ be the projection onto the {$\Sigma$} component.
Then the composition with the inclusion of $L$ into $\Sigma\times I$ induces a homomorphism $H_1(L)\to H_1(\Sigma)$.
Hence, we get a map $H_1(L)\otimes H_1(\Sigma)\to \Z$ defined by $\alpha\otimes\beta\mapsto \pi_*(\alpha)\cdot \beta$.
Using the hom-tensor adjunction, this defines a homomorphism $H_1(\Sigma)\to \Hom(H_1(L),\Z)$.
The universal coefficient theorem gives that $\Hom(H_1(L),\Z)\cong H^1(L)$, and Poincar\'e duality for $L$ gives that $H^1(L)\cong H_0(L)$.
Hence, we obtain a homomorphism $H_1(\Sigma)\to H_0(L)$.
Unpacking this, if $[L_1],\dots,[L_n]\in H_1(L)$ are the orientation generators for each component of $L$ then we can define a homomorphism $H_1(\Sigma)\to \Z^n$ by $\alpha\mapsto (\pi_*([L_1])\cdot\alpha,\cdots,\pi_*([L_n])\cdot\alpha)$, and then we can identify $\Z^n$ with $H_0(L)$.

Composing this with {the given homomorphism $\lambda:H_0(L)\to G$}, we define a homomorphism
\[
  h_{L,\lambda}:H_1(\Sigma) \to G.
\]
Unpacking things again, this has the formula
\[
  h_{L,\lambda}(\alpha) = (\pi_*([L_1])\cdot\alpha)\lambda_1 + \cdots + (\pi_*([L_n])\cdot\alpha)\lambda_n
\]
with $\lambda_1,\dots,\lambda_n$ as before.

\begin{definition}
  For $G$ an abelian group, let $R_G=\Z[A^{\pm 1}][X_{\pm g}\mid g\in G]/(X_0-1)$.
  Furthermore, for $n\in\N$, let $R_n=R_{\Z^{n}}$.
  We identify $R_1$ with $\Z[A^{\pm 1},X_0,X_1,X_2,\dots]/(X_0-1)$.
\end{definition}

Note that $\mathcal{H}(\Sigma)=R_{H_1(\Sigma)}$.
When $\phi:G\to G'$ is a homomorphism of abelian groups, then there is an induced ring homomorphism $\phi_*:R_G\to R_{G'}$ defined by $X_{\pm g}\mapsto X_{\pm \phi(g)}$.
Thus, the homomorphism $H_1(\Sigma)\to G$ associated to $\lambda$ induces a ring homomorphism $\mathcal{H}(\Sigma)\to R_G$, which, composed with the invariant from \Cref{thm:sfc-poly-1} defines a $\Z[A^{\pm 1}]$-linear map
\[
  \overline{h}_{L,\lambda}:\Sk(\Sigma)\to R_G.
\]
The state sum formulation of this is given by
\[
  \overline{h}_{L,\lambda}([L']) = \sum_S A^{a(S)-b(S)}(-A^2-A^{-2})^{b_0(S)} \prod_{C\in S} X_{\pm h_{L,\lambda}([C])},
\]
where $S$ ranges over all states for a diagram of the surface link $L'\subset \Sigma\times I$ and $h_{L,\lambda}([C])$ is from computing $h_{L,\lambda}$ on the homology class of an arbitrary orientation of the simple closed curve $C$.
Recall that $X_0=1$.

\begin{definition}
  \label{def:homological-arrow-poly}
  Let $L\subset\Sigma\times I$ be a framed surface link with $\Sigma$ a compact oriented surface, and let $\lambda:H_0({L})\to G$ be a homomorphism with $G$ an abelian group.
  Then the \emph{homological arrow polynomial} is $\mathcal{A}(L,\lambda)=\overline{h}_{L,\lambda}([L])$, which takes values in $R_G$.

  If $L$ has its components labeled by elements of an abelian group $G$, then we write $\mathcal{A}(L)\in R_G$ for $\mathcal{A}(L,\lambda)$ with $\lambda:H_0({L})\to G$ the homomorphism defined by the component labels.

  If $L$ has its components labeled from the set $\{1,\dots,n\}$ (allowing duplicate labels) then we write $\mathcal{A}(L)\in R_n$ for $\mathcal{A}(L,\lambda)$ with $\lambda:H_0({L})\to \Z^n$ defined by component labels, where $i\in\{1,\dots,n\}$ corresponds to the $i$th generator of $\Z^n$.

  If $L$ is an unframed labeled oriented surface link, then, giving $L$ an arbitrary framing and letting $D$ be its virtual link diagram, $(-A^3)^{-\writhe(D)}\mathcal{A}(L,\lambda)$ is the \emph{normalized homological arrow polynomial}.
\end{definition}

\begin{remark}
  In consideration of the state sum formulation, if the link $L$ has its components labeled by elements of the trivial abelian group $\{e\}$, then we have $R_{\{e\}}\cong \Z[A^{\pm 1}]$ and $\mathcal{A}(L)=(-A^2-A^{-2})\langle L\rangle$, where $\langle L\rangle$ is the Kauffman bracket of $L$.
  Note that the $R_n$ rings are a convenience for the case of labeling components by generators of a free abelian group, where for example the ring $R_1=R_{\Z}$ (not to be confused with $R_{\{e\}}$) is used to yield the arrow polynomial in \Cref{thm:arrow-poly-from-homol}.
\end{remark}

\begin{theorem}
  The homological arrow polynomial is an invariant of $G$-labeled framed virtual links.
\end{theorem}
\begin{proof}
  Consider a $G$-labeled framed oriented surface link $L$ in $\Sigma\times I$, and suppose $\Sigma\subseteq\Sigma'$ for $\Sigma'$ a compact oriented surface.
  The algebraic intersection numbers in $\Sigma$ are the same as those in $\Sigma'$, hence the homological arrow polynomials for $L$ in $\Sigma\times I$ and in $\Sigma'\times I$ are the same.
  Hence, the polynomial is invariant under destabilization.

  The polynomial is a framed virtual link invariant since it is also invariant under the induced action of orientation-preserving diffeomorphisms $f:\Sigma\to \Sigma$.
  This is because $\mathcal{A}(f_*L)=\overline{h}_{f_*L}([f_*L])$ and, in the state sum, the algebraic intersection numbers satisfy $f_*C\cdot f_*C=C\cdot C$.
\end{proof}

\begin{remark}
  As it has some similarity, we briefly give a sketch for how the virtual link invariant $\Xi(L)$ in \cite{Manturov2003}, with a modification, determines the homological arrow polynomial.
  Using the notation from the paper, for a state $s$ in the state sum there is a variable corresponding to $p(s)=\gamma\sqcup \gamma(s)$, where $\gamma$ is the shadow of the knot in $\Sigma$ and $\gamma(s)$ consists of the state curves.
  In principle, combining these curves into a single collection of curves loses information, and by instead considering triples $(\Sigma,\gamma,\gamma(s))$ with $\gamma$ oriented and $G$-labeled, then, with an appropriate modification to the notation of the five elementary equivalences in the paper, one obtains a similar invariant to $\Xi(L)$ that contains enough information to determine both the original $\Xi(L)$ and the homological arrow polynomial, since elementary equivalences preserve intersection numbers.
\end{remark}

\begin{definition}
  A link $L$ is \emph{$1$-labeled} if all its components are labeled by $1$ from the set $\{1\}$.
  In this case, the homological arrow polynomial $\mathcal{A}(L)$ takes values in $R_1=\Z[A^{\pm 1},X_0,X_1,\dots]/(X_0-1)$.
\end{definition}

\begin{theorem}
  Suppose $L$ is a framed virtual link, $G$ and $G'$ are abelian groups, and $\phi:G\to G'$ is a group homomorphism.
  If $\lambda$ is a $G$-labeling, then $\mathcal{A}(L,\phi\circ\lambda)=\phi_*\mathcal{A}(L,\lambda)$.
  In particular, $\mathcal{A}(L,\lambda)=\lambda_*\mathcal{A}(L,\id)$, so the identity labeling determines all other homological arrow polynomials.
\end{theorem}

\begin{proof}
  Recall that a group homomorphism $\phi:G\to G'$ induces a ring homomorphism $\phi_*:R_G\to R_{G'}$ defined by $X_{\pm g}\mapsto X_{\pm \phi(g)}$.
  We see that
  \begin{align*}
    \mathcal{A}(L,\phi\circ\lambda)
    &= \sum_S A^{a(S)-b(S)}(-A^2-A^{-2})^{b_0(S)} \prod_{C\in S} X_{\pm h_{L,\phi\circ\lambda}([C])} \\
    &= \sum_S A^{a(S)-b(S)}(-A^2-A^{-2})^{b_0(S)} \prod_{C\in S} X_{\pm \phi(h_{L,\lambda}([C]))} \\
    &= \sum_S A^{a(S)-b(S)}(-A^2-A^{-2})^{b_0(S)} \prod_{C\in S} \phi_*(X_{\pm h_{L,\lambda}([C])}) \\
    &= \phi_*\mathcal{A}(L,\lambda).
  \end{align*}
  From this, it follows that $\mathcal{A}(L,\lambda)=\mathcal{A}(L,\lambda\circ\id)=\lambda_*\mathcal{A}(L,\id)$.
\end{proof}

\begin{corollary}
  If $L$ is a framed oriented virtual link with components labeled from $\{1,\dots,n\}$, then, by taking $\mathcal{A}(L)$ and substituting $X_{\pm I}\mapsto X_{\abs{\sum_i{I_i}}}$ for each $I\in\Z^n$, the resulting polynomial in $R_1$ is the homological arrow polynomial for $L$ when it is $1$-labeled.
\end{corollary}

\begin{proof}
  Let $\lambda:H_0(L)\to\Z^n$ be the homomorphism associated to the labeling.
  Defining $\phi:\Z^n\to\Z$ by $\phi(I)=\sum_{i=1}^n I_i$, we see that $\phi_*:R_n\to R_1$ is defined by $X_{\pm I}\mapsto X_{\abs{\sum_i{I_i}}}$, and hence $\phi_*\mathcal{A}(L,\lambda)$ yields the stated substitution.
  By the previous theorem, this is $\mathcal{A}(L,\phi\circ\lambda)$, and $\phi\circ\lambda : H_0(L)\to \Z$ is the homomorphism associated to the $1$-labeling of $L$.
\end{proof}

\begin{lemma}
  \label{thm:h-evenness}
  For $L$ a $1$-labeled framed oriented surface link with a given framed diagram and for $C$ a component of a state of $L$, then $h_L(C)$ is even.
\end{lemma}
\begin{proof}
  In the diagram for $L$, since $L$ is oriented there is a well-defined left-handed pushoff, which can be thought of as being obtained by walking along the link in the link's orientation with one's left hand extended and tracing the path taken.
  We use the pushed-off version of $L$ when computing algebraic intersection numbers with state curves

  When smoothing a crossing, let us keep track of the local contribution to algebraic intersection number by drawing \emph{whiskers}, which are small rays from $L$ indicating the orientation of $L$ where $L$ intersects the state curves:
  \begin{equation*}
    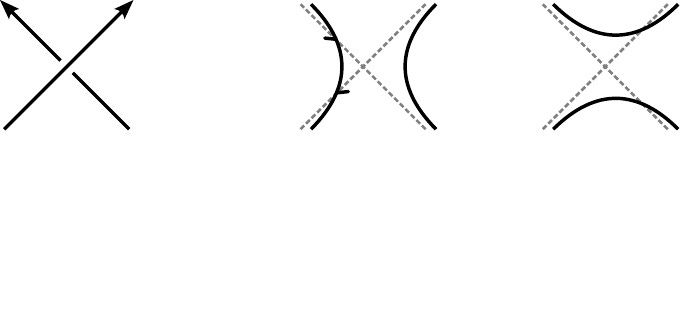
  \end{equation*}
  After giving the state curves an arbitrary orientation, the algebraic intersection number between $L$ and a state curve $C$ is the difference between the numbers of ``left'' whiskers and of ``right'' whiskers, from the point of view of someone taking a walk along $C$ in the direction of its orientation.
  Of course, the whiskers in the leftmost term of each expansion cancel immediately.

  We can add additional structure (similar to the ``decorated magnetic virtual graphs'' of \cite{Miyazawa2008} or the ``disorientations'' of \cite{Clark2009}) in the form of orientations of the state loops, where at whiskers the orientation reverses:
  \begin{equation*}
    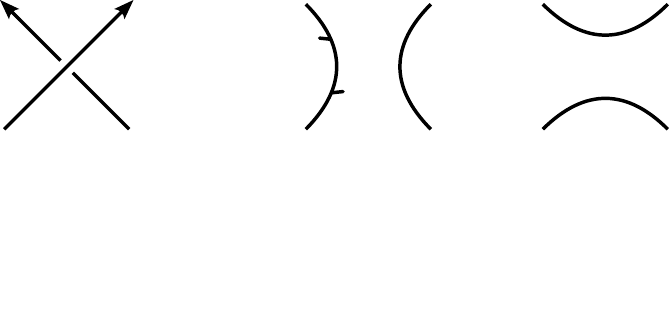
  \end{equation*}
  Using this, we see that each component of a state has an even number of whiskers due to the alternating orientations.
\end{proof}

\begin{remark}
  For computing the homological arrow polynomial of $G$-labeled framed virtual links, the idea of whiskers from the proof of \Cref{thm:h-evenness} can be modified by labeling each whisker by the label of the component of $L$ that intersects there.
  This idea is represented in \Cref{fig:whisker-expansion}.
  We can use the whisker reduction rules from that figure to arrange for each state loop to have exactly one whisker, and then, in the state sum, a state loop with a whisker labeled by $a\in G$ contributes the variable $X_{\pm a}$.

  For the $1$-labeled case, an algorithm for computing the arrow polynomial in this way (in consideration of \Cref{thm:arrow-poly-from-homol}) is given in \Cref{sec:computing-arrow}, and an algorithm for the homological arrow polynomial when components are labeled from $\{1,\dots,n\}$ is given in \Cref{sec:computing-harrow}.
\end{remark}

\begin{theorem}
  \label{thm:arrow-poly-from-homol}
  Given a $1$-labeled framed surface link $L$, then the Dye--Kauffman arrow polynomial and the homological arrow polynomial satisfy
  \begin{equation*}
    (-A^2-A^{-2})\langle L\rangle_{\mathrm{A}} = \mathcal{A}(L)|_{X_i=K_{i/2}\text{ for all $i$}},
  \end{equation*}
  where the division by two is justified by \Cref{thm:h-evenness}.
\end{theorem}
\begin{proof}
  The Dye--Kauffman arrow polynomial is from expanding crossings according to \Cref{fig:arrow-bracket-1} then reducing cusps according to the following two rules:
  \begin{equation*}
    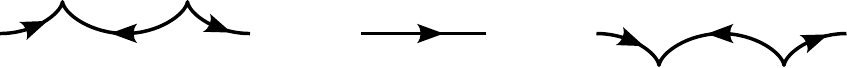
  \end{equation*}
  As a state sum,
  \begin{equation*}
    \langle L \rangle_{\mathrm{A}} = \sum_{S} A^{a(S)-b(S)} (-A^2-A^{-2})^{b_0(S)-1} \prod_{C\subseteq S} K_{\gamma(C)/2},
  \end{equation*}
  where $\gamma(C)$ is the number of cusps in component $C$ after cusp reduction.
  This is an even number in consideration of the alternating arc orientations.

  Consider replacing cusps in the arrow polynomial with whiskers in the following manner:
  \begin{equation*}
    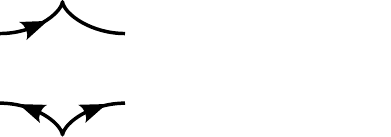
  \end{equation*}
  Notice that the relations for the arrow polynomial in \Cref{fig:arrow-bracket-1} after this replacement become the relations in the proof of \Cref{thm:h-evenness}.
  Furthermore, notice that the cusp cancelation rules are precisely the whisker cancelation rules.
  Hence, the number of cusps after cusp cancellation for a given state with respect to the arrow polynomial is equal to the net number of whiskers for that state with respect to the homological arrow polynomial.
  From this we complete the proof.
\end{proof}

\begin{example}
  The right-handed virtual Hopf link with components labeled by $1$ and $2$ has homological arrow polynomial $(-A^2-A^{-2})(AX_{1,-1}+A^{-1}X_{1,1})$:
  \[
    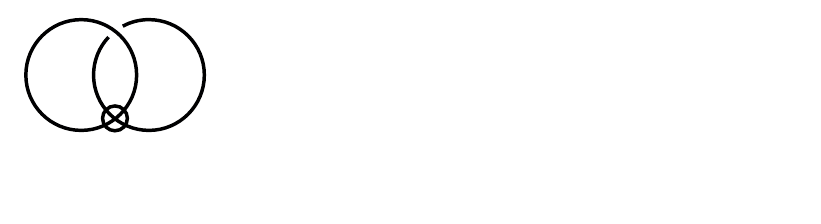
  \]
  Hence, if each component is labeled by $1$ the corresponding polynomial is $(-A^2-A^{-2})(A+A^{-1}X_2)$ and thus the arrow polynomial is $A+A^{-1}K_1$.
\end{example}

\section{Properties}
\label{sec:properties}

\subsection{Basic properties}

\begin{proposition}
  Suppose $G$ is an abelian group.
  If $L$ and $L'$ are $G$-labeled framed oriented virtual links, then
  \begin{equation*}
    \mathcal{A}(L\sqcup L')=\mathcal{A}(L)\mathcal{A}(L'). 
  \end{equation*}
\end{proposition}
\begin{proof}
  One can factor the state sum formulation of the homological arrow polynomial.
\end{proof}

\begin{proposition}
  \label{thm:orientation-reversal}
  If $L$ is a labeled framed oriented virtual link with labels from $\{1,\dots,n\}$ and $r_iL$ is $L$ but with all components having label $i$ given reversed orientation, then $\mathcal{A}(r_iL)=\left.\mathcal{A}(L)\right|_{X_{\pm I}\mapsto X_{\pm r_iI}}$, where $r_i(m_1,\dots,m_i,\dots,m_n)=(m_1,\dots,-m_i,\dots,m_n)$.
\end{proposition}
\begin{proof}
  By reversing the orientation of all components of a particular label, in the state sum the whiskers of that label are all reflected over their respective component.
  Hence, the roles of ``left'' and ``right'' whiskers are reversed in the difference, so that label's count is negated.
\end{proof}

This proposition can be generalized to the case of labelings where the label group is a product of abelian groups, one per component of the link, with each component labeled by an element of its own group.
Without this additional structure, there is at least the following:

\begin{proposition}
  Suppose $G$ is an abelian group.
  If $L$ is a $G$-labeled framed oriented virtual link with labeling $\lambda$ and $rL$ is $L$ but with opposite orientations, then
  \[
    \mathcal{A}(rL,\lambda) = \mathcal{A}(L,-\lambda) = \mathcal{A}(L,\lambda),
  \]
  where $-\lambda$ is the composition of $\lambda$ with negation.
\end{proposition}

The following corollary is \cite[Theorem 1.1]{Dye2008} (note that the theorem statement claims it is independent of the orientation of link diagram, but the proof assumes only total orientation reversal).
\begin{corollary}
  If $L$ is a $1$-labeled framed virtual link, then $\mathcal{A}(rL)=\mathcal{A}(L)$, where $rL$ is $L$ but with opposite orientations.
\end{corollary}

\begin{proposition}
  Suppose $G$ is an abelian group.
  If $L$ is a $G$-labeled framed oriented virtual link and $mL$ is its mirror image, then
  \begin{equation*}
    \mathcal{A}(mL)=\left.\mathcal{A}(L)\right|_{A\mapsto A^{-1}}. 
  \end{equation*}
\end{proposition}

\subsection{Genus bounds}
\label{sec:genus-bounds}

\begin{figure}[tb]
  \centering
  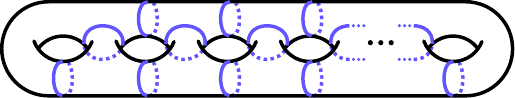
  \caption{For a surface $\Sigma$ of genus $g(\Sigma)\geq 2$, a maximal collection of nonseparating pairwise non-isotopic disjoint loops gives a pants decomposition, which always has $2(g(\Sigma)-1)$ pants.
    Each of the pants is incident to three curves, so there are $3(g(\Sigma)-1)$ curves in total. }
  \label{fig:high-genus-pants}
\end{figure}

For a $G$-labeled framed oriented surface link $L\subset\Sigma\times I$, the homological arrow polynomial $\mathcal{A}(L)$ consists of monomial terms of the form $p X_{\pm g_1}\cdots X_{\pm g_k}$ for nonzero $p\in\Z[A^{\pm 1}]$ and nonzero $g_i\in G$ for all $1\leq i\leq k$, allowing duplicate factors.
Each $X_{\pm g_i}$ comes from a loop in a state of the state sum, and since $g_i\neq 0$, it corresponds to a non-separating loop in $\Sigma$.

If $\pm g_i\neq \pm g_j$, then the loops that produced $g_i$ and $g_j$ must not be homologous and thus not be isotopic.
If the genus $g(\Sigma)$ is at most $1$, then the maximum number of nonseparating pairwise non-isotopic disjoint loops is $g(\Sigma)$.
For $g(\Sigma)\geq 2$, then a pants decomposition gives an upper bound of $3g(\Sigma)-3$ for the number of nonseparating pairwise non-isotopic disjoint loops (see \Cref{fig:high-genus-pants}).
Thus, we may generalize the \cite[Theorem 4.5]{Dye2008} genus bound to use the homological arrow polynomial in the following theorem.
\begin{theorem}
  \label{thm:genus-bound}
  Let $L\subset\Sigma\times I$ be a labeled framed oriented surface link with $\Sigma$ closed, and let $p X_{\pm g_1}\cdots X_{\pm g_k}$ be a term of $\mathcal{A}(L)$, where $p\in\Z[A^{\pm 1}]$ is nonzero and $g_i\in G$ is nonzero for all $1\leq i\leq k$.
  If $m$ is the number of pairwise distinct indices (that is, $m := \abs{ \{ \pm g_i : 1\leq i \leq k\}}$), then either
  \begin{enumerate}
  \item $g(\Sigma)\geq m$ if $g(\Sigma)\leq 1$, or
  \item $g(\Sigma) \geq \frac{1}{3}m+1$ if $g(\Sigma)\geq 2$.
  \end{enumerate}
\end{theorem}
\begin{corollary}
  \label{thm:genus-bound-virt}
  Let $m$ be the number of pairwise distinct indices in the $X_{\pm g}$ variables of a term of the homological arrow polynomial of a labeled framed oriented virtual link $L$.
  If $m$ is $0$ or $1$, then $m$ is a lower bound for the virtual genus of $L$.
  Otherwise, if $m\geq 2$, then $\lceil\frac{1}{3}m+1\rceil$ gives a lower bound for the virtual genus of $L$.
\end{corollary}

\begin{remark}
  Given any collection of nonseparating pairwise non-isotopic disjoint loops in a closed oriented surface, there exists an immersed loop whose algebraic intersection number with each of these loops is nonzero and even.
  Hence, we do not expect the $3g-3$ bound to be improved except perhaps by considering the whole set of $m$ numbers for every term in the homological arrow polynomial.
\end{remark}

\subsubsection{Kishino's knot} % PD[Xp[7,4,6,5], Xm[6,8,5,7], Xm[1,3,8,2], Xp[4,1,3,2]]

\begin{figure}[tb]
  \centering
  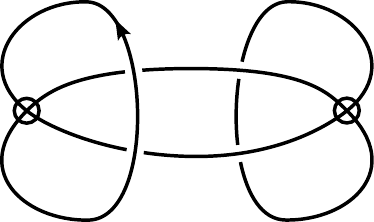
  \caption{The virtual knot 4.55, known as \emph{Kishino's knot}.}
  \label{fig:kishino}
\end{figure}

The virtual knot 4.55 (see \Cref{fig:kishino}) is known to have a trivial Jones polynomial and a non-trivial arrow polynomial:
\begin{align*}
  \langle K_{4.55}\rangle_{\mathrm{NA}} = A^4 + 1 + A^{-4} - (A^4 + 2 + A^{-4}) K_1^2 + 2 K_2.
\end{align*}
By \Cref{thm:genus-bound-virt}, we see the virtual genus of Kishino's knot is at least $1$ (hence is definitely not a classical knot).
The diagram for Kishino's knot can be drawn on a surface of genus $2$.
In \Cref{sec:checkerboard,sec:computational-investigations}, we discuss the $n$-cabled arrow polynomial, which is the arrow polynomial of an $n$-cabling of a given virtual link.
The $n$-cable of a virtual link has the same virtual genus.
We calculate the $2$-cabled arrow polynomial of Kishino's knot to be
\begin{equation*}
  \begin{gathered}
    -A^{-18} + A^{-14} + 7A^{-10} + 15A^{-6} + 19A^{-2} + 19 A^{2} + 15 A^{6} + 7 A^{10} + A^{14} - A^{18} \\
    {} + K_1^2 (-2A^{-14} - 8A^{-10} - 14A^{-6} - 16A^{-2} - 16 A^{2} - 14 A^{6} - 8 A^{10} - 2 A^{14}) \\
    {} + K_2^2 (A^{-18} - A^{-14} - 6A^{-10} - 12A^{-6} - 14A^{-2} - 14 A^{2} - 12 A^{6} - 6 A^{10} - A^{14} + A^{18}) \\
    {} + K_3^2 (-2A^{-10} - 4A^{-6} - 2A^{-2} - 2 A^{2} - 4 A^{6} - 2 A^{10})
     + K_4^2 (-2A^{-2} - 2 A^{2}) \\
    {} + K_1^2K_2(4A^{-14} + 22A^{-10} + 50A^{-6} + 68A^{-2} + 68 A^{2} + 50 A^{6} + 22 A^{10} + 4 A^{14}) \\
    {} + K_1 K_2 K_3 (2A^{-14} + 8A^{-10} + 20A^{-6} + 34A^{-2} + 34 A^{2} + 20 A^{6} + 8 A^{10} + 2 A^{14}) \\
    {} + K_2^2K_4(2A^{-10} + 6A^{-6} + 8A^{-2} + 8 A^{2} + 6 A^{6} + 2 A^{10}) \\
    {} + K_1^4(-A^{-14} - 7A^{-10} - 21A^{-6} - 35A^{-2} - 35 A^{2} - 21 A^{6} - 7 A^{10} - A^{14}) \\
    {} + K_1^2K_2^2(-2A^{-14} - 12A^{-10} - 32A^{-6} - 50A^{-2} - 50 A^{2} - 32 A^{6} - 12 A^{10} - 2 A^{14}) \\
    {} + K_2^4(-A^{-14} - 4A^{-10} - 8A^{-6} - 11A^{-2} - 11 A^{2} - 8 A^{6} - 4 A^{10} - A^{14})
    .
  \end{gathered}
\end{equation*}
Since $K_1K_2K_3$ appears, we get the virtual genus lower bound of $\frac{1}{3}\cdot 3+1=2$ for Kishino's knot.
Therefore the virtual genus is exactly $2$.
This fact was previously shown in \cite{Dye2004}, where they analyzed state curves of the Kauffman bracket expansion of Kishino's knot in a representative genus-$2$ surface.

\subsection{Crossing number bounds}
\label{sec:crossing-number}

The \emph{crossing number} $c(L)$ of a surface link or virtual link $L$ is the minimal number of crossings over all diagrams of the link.

The following proposition has not proved to be useful for determining the crossing number, since, except for the unknot, no virtual knot in Green's census gives an equality, and only $49$ out of $2565$ are one off from an equality.
\begin{proposition}
  Let $L$ be a surface link labeled from $\{1,\dots,n\}$, and let $p X_{\pm I_1}\cdots X_{\pm I_k}$ be a term of $\mathcal{A}(L)$, where $p\in\Z[A^{\pm 1}]$ and $I_i\in \Z^n$ nonzero for all $1\leq i\leq k$.
  With $s=\sum_{i=1}^k\sum_{j=1}^n\abs{I_{ij}}$, then $c(L)\geq \frac{1}{2}s$.
\end{proposition}
\begin{proof}
  Each crossing in the state sum introduces two whiskers.
\end{proof}

\begin{corollary}
  If $p K_{i_1}\cdots K_{i_k}$ is a term of the arrow polynomial $\langle L\rangle_{\mathrm{A}}$ of a virtual link $L$, where $p\in\Z[A^{\pm 1}]$ and $i_j\geq 1$ for all $1\leq j\leq k$, then $c(L)\geq \frac{1}{2}(i_1+\dots+i_k)$.
\end{corollary}

\subsection{Nullhomologous virtual links}

\begin{definition}
  For $R$ a commutative ring, an oriented surface link $L\subset\Sigma\times I$ is \emph{$R$-nullhomologous} if $[L]=0$ in $H_1(\Sigma\times I;R)$.
  An oriented virtual link is \emph{$R$-nullhomologous} if it has a representative surface link that is $R$-\nullhomologous{}.
  We say a surface link or virtual link is \emph{nullhomologous} (or \emph{almost classical}) if it is $\Z$-\nullhomologous{}.
\end{definition}

An important part of the theory of virtual links is Kuperberg's characterization, \Cref{thm:what-is-virtual-link}.
For a thickened surface $\Sigma\times I$, a \emph{vertical annulus} $A\subset\Sigma\times I$ is a properly embedded annulus that is isotopic to $C\times I$ for some simple closed curve $C\subset\Sigma$, and, for a given surface link $L\subset\Sigma\times I$, such an annulus disjoint from $L$ is called \emph{essential} if it does not bound a ball in $\Sigma\times I-L$.
For a surface link $L\subset\Sigma\times I$ and a vertical annulus in the complement of $A$, there is a virtually equivalent surface link called the \emph{destabilization of $L\subset\Sigma\times I$ along $A$}, described as follows.
After an ambient isotopy we may assume $A=C\times I$, and, with $\nu(C)$ being a tubular neighborhood of $C$ in $\Sigma$ such that $\nu(C)\times I$ is disjoint from $L$, then $L\subset (\Sigma-\nu(C))\times I$ is a destabilization of $L\subset\Sigma\times I$.
Capping off each component of $\partial(\Sigma-\nu(C))$ with a disk yields a closed surface $\Sigma'$, and we call $L\subset\Sigma'\times I$ the \emph{closed destabilization} of $L\subset\Sigma\times I$ along $A$.

A \emph{spanning} surface link $L\subset\Sigma\times I$ is one for which $L$ meets each component of $\Sigma\times I$.
\begin{theorem}[{\hspace{1sp}\cite{Kuperberg2003}}]
  \label{thm:what-is-virtual-link}
  Given a virtual link $L$, there is a representative spanning surface link $L\subset\Sigma\times I$ with $\Sigma$ closed such that every vertical annulus in $\Sigma\times I-L$ bounds a ball disjoint from $L$.
  This surface link is unique up to weak equivalence (the equivalence relation generated by isotopy and transformations induced by orientation-preserving self-diffeomorphisms of $\Sigma$).
\end{theorem}
\begin{proof}[Proof sketch]
  Given a disjoint union $\mathcal{A}$ of vertical annuli in $\Sigma\times I-L$, let the \emph{pruned destabilization along $\mathcal{A}$ of $L\subset\Sigma\times I$} be the result of taking the closed destabilization of $L\subset\Sigma\times I$ along all the annuli in $\mathcal{A}$ and then throwing away every component of the resulting thickened surface that does not meet $L$, yielding a spanning surface link that is virtually equivalent to $L\subset\Sigma\times L$.
  An \emph{irreducible descendant} of $L\subset\Sigma\times I$ is a pruned destabilization where every vertical annulus in the complement of $L$ bounds a ball disjoint from $L$.
  If there were a spanning surface link with multiple non-diffeomorphic irreducible descendants, we could let $L\subset\Sigma\times I$ be one with $g(\Sigma)-b_0(\Sigma)+b_0(L)\in\N$ minimal.
  Thus, the closed destabilization along any essential vertical annulus for this surface link has a unique irreducible descendant.
  
  The way the proof proceeds is to suppose there are nonempty disjoint unions $\mathcal{A}_1$ and $\mathcal{A}_2$ of essential vertical annuli in $\Sigma\times I$ whose pruned destabilizations have non-diffeomorphic irreducible descendants that, when put in general position, intersect in the fewest number of curves.
  The intersection is analyzed to show that either $g(\Sigma)-b_0(\Sigma)+b_0(L)$ was not minimal or that either $\mathcal{A}_1$ or $\mathcal{A}_2$ can be modified to have fewer components in the intersection without changing the diffeomorphism class of its irreducible descendent.

  Therefore, there is a unique irreducible descendant for every surface link and hence for each virtual equivalence class.
\end{proof}

We aim to show that the property of a surface link being $R$-\nullhomologous{} is preserved under closed destabilizations, thereby proving that it is reasonably well-behaved as a property of virtual links.
For this purpose, we give a diagrammatic characterization of a surface link being $R$-\nullhomologous{}, which is the existence of what we call an $R$ Dehn numbering.
This is the surface link version of the numberings used by Alexander in \cite{Alexander1928} with $R=\Z$, which there represented the abelianization of the Dehn presentation of the knot group.
Note that for surface link diagrams that are \emph{cellular embeddings}, where each complementary region is a disk, then there is a one-to-one correspondence between Dehn numberings and what are known as Alexander numberings (see, for example, \cite{Boden2017}).
The difference is that in a Dehn numbering the regions are assigned elements of $R$ whereas in an Alexander numbering the diagram's edges are.

\begin{definition}
  \label{def:dehn-numbering}
  Suppose $R$ is commutative (unital) ring, $L\subset \Sigma\times I$ is an oriented surface link with $\Sigma$ closed, and $D\subset \Sigma$ is a diagram for $L$.
  Letting $\mathcal{R}$ denote the regions in the complement of $D$, we say a function $f:\mathcal{R}\to R$ is an \emph{$R$ Dehn numbering} for $D$ if for every edge in $D$, with $A$ being the region to the right of the edge (from the point of view of someone traveling along the edge according to the orientation) and $B$ being the region to the left of the edge like so
  \begin{equation*}
    B \raisebox{-1.3em}{\marginbox{8pt 0pt}{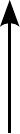}} A
  \end{equation*}
  then $f(B)=f(A)+1$, where $1\in R$ is the unit element.
  (See \Cref{fig:4n105,fig:vlink1} for examples.)
\end{definition}

\begin{figure}[tb]
  \centering
  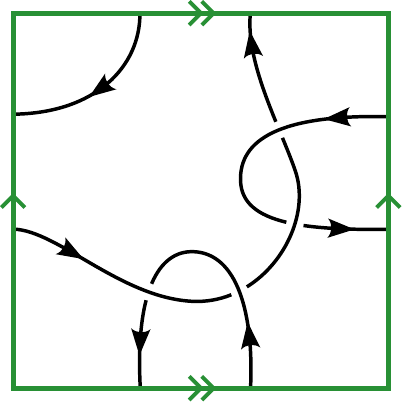
  \caption{The nullhomologous virtual knot 4.105, drawn on a torus with a $\Z$ Dehn numbering.}
  \label{fig:4n105}
\end{figure}

\begin{figure}[tb]
  \centering
  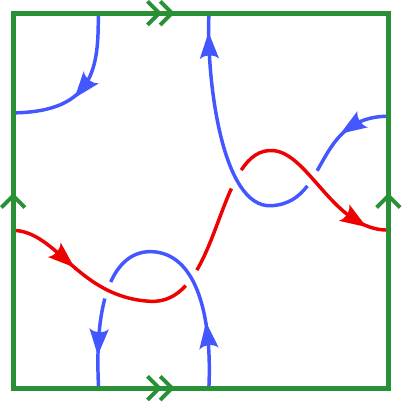
  \caption{A two-component nullhomologous virtual link drawn on a torus with a $\Z$ Dehn numbering.
  The components are colored red and blue.}
  \label{fig:vlink1}
\end{figure}

\begin{lemma}
  \label{thm:dehn-numbering}
  Suppose $R$ is a commutative ring, $L\subset \Sigma\times I$ is an oriented surface link with $\Sigma$ closed, and $D\subset \Sigma$ is a diagram for $L$.
  The link $L$ is $R$-\nullhomologous{} if and only if there exists an $R$ Dehn numbering for $D$.
\end{lemma}
\begin{proof}
  For the following we use cellular homology.
  Give $\Sigma$ a CW complex structure from a triangulation of $\Sigma$ such that $D$ is a subcomplex, and arrange for the orientations of each of the $2$-cells to coincide with the orientation of $\Sigma$ and for the orientations of each $1$-cell in $D$ to coincide with the orientations of the diagram's edges.
  Let $\alpha\in C_1(\Sigma;R)$ be the $1$-chain that's the sum of all the $1$-cells in $D$.
  Evidently, if $\pi:\Sigma\times I\to \Sigma$ is the standard projection, then $\pi_*([L])=[\alpha]$ in $H_1(\Sigma;R)$.
  Since $\pi$ is a homotopy equivalence, then this means $L$ being $R$-\nullhomologous{} is equivalent to $[\alpha]=0$.
  By definition, $[\alpha]=0$ if there is some $2$-chain $\beta\in C_2(\Sigma;R)$ whose boundary is $\alpha$.
  Since every $1$-cell is contained in the boundary of exactly two $2$-cells, such a $\beta$ is an $R$-valued function on $2$-cells characterized by two properties: (1) whenever two $2$-cells share a $1$-cell outside of $D$, then they take on the same value and (2) whenever two $2$-cells {$A$ and $B$} share a $1$-cell from $D$, then, with $A$ and $B$ arranged with respect to the orientation of $D$ as in \Cref{def:dehn-numbering}, we have $\beta(B) = \beta(A) + 1$.
  Hence, the data for such a $\beta$ is equivalent to the data of an $R$ Dehn numbering.
\end{proof}

\begin{lemma}
  \label{thm:closed-destab-once}
  Suppose $R$ is a commutative ring, $L\subset \Sigma\times I$ is an $R$-\nullhomologous{} oriented surface link with $\Sigma$ closed, and $A\subset\Sigma\times I-L$ is a vertical annulus.
  Then the closed destabilization of $L\subset\Sigma\times I$ along $A$ is $R$-\nullhomologous{}.
\end{lemma}
\begin{proof}
  After an isotopy, we may assume that there is a simple closed curve $C\subset \Sigma$ such that $A=C\times I$.
  Fixing a neighborhood of $C$, there is another isotopy such that the projection onto $\Sigma$ yields a diagram for $L$, and note that $C$ is in the complement of this diagram.
  Since $L$ is $R$-\nullhomologous{}, there is an $R$ Dehn numbering for the diagram.
  The curve $C$ is entirely contained in some region outside the diagram; let $r\in R$ be this region's value.
  If we remove $\nu(C)$ from $\Sigma$ and glue in two disks, we can obtain an $R$ Dehn numbering for the diagram on the resulting surface by giving both of the disk's regions the value $r$ and using the original Dehn numbering for all the unalterned regions.
  Hence $L$ is $R$-\nullhomologous{} in the closed destabilization along $A$.
\end{proof}

\begin{proposition}
  \label{thm:nullhom-minimal}
  Suppose $L\subset\Sigma\times I$ is a spanning oriented surface link with $\Sigma$ closed.
  If $L$ as a virtual link is $R$-\nullhomologous{} and if this surface link realizes its virtual genus, then $L$ as a surface link is $R$-\nullhomologous{}.
\end{proposition}
\begin{proof}
  Let $L\subset\Sigma\times I$ be a spanning surface link that is $R$-\nullhomologous{}.
  If it is $R$-\nullhomologous{} as a virtual link, then there is some virtually equivalent $R$-\nullhomologous{} surface link $L\subset\Sigma'\times I$.
  By \Cref{thm:what-is-virtual-link}, there is a finite collection $\mathcal{A}=A_1\sqcup A_2\sqcup\dots\sqcup A_n\subset\Sigma'\times I$ of vertical annuli such that the pruned destabilization of $L\subset\Sigma'\times I$ along $\mathcal{A}$ gives the unique minimal representative surface link $L\subset\Sigma''\times I$ under virtual equivalence.
  Hence, $L\subset\Sigma''\times I$ is $R$-\nullhomologous{} by \Cref{thm:closed-destab-once}.
  If additionally $L\subset\Sigma\times I$ realizes the virtual genus, then there is a homeomorphism carrying $(\Sigma\times I,L)$ to $(\Sigma''\times I, L)$ by the uniqueness property, so $L\subset\Sigma\times I$ is $R$-\nullhomologous{}.
\end{proof}

\begin{remark}
  It is not true that every representative surface link of an $R$-\nullhomologous{} virtual link is $R$-\nullhomologous{}.
  For example, consider a homologically nontrivial simple closed curve $C$ in a closed oriented surface $\Sigma$.
  A way to manufacture more examples is to take an $R$ Dehn numbering of a diagram with two regions with different values, cutting out a disk from each of their interiors, and gluing the resulting boundaries together.
\end{remark}

It should be noted that the only rings $R$ worth considering for $R$-\nullhomologous{} surface links are quotients of $\Z$, as the following lemma implies.

\begin{lemma}
  \label{thm:z-mod-n-suff}
  For every commutative ring $R$, there exists some $n\in\N$ such that whenever $L\subset\Sigma\times I$ is a $\Z/n\Z$-\nullhomologous{} oriented surface link then $L$ is $R$-\nullhomologous{}.
\end{lemma}
\begin{proof}
  Let $R$ be a commutative ring, let $i:\Z\to R$ be the {unique ring homomorphism}, and let $n\in\N$ be such that $\ker i=n\Z$ (note that this allows for $n=0$ when $i$ is an injection).
  Hence $i$ factors as {the composition} $\Z\surj \Z/n\Z\incl R$ {of ring homomorphisms}.

  Suppose $L\subset\Sigma\times I$ is a $\Z/n\Z$-\nullhomologous{} oriented surface link.
  Recall that $H_0(\Sigma\times I)$ is a free $\Z$-module, so for every commutative ring $R'$ the Universal Coefficient Theorem gives an isomorphism $H_1(\Sigma\times I)\otimes R'\to H_1(\Sigma\times I;R')$.
  We also have that $H_1(\Sigma\times I)\approx \Z^k$ for some $k\in\N$, where we can choose a basis such that $[L]$ is represented by $(r,0,\dots,0)$ for some $r\in\Z$.
  Thus, the induced maps
  \[
    H_1(\Sigma\times I) \to H_1(\Sigma\times I;\Z/n\Z) \to H_1(\Sigma\times I;R)
  \]
  from the factoring of $i$ correspond to maps $\Z^k\to (\Z/n\Z)^k\to R^k$ using the fact that $(\Z^k)\otimes R'\approx (R')^k$.
  Hence, $L$ being $\Z/n\Z$-\nullhomologous{} means that the image of $(r,0,\dots,0)$ in $(\Z/n\Z)^k$ is $(0,0,\dots,0)$, and thus its image in $R^k$ is $(0,0,\dots,0)$ as well.
  Therefore, $L$ is $R$-\nullhomologous{}.
\end{proof}

\begin{remark}
  \label{rmk:checkerboard-z2}
  According to \Cref{thm:dehn-numbering}, the $\Z/2\Z$-\nullhomologous{} surface links correspond to those with a $\Z/2\Z$ Dehn numbering.
  Taking $0$ and $1$ to be respectively black and white, one sees that $\Z/2\Z$-\nullhomologous{} surface links are the same as checkerboard colorable surface links.
  This well-known fact is recalled in \Cref{thm:checkerboard-z2}.
\end{remark}

\begin{remark}
  Furthermore, $\Z$-\nullhomologous{} (resp. $\Z/2\Z$-nullhomologous) links have oriented (resp. unoriented) spanning surfaces $F\subset\Sigma\times I$.
  One can generalize this notion to $\Z/n\Z$-\nullhomologous{} links, where we allow the surfaces to have singularities locally modeled on the product of a cone on $n$ points and an open interval, with induced orientations.
  (The $n=0$ case is just oriented surfaces, and the $n=2$ case is effectively unoriented surfaces, since the singularities correspond to local changes in orientation.)
  There is an algorithm like the Seifert algorithm to produce such singular surfaces for $\Z/n\Z$-\nullhomologous{} links, and we can use the usual cut-and-paste techniques on them to prove \Cref{thm:closed-destab-once}, but we do not explore this approach here.
\end{remark}

\begin{theorem}
  \label{thm:z-nullhomologous-poly}
  Suppose $G$ is an abelian group, and suppose $L\subset\Sigma\times I$ is a $G$-labeled framed oriented surface link such that for each $g\in G$, the sublink of those components labeled by $g$ is nullhomologous.
  Then,
  \[ \mathcal{A}(L)\in\Z[A^{\pm 1}]\subseteq R_G. \]
  In particular, if $L$ is \nullhomologous{} then $\langle L\rangle_{\mathrm{A}}=\langle L\rangle\in\Z[A^{\pm 1}]$.
\end{theorem}
\begin{proof}
  For each $g\in G$, let $L_g$ denote the sublink of $L$ consisting of those components labeled by $g$.
  Letting $\lambda:H_0(L)\to G$ be the labeling, for each $g\in G$ let $\lambda_g$ denote its restriction to $H_0(L_g)$.
  We have that
  \[
    h_{L,\lambda}(\alpha) = \sum_{g\in G}h_{L_g,\lambda_g}(\alpha).
  \]
  Since each $L_i$ is nullhomologous, each function $h_{L_g,\lambda_g}$ is the constant-zero function, and hence $h_{L,\lambda}$ is the constant-zero function.
  Therefore, only $X_0=1$ appears in the state sum for $\mathcal{A}(L)$.
\end{proof}

The following is a restatement of \cite[Theorem 1.5]{Dye2008} and \cite[Proposition 5.8]{Miyazawa2008} in terms of the homological arrow polynomial.
\begin{corollary}
  If $L$ is a classical link (that is, a surface link in $S^2\times I$), then
  \begin{equation*}
    \mathcal{A}(L) = (-A^2-A^{-2})\langle L\rangle,
  \end{equation*}
  where $\langle L\rangle$ is the Kauffman bracket.
\end{corollary}
\begin{proof}
  {The link is \nullhomologous{} since} $H_1(S^2\times I;\Z)=0$, and thus $\langle L\rangle_{\mathrm{A}}=\langle L\rangle$ by \Cref{thm:z-nullhomologous-poly}.
  We then apply \Cref{thm:arrow-poly-from-homol}.
\end{proof}

\begin{example}
  The virtual knot 4.105 has virtual genus $1$ and is \nullhomologous{} (see \Cref{fig:4n105}), hence we expect $\mathcal{A}(K_{4.105})\in\Z[A^{\pm1}]$.
  Indeed:
  \begin{equation*}
    \mathcal{A}(K_{4.105}) = (-A^2-A^{-2})(A^8 + 1 - A^{-4}).
  \end{equation*}
\end{example}

\begin{example}
  The virtual link $L$ in \Cref{fig:vlink1}, with the red component labeled by $1$ and the blue component labeled by $2$ has
  \begin{equation*}
    \mathcal{A}(L) = (-A^2-A^{-2})\left(-A^6+A^2-2A^{-2} + (-A^2+A^{-2}+A^{-6}-A^{-10})X_{1,-1}^2\right).
  \end{equation*}
  We can see how its arrow polynomial is in $\Z[A^{\pm 1}]$ since $X_{1,-1}\mapsto X_{0}$ when passing to the polynomial for the $1$-labeling of $L$.
  Thus, the homological arrow polynomial can sometimes detect non-classicality when the arrow polynomial cannot.
  However, if we reverse the orientation of the blue component, the corresponding arrow polynomial is
  \begin{equation*}
    \langle L\rangle_{\mathrm{A}} = -A^6+A^2-2A^{-2} + (-A^2+A^{-2}+A^{-6}-A^{-10})K_1^2,
  \end{equation*}
  so the set of arrow polynomials of all orientations of each link component is able to detect non-classicality in this case.
\end{example}

We now consider $\Z/n\Z$-\nullhomologous{} surface links more carefully.
If $L\subset \Sigma\times I$ is $\Z/n\Z$-\nullhomologous{}, then, as in the proof of \Cref{thm:z-mod-n-suff}, we have that $[L]\in nH_1(\Sigma\times I)$.
Hence, using the notation from \Cref{sec:homological-arrow-poly}, we have that for all $\alpha\in H_1(\Sigma)$ that $\pi_*([L])\cdot \alpha\in n\Z$.
This puts constraints on what can appear in the indices of the $X$ variables.

\begin{theorem}
  \label{thm:z-mod-n-nullhomologous-poly}
  Suppose $G$ is an abelian group, and suppose $L\subset\Sigma\times I$ is a $G$-labeled framed oriented surface link such there is an $n\in\N$ such that for each $g\in G$, the sublink of those components labeled by $g$ is $\Z/n\Z$-\nullhomologous{}.
  Then, with $nG=\{ng\mid g\in G\}$, we have that
  \[ \mathcal{A}(L)\in R_{nG}. \]
  In particular, if $nG=0$ then
  \[ \mathcal{A}(L)\in \Z[A^{\pm 1}].
  \]
\end{theorem}
\begin{proof}
  The proof is similar to the one for \Cref{thm:z-nullhomologous-poly}, but rather than making use of the fact that $\pi_*([L_g])\cdot \alpha=0$ we now use that $\pi_*([L_g])\cdot \alpha\in n\Z$.
\end{proof}

\begin{corollary}
  Suppose $L$ is a $1$-labeled $\Z/n\Z$-\nullhomologous{} virtual link.
  Then
  \[ \mathcal{A}(L)\in\Z[A^{\pm 1},X_n,X_{2n},X_{3n},\dots].
  \]
  Furthermore, if $n$ is odd, then $\mathcal{A}(L)\in\Z[A^{\pm 1},X_{2n},X_{4n},X_{6n},\dots]$.
\end{corollary}
\begin{proof}
  By \Cref{thm:z-mod-n-nullhomologous-poly}, the only variables that can appear are of the form $X_{kn}$ for $k\in\N$, which gives the first part.
  When $n$ is odd, then, since $kn$ must be even by \Cref{thm:h-evenness}, we need for $k$ to be even, which gives the second part.
\end{proof}

\begin{remark}
  This corollary gives nothing for $\Z/2\Z$-\nullhomologous{} virtual links since it merely predicts that $\mathcal{A}(L)\in\Z[A^{\pm 1},X_2,X_{4},X_{6},\dots]$, which is already true for all virtual links by \Cref{thm:h-evenness}.
  We handle the specific properties of $\Z/2\Z$-\nullhomologous{} virtual links in \Cref{sec:checkerboard}.
\end{remark}

\subsection{Checkerboard colorability}
\label{sec:checkerboard}

A virtual link $L$ is \emph{checkerboard colorable} (introduced in \cite{Kamada2002}) if it has a surface link representative $L\subset \Sigma\times I$ such that the diagram for $L$ in that surface is checkerboard colorable.
An example of a checkerboard colorable virtual knot was given in \Cref{fig:vknot_3_5}.
As the following proposition shows, a virtual link is checkerboard colorable if and only if it is $\Z/2\Z$-\nullhomologous{}.

\begin{proposition}[{\hspace{1sp}\cite[Proposition 1.7]{Boden2019}}]
  \label{thm:checkerboard-z2}
  For a surface link $L\subset\Sigma\times I$, the following are equivalent:
  \begin{enumerate}
  \item The link is checkerboard colorable.
  \item The link is the boundary of an unoriented spanning surface $F\subset\Sigma\times I$.
  \item The link is $\Z/2\Z$-\nullhomologous{}.
  \end{enumerate}
\end{proposition}
\begin{proof}
  That (1) implies (2) is that, given a checkerboard coloring, we can construct a checkerboard surface: take all the regions of a particular color, and at crossings connect them using half twists.
  That (2) implies (3) is that the unoriented surface $F$ can be thought of as a $2$-chain with $\Z/2\Z$ coefficients.
  That (3) implies (1) is the observation in \Cref{rmk:checkerboard-z2}, which is that a checkerboard coloring is the data of a $\Z/2\Z$ Dehn numbering.
\end{proof}

One class of checkerboard colorable virtual links are the alternating virtual links.
\begin{proposition}[{\hspace{1sp}\cite[Lemma 7]{Kamada2002}}]
  \label{thm:alt-checkerboard}
  If $L$ is an alternating virtual link, then it is checkerboard colorable.
\end{proposition}
\begin{proof}
  Consider a diagram $D$ for $L$ on a closed oriented surface $\Sigma$ such that each region of $\Sigma-D$ is a disk --- that is, the diagram is cellularly embedded.
  We then perform the $A$ smoothing of each crossing, obtaining a closed unoriented $1$-manifold $S$:
  \[
    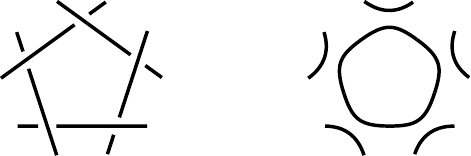
  \]
  After a small perturbation, each component of $S$ lies in $\Sigma-D$ and is thus $\Z/2\Z$-\nullhomologous{}.
  Since $S$ is $\Z/2\Z$-\homologous{} to $L$, it follows from \Cref{thm:checkerboard-z2} that $L$ is checkerboard colorable.
\end{proof}

As a quick corollary, by using the idea of the proof we get the following result for \emph{positive} virtual links, which are virtual links whose crossings are all positive.
\begin{corollary}
  If $L$ is a positive alternating virtual link, then it is $\Z$-\nullhomologous{}, and therefore $\mathcal{A}(L)\in\Z[A^{\pm 1}]$.
\end{corollary}
\begin{proof}
  Since $L$ is a positive surface link, the $A$ smoothing yields an oriented $1$-manifold $S$ that is $\Z$-\homologous{} to $L$.
    Furthermore, since $L$ is alternating, each component of $S$ bounds a disk, and thus, as it is oriented, $S$ is $\Z$-\nullhomologous{}.
    Therefore $L$ is $\Z$-\nullhomologous{}.
\end{proof}

An interesting fact about having a checkerboard coloring is that the state loops can be canonically oriented.
The coloring induces a coloring of the regions in the complement of a given state, and the state loops are the oriented boundary of the resulting black region:
\begin{equation*}
  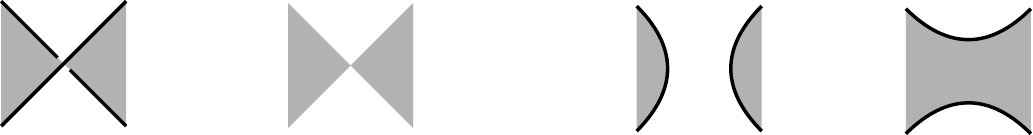
\end{equation*}
Hence, a framed oriented virtual link $L$ with a fixed checkerboard coloring $\chi$ does not have a sign ambiguity in the algebraic intersection numbers, and so we may define a polynomial $\mathcal{A}(L,\chi)\in\Z[A^{\pm 1}][X_I:I\in\Z^{n}]/(X_0-1)$.
One has $\mathcal{A}(L)=\left.\mathcal{A}(L,\chi)\right|_{X_I\mapsto X_{\pm I}}$.

The dependence on the checkerboard coloring is that, if $\chi'$ is the other checkerboard coloring, then $\mathcal{A}(L,\chi')=\mathcal{A}(L,\chi)|_{X_I\mapsto X_{-I}}$.
This is because the induced orientations on state loops for $\chi'$ are the reverse of those for $\chi$.

Suppose $L$ is $1$-labeled and consider the whiskers in the expansion for $\mathcal{A}(L)$.
The whiskers come in pairs at each crossing, where evidently one points into the black region and the other points into the white region.
Hence, if we interpret the whiskers as vectors inducing a flow across the state circles between the white region and the black region, the total flow into the black region is zero.
Thus, if $p X_{i_1}X_{i_2}\cdots X_{i_k}$ is a monomial of $\mathcal{A}(L,\chi)$ with $p\in\Z[A^{\pm 1}]$ and $i_j\in\Z$ for each $1\leq j\leq k$, allowing duplicates, then $\sum_{j=1}^ki_j=0$.
We can use this observation to reprove the following:

\begin{theorem}[{\hspace{1sp}\cite[Theorem 4.3]{Deng2020}}]
  \label{thm:deng2020}
  Let $L$ be a checkerboard colorable $1$-labeled oriented framed virtual link.
  For each monomial $p X_{i_1}X_{i_2}\cdots X_{i_k}$ of $\mathcal{A}(L)$ with $p\in\Z[A^{\pm 1}]$, $k\geq 1$, and $i_j\geq 1$ for all $1\leq j\leq k$, allowing duplicate factors, then
  \begin{enumerate}
  \item $\sum_{j=1}^ki_j\equiv 0\pmod{4}$, and
  \item $i_k\leq \sum_{j=1}^{k-1}i_j$.
  \end{enumerate}
  In particular, $k\geq 2$.
\end{theorem}

\begin{proof}
  Given a checkerboard coloring $\chi$ for $L$, such a monomial comes from (at least one) monomial of $\mathcal{A}(L,\chi)$ in the sense that there exists a monomial $p' X_{i'_1}X_{i'_2}\cdots X_{i'_k}$ of $\mathcal{A}(L,\chi)$ such that $i_j=\abs{i'_j}$ for each $1\leq j\leq k$.
  \Cref{thm:h-evenness} implies $i_j$ is even for all $j$.
  (One may also use the alternative whisker accounting from \Cref{fig:altern-whisker-expansion} to see that there are always an even number of whiskers going into and going out of the black region, and so each $i'_j$ is even.)
  This implies $i'_j\equiv i_j\pmod{4}$ for all $j$, and thus (1) follows from $\sum_{j=1}^ki'_j=0$.

  Next, by the assumption that $k\geq 1$, we have $i'_k=-\sum_{j=1}^{k-1}i'_j$, to which we apply the triangle inequality to obtain (2):
  \begin{equation*}
    i_k = \abs{i'_k} = \left\lvert\sum_{j=1}^{k-1}i'_j \right\rvert \leq \sum_{j=1}^{k-1} \abs{i'_j} = \sum_{j=1}^{k-1} i_j.
  \end{equation*}
  Lastly, $k\neq 1$ since (2) implies $i_1\leq 0$ but $i_1\geq 1$.
\end{proof}

\begin{remark}
  Condition (1) is equivalently that $\mathcal{A}(L)|_{X_k\mapsto x^k}$ is a polynomial in $\Z[A^{\pm 1},x^4]$ if $L$ is checkerboard colorable.
\end{remark}

In \cite{Deng2020}, they use this theorem to resolve checkerboard non-colorability of six of the seven exceptions from \cite{Imabeppu2016}, which are 4.55, 4.56, 4.59, 4.72, 4.76, 4.77, 4.96, where the theorem says nothing about 4.72 since $\langle K_{4.72}\rangle_{\mathrm{A}}=1$.
In all six cases, the condition that $k\geq 2$ for each term suffices.

We resolve the case of 4.72 by using cabled arrow polynomials.
Given a framed surface link $L\subset \Sigma\times I$, recall that the $n$-cabling is the framed surface link obtained from taking the framing annulus for $L$ and subdividing it into $n$ parallel copies.
The homology class of the $n$-cabling is $n[L]$, so when $n$ is even the $n$-cabling is always checkerboard colorable, and when $n$ is odd then the link is checkerboard colorable if and only if its $n$-cabling is.
We can see how checkerboard coloring is preserved in a direct way, for example with the $3$-cabling:
\begin{equation*}
  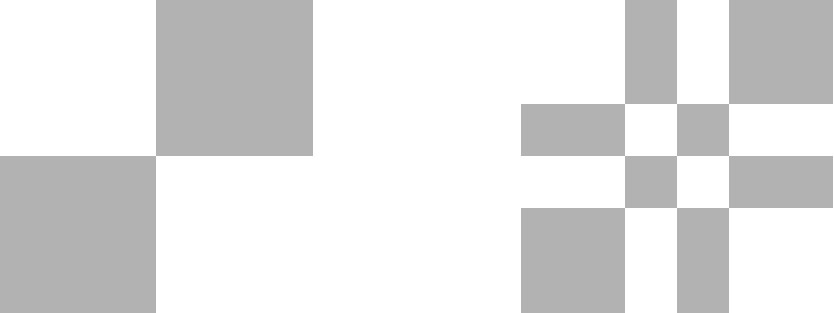
\end{equation*}
The arrow polynomial of the $3$-cabling of the $0$-framed 4.72 is
\begin{equation*}
  \begin{gathered}
    - 2A^{-32} + 4A^{-28} + 29A^{-24} + 108A^{-20} + 273A^{-16} + 575A^{-12} +
    952A^{-8} + 1298A^{-4} + 1426 \\{}+ 1298 A^4  + 952 A^8 + 575 A^{12} + 273 A^{16} +
    108 A^{20} + 29 A^{24} + 4 A^{28} - 2 A^{32} \\
    {} + K_3 (-8A^{-30} - 28A^{-26} - 69A^{-22} - 116A^{-18} - 130A^{-14} - 41
    A^{-10} + 111A^{-6} + 242A^{-2} + 242 A^2 \\{}+ 111 A^6 - 41 A^{10} -
    130 A^{14} - 116 A^{18} - 69 A^{22} - 28 A^{26} - 8 A^{30}) \\
    {} + K_{6} (8A^{-28} + 44A^{-24} + 142A^{-20} + 328A^{-16} + 618A^{-12} +
    944A^{-8} + 1210A^{-4} + 1308 + 1210 A^4 \\{}+ 944 A^8 + 618 A^{12} + 328 A^{16} +
    142 A^{20} + 44 A^{24} + 8 A^{28}) \\
    {} + K_{9}(A^{-26} + 4A^{-22} + 17
    A^{-18} + 44A^{-14} + 80A^{-10} + 108A^{-6} + 121A^{-2} + 121 A^2 +
    108 A^6 \\{}+ 80 A^{10} + 44 A^{14} + 17 A^{18} + 4 A^{22} + A^{26}) \\
    {} + K_3^2(2A^{-36} - 24A^{-28} - 115A^{-24} - 327A^{-20} - 709A^{-16}
    - 1252A^{-12} - 1857A^{-8} - 2347A^{-4} -2534\\{} - 2347 A^4 - 1857 A^8 -
    1252 A^{12} - 709 A^{16} - 327 A^{20} - 115 A^{24} - 24 A^{28} + 2 A^{36}) \\
    {} + K_3 X_{6}(-10A^{-26} - 39A^{-22} - 88A^{-18} - 148A^{-14} - 219A^{-10} -
       300A^{-6} - 360A^{-2} - 360 A^2 \\{}- 300 A^6 - 219 A^{10} - 148 A^{14} -
       88 A^{18} - 39 A^{22} - 10 A^{26})  \\
    {} + K_3^3(4A^{-30} + 18A^{-26} + 44A^{-22} + 82A^{-18} + 126A^{-14} + 165A^{-10} + 190A^{-6}
    + 199A^{-2} + 199 A^2 \\{}+ 190 A^6 + 165 A^{10} +
    126 A^{14} + 82 A^{18} + 44 A^{22} + 18 A^{26} + 4 A^{30}).
  \end{gathered}
\end{equation*}
As a polynomial over $\Z[A^{\pm 1}]$, the monomials with nonzero coefficients are $1$, $K_3$, $K_{6}$, $K_{9}$, $K_{3}^2$, $K_3K_{6}$, and $K_3^3$, so we notice that the $k\geq 2$ condition of \Cref{thm:deng2020} does not hold, and hence the $3$-cabling of 4.72 (and thus 4.72 itself) is not checkerboard colorable.

This completes Imabeppu's characterization of checkerboard colorability of all virtual knots up to four crossings using arrow polynomials.

\subsection{Alternating virtual links}
\label{sec:alternating-links}

For classical links, a \emph{reduced} alternating link diagram is one with no nugatory crossings --- that is, the identities of the opposite regions at each crossing are distinct.
Since link diagrams are checkerboard colorable, in a reduced diagram the identities of all four regions around each crossing are distinct.
Kamada in \cite{Kamada2004} calls a virtual link diagram \emph{proper} if, when it is cellularly embedded, the identities of all four regions around each crossing are distinct.
The Kauffman--Murasugi--Thistlethwaite theorem says that a connected reduced alternating (classical) link diagram $D$ satisfies $\breadth_A\langle D\rangle=4c(D)$.
Kamada shows that if $D$ is a connected proper alternating virtual link diagram, then $\breadth_A\langle D\rangle =4(c(D)-g(D))$.

\begin{figure}[tb]
  \centering
  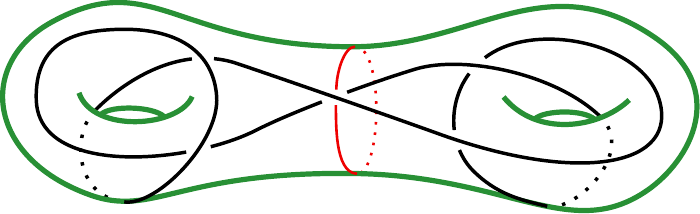
  \caption{A non-reduced surface knot diagram on a genus-$2$ surface. The red loop demonstrates that the diagram is not reduced since it is a separating loop that intersects the knot through a single crossing.}
  \label{fig:non-reduced}
\end{figure}

A generalization of this result is in \cite{Boden2019}, where they define a \emph{reduced} surface link diagram to be one that is cellularly embedded such that there are no separating simple closed curves in the surface that intersect the diagram through a single crossing (see \Cref{fig:non-reduced}).
Equivalently, a cellularly embedded surface link diagram is reduced if, whenever the two regions opposite a crossing are identical, then the surface does not become disconnected if one removes the interior of that region along with a neighborhood of the crossing.
They define the \emph{homological Kauffman bracket} by weighting each state of the Kauffman bracket using homological information in a Krushkal-polynomial-like manner.
With it they prove two of the Tait conjectures generalized to virtual links: if $D_1$ and $D_2$ are reduced connected alternating virtual link diagrams for the same virtual link $L$, then $c(D_1)=c(D_2)=c(L)$ and $\writhe(D_1)=\writhe(D_2)$.

\begin{figure}[tb]
  \centering
  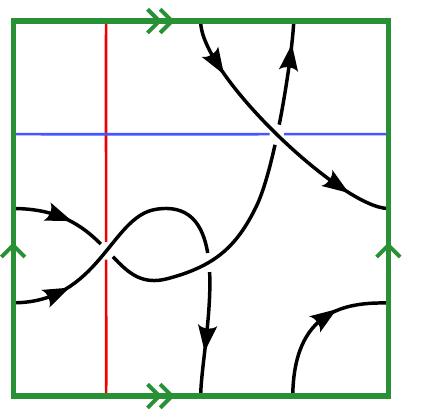
  \caption{A reduced diagram of $3.7$ on a torus. The blue loop demonstrates the virtual knot is not $h$-reduced, and the red and blue loops each demonstrate it is not proper.}
  \label{fig:3.7-loops}
\end{figure}

As a curiosity, we prove an analogue of the Kauffman--Murasugi--Thistlethwaite theorem using the arrow polynomial.
\begin{definition}
  Suppose $G$ is an abelian group.
  A $G$-labeled surface link diagram for a surface link $L$ with label function $\lambda$ is \emph{$h$-reduced} if every simple closed curve $C$ that intersects the diagram through a single crossing has $h_{L,\lambda}([C])\neq 0$.
\end{definition}
For example, if the virtual knot in \Cref{fig:3.7-loops} is $1$-labeled, the red loop is allowed in an $h$-reduced diagram but the blue loop is not.
Every proper diagram is $h$-reduced, and every $h$-reduced cellularly embedded diagram is reduced.
Given an abelian group homomorphism $f:G\to G'$, if $(L,f\circ \lambda)$ is $h$-reduced, then $(L,\lambda)$ is $h$-reduced as well.
Hence if a diagram is not $h$-reduced with respect to the identity labeling $\lambda:H_0(L)\to H_0(L)$ then there is no labeling with respect to which it is $h$-reduced.

\begin{definition}
  Let $D$ be a labeled virtual link diagram for a $G$-labeled virtual link $L$ with label function {$\lambda$}, and for a state $S$ of $D$, let $i(S)$ be the number of loops $C\subseteq S$ such that $h_{L,\lambda}([C])=0$ (the number of ``inessential'' state loops).
  The diagram $D$ is called \emph{$A$-$h$-adequate} if for every state $S'$ with $b(S')=1$, then $i(S')\leq i(S_A)$.
  Similarly, $D$ is called \emph{$B$-$h$-adequate} if for every state $S'$ with $a(S')=1$, then $i(S')\leq i(S_B)$.
  If $D$ is both $A$-$h$-adequate and $B$-$h$-adequate, then $D$ is called \emph{$h$-adequate}.
\end{definition}

\begin{proposition}
  Let $D$ be an $h$-reduced $G$-labeled alternating cellularly embedded diagram.
  Then $D$ is $h$-adequate.
\end{proposition}
\begin{proof}
  Since the diagram is alternating {and cellularly embedded}, the $A$-state $S_A$ consists of simple closed curves that bound disks in the surface, as in \Cref{thm:alt-checkerboard}.
  Checkerboard-colored diagrams have an alternative whisker expansion for the homological arrow polynomial from pushing state curves away from crossings into the black region, where near crossings whiskers only appear when the crossing becomes contained within the black region (see \Cref{fig:altern-whisker-expansion}).
  By recoloring as needed, we may assume the $S_A$ state loops bound black disks and that every black disk is disjoint from the crossings, hence there are no whiskers and each loop is nullhomotopic.
  Thus $i(S_A)=b_0(S_A)$.

  Let $S'$ be a state with $b(S')=1$, and let $c$ refer to the $B$-smoothed crossing.
  This introduces four whiskers, and, by consideration of the checkerboard coloring, either two state loops merge into one ($b_0(S')=b_0(S_A)-1$), or one state loop splits into two ($b_0(S')=b_0(S_A)+1$).
  In the case where two state loops merge into one, we have that
  \[ i(S')\leq b_0(S')=b_0(S_A)-1 < b_0(S_A)= i(S_A). \]
  In the case where a state loop splits into two, then, since every crossing is in the white region of $S_A$, the loop bounds a black disk that in $S'$ becomes an annulus:
  \[
    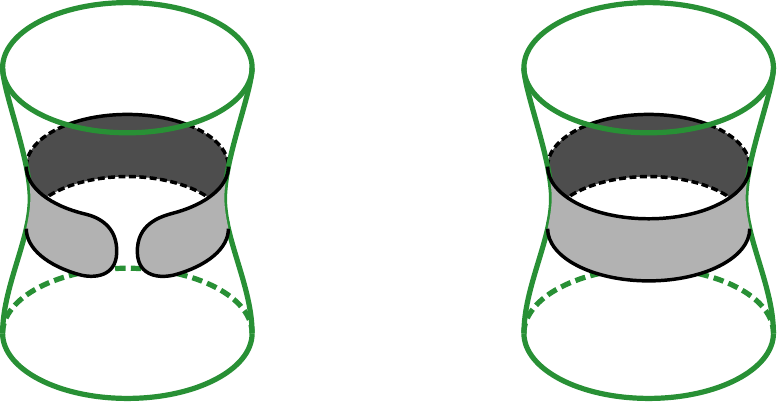
  \]
  Thus, there is a simple closed curve $C$ contained in this annulus that intersects the diagram only through $c$.
  There are only two relevant possibilities for $c$ from \Cref{fig:altern-whisker-expansion} that go from a whisker-free $A$ smoothing to a $B$ smoothing that joins the black region to itself:
  \[
    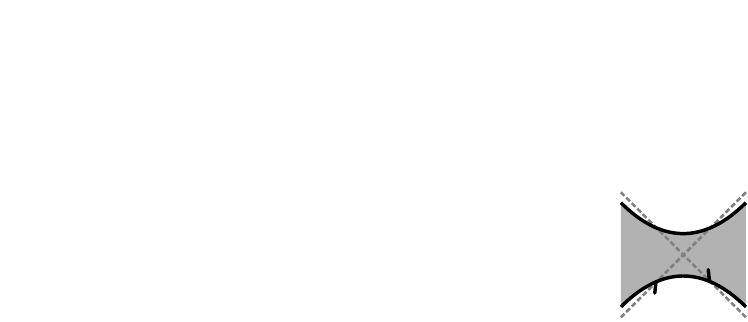
  \]
  Hence, both state loops associated to $c$ in the $B$ smoothing get an $a+b$ whisker in the first possibility and a $-a+b$ whisker in the second.
  Additionally, in the first possibility we have $h_{L,\lambda}(\pm[C])=\pm(a+b)$ and in the second $h_{L,\lambda}(\pm[C])=\pm(-a+b)$, and since the diagram is $h$-reduced, we have $a+b\neq 0$ in the first possibility and $-a+b\neq 0$ in the second.
  Thus, $i(S')=i(S_A)-2< i(S_A)$.

  Since in both cases we have shown that $i(S')\leq i(S_A)$, we have proved that the diagram is $A$-$h$-adequate.
  By a similar argument, the diagram is also $B$-$h$-adequate, and thus $h$-adequate.
\end{proof}

\begin{figure}[tb]
  \centering
  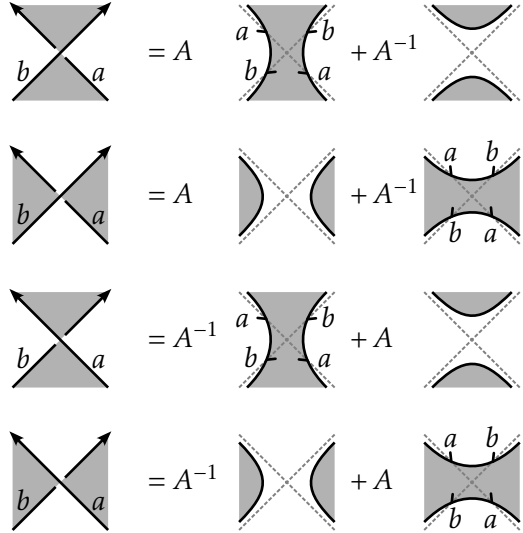
  \caption{Alternative whisker expansion for checkerboard colorable surface links by pushing state curves into the black regions where possible.
  As usual, $A$ smoothings have an $A$ coefficient and $B$ smoothings have an $A^{-1}$ coefficient.}
  \label{fig:altern-whisker-expansion}
\end{figure}

Let $\mathcal{A}'(L)$ be the result of taking $\mathcal{A}(L)$ and substituting $X_{\pm I}\mapsto X_{\pm I}(-A^2-A^{-2})^{-1}$ for each nonzero index $I$, which yields a polynomial in $R_n$.
As a state sum, this polynomial is given by
\[   \mathcal{A}'(L) = \sum_S A^{a(S)-b(S)}(-A^2-A^{-2})^{i(S)} \prod_{C\in S} X_{\pm h_{L,\lambda}([C])}, \]
with everything meaning the same as in the state sum right before \Cref{def:homological-arrow-poly}.

Considering a nonzero $p\in R_n$ as a polynomial in $A$, let $\maxdeg_Ap$ and $\mindeg_Ap$ respectively denote the maximum and minimum degrees of $A$ appearing in monomials of $p$, and let $\breadth_Ap=\maxdeg_Ap-\mindeg_Ap$.

\begin{proposition}
  Let $D$ be an $h$-adequate $G$-labeled virtual link diagram for a virtual link $L$.
  Then $\breadth_A(\mathcal{A}'(L))=2c(D) + 2i(S_A) + 2i(S_B)$.
\end{proposition}
\begin{proof}
  By the usual argument, if $D$ is $A$-$h$-adequate then nothing cancels out the $A$-state's contribution to $\maxdeg_A(\mathcal{A}'(L))$ or the $B$-state's contribution to $\mindeg_A(\mathcal{A}'(L))$.
  The maximal $A$-degree from the term for the $A$-state $S_A$ is $a(S_A)-b(S_A)+2i(S_A)=c(D)+2i(S_A)$.
  Similarly, the minimal $A$-degree of the term for the $B$-state $S_B$ is $a(S_B)-b(S_B)-2i(S_B)=-c(D)-2i(S_B)$.
  Hence, $\breadth_A(\mathcal{A}'(L))=2c(D) + 2i(S_A) + 2i(S_B)$.
\end{proof}

\begin{corollary}
  Let $D$ be an $h$-reduced connected alternating $G$-labeled virtual diagram for a virtual link $L$.  Then $\breadth_A(\mathcal{A}'(L)) = 4c(D) - 4g(D) + 4$, where $g(D)$ is the genus of the surface that $D$ cellularly embeds into.
\end{corollary}
\begin{proof}
  Let $\Sigma$ be the connected surface that $D$ cellularly embeds into.
  Since $D$ is alternating there is a checkerboard coloring, and by swapping colors if needed we may assume that in the $S_A$ state each loop bounds a black disk and in the $S_B$ state each loop bounds a white disk.
  We have that $i(S_A)$ is the number of black disks and $i(S_B)$ is the number of white disks.
  Using the crossings as $0$-cells, the edges between them as $1$-cells, and the white and black disks as $2$-cells, we obtain a cellular decomposition of $\Sigma$.
  Since $D$ is a $4$-regular graph, then by Euler characteristics we have that $2-2g(\Sigma) = c(D)-2c(D)+(i(S_A)+i(S_B))$.
  {Substituting $i(S_A)+i(S_B)=2-2g(\Sigma)+c(D)$ into the previous proposition completes the proof.}
\end{proof}

\subsection{Connect sums}

A virtual link $L$ is called a \emph{connect sum} if there is a representative surface link $L\subset\Sigma\times I$ and a vertical annulus $A\subset\Sigma\times I$ such that $A$ separates $\Sigma\times I$, $A$ meets $L$ transversely, and $\abs{L\cap A}=2$.
A connect sum has a diagram $D$ in a surface $\Sigma$ such that there exists a separating circle $C\subset\Sigma$ avoiding the crossings of $D$ that intersects the arcs of $D$ transversely in two points.
This circle represents $\Sigma$ as a connect sum $\Sigma_1\csum\Sigma_2$, and by cutting $D$ apart along $C$, putting the respective pieces on $\Sigma_1$ and $\Sigma_2$, and connecting the endpoints of the cut arcs by trivial arcs, one obtains two surface link diagrams $D_1$ and $D_2$ in $\Sigma_1$ and $\Sigma_2$, respectively.
If $L_1$ and $L_2$ are the respective corresponding virtual links, then we say $L=L_1\csum L_2$.

Unlike for connect sums of oriented knots, the connect sum of virtual knots is not a well-defined operation. The notation is only meant to signify that these three virtual links stand in this relation.
For example, the nontrivial virtual knot in \Cref{fig:non-reduced} is a connect sum of two unknots, where perturbing the red loop to avoid the crossing yields the necessary separating circle for the decomposition.
Another connect sum of two unknots that is not equivalent to this one is Kishino's knot (\Cref{fig:kishino}).

The virtual Kauffman bracket satisfies the relation $\langle L\rangle = \langle L_1\rangle \langle L_2\rangle$, but the arrow polynomial cannot satisfy such a relation in general.
For example, the virtual knot in \Cref{fig:non-reduced} and Kishino's knot both have nontrivial (and unequal) arrow polynomials.
However, there is still such a relation in certain cases.
Note that the $-A^2-A^{-2}$ factor that will appear is from the fact that we do not normalize the homological arrow polynomial to be $1$ for the unknot.

Given a surface link $L\subset\Sigma\times I$ with a connect sum annulus $A\subset\Sigma\times I$, then after choosing an arc in $A$ between the two points of $A\cap L$ we can produce two surface links $L_1$ and $L_2$ in $\Sigma\times I-A$ by using the arc to perform a saddle move on $L$.
These links have the property that $[L]=[L_1]+[L_2]$ in $H_1(\Sigma\times I)$.

\begin{proposition}
  \label{thm:connect-sum}
  Let $L\subset\Sigma\times I$ be a framed oriented surface link that is a connect sum, and, after choosing an arc in the connect sum annulus, let $L_1$ and $L_2$ be the framed oriented surface links in $\Sigma\times I$ as above.
  If $L$ is $1$-labeled and $L_2$ is $\Z$-\nullhomologous{}, then
  \begin{equation*}
    (-A^2-A^{-2})\mathcal{A}(L)=\mathcal{A}(L_1)\mathcal{A}(L_2).
  \end{equation*}
  In particular, this equation holds if the connect sum annulus bounds a ball in $\Sigma\times I$.
\end{proposition}
\begin{proof}
  Since we assume $L_2$ is $\Z$-\nullhomologous{} in $\Sigma\times I$, then $[L]=[L_1]$, and so when computing the arrow polynomial we can omit all whiskers that come from the $L_2$ side of $L$.
  Consider the following state sum for a partial expansion of $\mathcal{A}(L)$, where $s$ ranges over states from smoothing crossings on the $L_2$ side, $\ell(s)$ counts the number of state loops that form on the $L_2$ side, and $L_s$ is the surface link from applying the smoothings and removing all the state loops that form on the $L_2$ side:
  \[
    \mathcal{A}(L) = \sum_s A^{a(s)-b(s)} (-A^2-A^{-2})^{\ell(s)} \mathcal{A}(L_s).
  \]
  Evidently, each $L_s$ is virtually equivalent to $L_1$ since after removing all the state loops one has a trivial arc on the $L_2$ side.
  Since $\ell(s)$ counts all but one state loop if we were to apply the same smoothings to $L_2$ itself, we have that
  \[
    \mathcal{A}(L) = \mathcal{A}(L_1) \sum_s A^{a(s)-b(s)} (-A^2-A^{-2})^{\ell(s)} = \mathcal{A}(L_1) \langle L_2\rangle.
  \]
  By \Cref{thm:z-nullhomologous-poly} we can then substitute $\langle L_2\rangle = \mathcal{A}(L_2)/(-A^2-A^{-2})$, completing the proof.
\end{proof}

\begin{remark}
  It is important that $L_2$ is $\Z$-\nullhomologous{} within the same thickened surface --- the proposition is not true if we only require that $L_2$ be $\Z$-\nullhomologous{} as a virtual link.
  Kishino's knot is a counterexample.
\end{remark}

\begin{remark}
  The hypothesis for the proposition has another equivalent formulation.
  Letting $A$ be the connect sum annulus, let $M_1$ and $M_2$ be the two sides of $\Sigma\times I$ after cutting along $A$.
  Then the homological arrow polynomial is multiplicative if $[L]$ is in the image of the map $H_1(M_1)\to H_1(\Sigma\times I)$ induced from the inclusion $M_1\incl \Sigma\times I$.
\end{remark}

\subsection{Mutation}

For a surface link $L\subset\Sigma\times I$, a \emph{Conway annulus} is a vertical annulus $A\subset\Sigma\times I$ that separates $\Sigma\times I$ such that $A$ intersects $L$ transversely in exactly four points.
This corresponds to a separating circle in $\Sigma$ that intersects the diagram for $L$ in exactly four points.

We can cut $\Sigma\times I$ along $A$ (cutting $L$ in four points) to get two \emph{virtual tangles} with four boundary points.
We can then choose a self-diffeomorphism of $A$ that carries the four points to themselves to glue the two virtual tangles back together.
If the diffeomorphism is chosen such that the pieces are glued together with matching orientations and such that the diffeomorphism acts on the four boundary points as an element of the Klein four-group (i.e., it's either the identity or the boundary points are partitioned into two two-element orbits), then we call the resulting surface link a \emph{Conway mutation} of $L\subset \Sigma\times I$.
Regarding the orientation condition, one might imagine $\Sigma\times I$ being embedded in an ambient $\R^3$ and then rotating one of the pieces before regluing.

For a virtual link diagram in the plane, mutation corresponds to finding a disk whose boundary transversely intersects the diagram in four non-crossing points and then applying any composition of the following operation: choose a reflection that switches pairs of these four points, then apply it to the interior of the disk while taking its mirror image (see \Cref{fig:mutation}).
The mirror image is so that it is as if we are picking up the disk and rotating it in $3$-space.

The arrow polynomial is not invariant under mutation.
For example, the mutant virtual knots in \Cref{fig:mutants} (from \cite{Folwaczny2012}) have
\begin{align*}
  \langle K_{3.2}\rangle_{\mathrm{NA}} &=  A^{-8} - A^{-4} + 1 + (-A^{-2} + A^{2}) K_1 \\
  \langle K_{5.632}\rangle_{\mathrm{NA}} &=  -A^{-4} + 1 - A^{-4} K_4 + (A^{-8} + A^{-4}) K_2^2 + (-A^{-2} + A^{2}) K_1.
\end{align*}
Virtual Jones polynomials are invariant under mutation, and, as expected, substituting $K_i=1$ for all $i$ yields identical polynomials.

% PD[Xp[6,7,5,8], Xp[1,2,10,3], Xp[5,8,4,9], Xm[4,10,3,9], Xm[2,7,1,6]]
% PD[Xp[1,6,10,7], Xm[4,9,3,8], Xp[10,7,9,8], Xp[2,5,1,6], Xm[5,3,4,2]]

\begin{figure}[tb]
  \centering
  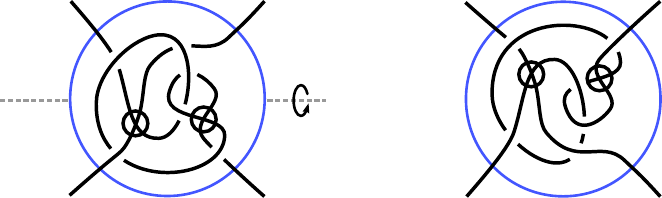
  \caption{Example effect of a mutation inside the mutation disk.}
  \label{fig:mutation}
\end{figure}

\begin{figure}[tb]
  \centering
  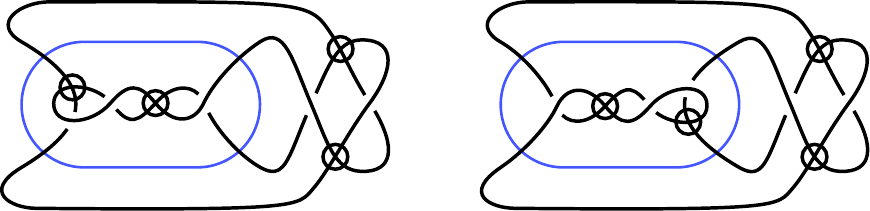
  \caption{Mutant virtual knots (3.2 and 5.632) with different arrow polynomials.
    The mutation disk within the blue loop has been rotated 180 degrees.}
  \label{fig:mutants}
\end{figure}

However, the arrow polynomial is invariant under certain types of mutations.
\begin{proposition}
  Let $C$ be a Conway circle for a diagram of a surface link $L$ in $\Sigma$, and let $\Sigma_1$ and $\Sigma_2$ be the closures of the two pieces of $\Sigma-C$.
  If $\pi_*([L])$ is in the image of the map $H_1(\Sigma_1)\to H_1(\Sigma)$ induced from the inclusion $\Sigma_1\incl\Sigma$, then the arrow polynomial for $L$ is invariant under every mutation with respect to $C$.

  In particular, if $C$ bounds a disk in $\Sigma$ then the arrow polynomial is invariant under mutations with respect to $C$.
\end{proposition}
\begin{proof}
  This is similar to \Cref{thm:connect-sum}.
  If $\pi_*([L])$ is in the image of $H_1(\Sigma_1)\to H_1(\Sigma)$ then we can choose a link on the $\Sigma_1$ side homologous to $L$ to compute the arrow polynomial, and so in the whisker expansion we can omit all whiskers from the $\Sigma_2$ side.
  We can compute a partial expansion using smoothings on the $\Sigma_2$ side, and, after removing closed state loops from the $\Sigma_2$ side, we
  put $\mathcal{A}(L)$ into the form of a $\Z[A^{\pm 1}]$-linear combination of the three arrow polynomials with the $\Sigma_2$ side replaced by each of the following three virtual tangles:
  \[
    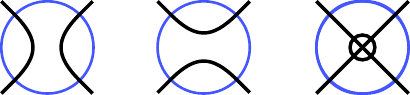
  \]
  These are all invariant under mutation, so therefore $\mathcal{A}(L)$ is as well.
\end{proof}

\begin{corollary}
  If $C$ is a Conway circle for a virtual link diagram of a link $L$ such that one side of $C$ contains no virtual crossings, then the arrow polynomial for $L$ is invariant under mutations with respect to $C$.
\end{corollary}

\section{Computational investigations}
\label{sec:computational-investigations}

We have created computer programs to compute the arrow polynomial.
\Cref{sec:computing-arrow,sec:computing-harrow} have reference implementations written in Mathematica for the arrow polynomial and the homological arrow polynomial.
These use a ``virtual Temperley--Lieb planar algebra'' that has been augmented with whiskers, and the implementation uses the fact that elements of this algebraic structure can be reduced to a normal form using simple rules.

There is a more-efficient JavaScript implementation of the arrow polynomial in KnotFolio\cite{KnotFolio}, which is an online tool at \url{https://kmill.github.io/knotfolio/} for computing invariants of knots, and it includes a feature to identify virtual knots according to Green's census\cite{Green2004}.
The tool can compute the \emph{$n$-cabled arrow polynomial}, which is the arrow polynomial of the $n$-cabling of the $0$-framing of a given oriented virtual link (and so the $1$-cabled arrow polynomial is the normalized arrow polynomial itself).
By computing these polynomials for small values of $n$ and consulting a pre-computed table, KnotFolio is able to uniquely identify almost all virtual knots with up to five crossings.

We also have a fast implementation of cabled arrow polynomials at \url{https://github.com/kmill/arrow_poly} written in Lean~4, a relatively new functional programming language that compiles to C and has a sophisticated dependent type system, which allows one to, entirely within Lean~4, state theorems about programs and prove them\cite{Moura2021}.
We have so far only used these proof capabilities to prove basic correctness properties of the program, but we plan to prove the program computes virtual knot invariants (and, once Lean~4 supports WASM, this implementation will replace the one in KnotFolio).
In the meantime, our confidence in the implementations stems from the fact that the JavaScript and Lean~4 implementations, which were each implemented from scratch, calculated the same $1$- and $2$-cabled arrow polynomials for all the virtual knots in Green's census with up to five crossings.

While the JavaScript implementation was designed with speed in mind, the Lean~4 implementation is significantly faster.
The JavaScript computation of all $1$- and $2$-cabled arrow polynomials for the $2565$ virtual knots in Green's census with up to five crossings took $9.5$ days of computer time on an Intel Xeon E5-2665, but the Lean~4 computation of $1$-, $2$-, and $3$-cabled arrow polynomials for the same knots took $2.8$ days of computer time on the same CPU.
Note that the computation of $n$-cabled arrow polynomials for virtual knots with $c$ crossings is similar in complexity to computing virtual Jones polynomials of virtual knots with $cn^2$ crossings, and so the Lean~4 program is effectively handling $45$-crossing virtual knots, with the caveat that it is pre-computing a cache of expansions of common tangles to accelerate the computation.
We tested the program on $4$-cabled arrow polynomials of two $5$-crossing virtual knots, and each polynomial took about ten hours to compute.

We should mention that the significance of $n$-cabled arrow polynomials is that they can be used to compute what might be called the \emph{colored arrow polynomials}, where the $n$th colored arrow polynomial is from taking an $n$-cabling of the virtual link with spliced-in $n$th Jones--Wenzl projectors.
This yields an invariant in $\Q(A)[K_1,K_2,\dots]$.
Just like for colored Jones polynomials, the first $n$ cabled arrow polynomials determine the first $n$ colored arrow polynomials and vice versa.
Also, note that given the first $n$ cabled arrow polynomials of a virtual link one can compute the first $n$ cabled arrow polynomials for the virtual link with any framing whatsoever --- this is justification for taking the $0$-framing in particular.

Now that we have described the programs, we now describe the degree to which cabled arrow polynomials distinguish virtual knots.
We are using the Green census of unoriented virtual knots modulo vertical and horizontal mirror images (\texttt{knots-6.txt}).

For the $117$ virtual knots with up to four crossings, the arrow polynomial by itself distinguishes all but the $62$ in \Cref{tab:arrow-non-unique-4-1}, but the $1$- and $2$-cabled arrow polynomials together fully determine the virtual knot.
In comparison, the $1$- and $2$-cabled Jones polynomials are unable to distinguish the virtual knots listed in \Cref{tab:jones-non-unique-4-2}.

For the $2565$ virtual knots with up to five crossings, the $1$- and $2$-cabled arrow polynomials together distinguish all but the $22$ listed in \Cref{tab:arrow-non-unique-5-2}.
The $1$-, $2$-, and $3$-cabled arrow polynomials distinguish all the five-crossing virtual knots except for the $18$ in \Cref{tab:arrow-non-unique-5-3}.
In comparison, the $1$-, $2$-, and $3$-cabled Jones polynomials fail to distinguish the $74$ in \Cref{tab:jones-non-unique-5-3}.
The Alexander polynomial is able to distinguish the four marked pairs in \Cref{tab:arrow-non-unique-5-3}, hence there are only five undistinguished pairs using only these invariants.

We also computed the $4$-cabled arrow polynomials of 5.196 and 5.1662 and found that it is able to distinguish this pair.
Since these polynomials each take ten hours to compute we have not pursued calculating the $4$-cabled arrow polynomials of the remaining pairs at this time.

It is tempting to take advantage of the fact that an $n$-cabling of a virtual knot is an $n$-component virtual link, so we might use homological arrow polynomials.
However, note that the components in the $n$-cabling are homologous to one another and so the $n$-cabled arrow polynomial determines the homological arrow polynomial of the $n$-cabling.

\begin{table}[tb]
  \centering
  \caption{Virtual knots up to four crossings with non-unique arrow polynomials.  Each cell consists of virtual knots sharing arrow polynomials.}
  \label{tab:arrow-non-unique-4-1}
  \def\arraystretch{1.2}
  \footnotesize
  \begin{tabular}{|l|l|l|l|}
    \hline
    0.1, 4.46, 4.72, 4.98, 4.107 &
    2.1, 4.33, 4.44 &
    3.2, 4.27 &
    3.6, 4.105 \\
    \hline
    3.7, 4.85, 4.96, 4.106 &
    4.1, 4.7 &
    4.2, 4.8, 4.51, 4.71 &
    4.4, 4.5, 4.18, 4.30 \\
    \hline
    4.9, 4.61 &
    4.11, 4.63 &
    4.13, 4.55, 4.56 &
    4.15, 4.29 \\
    \hline
    4.16, 4.68 &
    4.19, 4.67 &
    4.20, 4.34 &
    4.25, 4.43 \\
    \hline
    4.26, 4.47, 4.97 &
    4.28, 4.45, 4.83 &
    4.38, 4.49 &
    4.40, 4.52 \\
    \hline
    4.50, 4.70 &
    4.58, 4.75 &
    4.59, 4.76, 4.77 &
    4.99, 4.108 \\
    \hline
  \end{tabular}
\end{table}

\begin{table}[tb]
  \centering
  \caption{Virtual knots up to five crossings with non-unique $1$- and $2$-cabled arrow polynomials.  Each cell consists of virtual knots with the same such polynomials.}
  \label{tab:arrow-non-unique-5-2}
  \def\arraystretch{1.2}
  \footnotesize
  \begin{tabular}{|l|l|l|l|}
    \hline
    5.196, 5.1662 &
    5.197, 5.1657 &
    5.204, 5.1670 &
    5.205, 5.1665 \\
    \hline
    5.287, 5.1168 &
    5.294, 5.1175 &
    5.295, 5.1176 &
    5.302, 5.1183 \\
    \hline
    5.757, 5.760 &
    5.1113, 5.1124 &
    5.2322, 5.2411 \\
    \cline{1-3}
  \end{tabular}
\end{table}

\begin{table}[tb]
  \centering
  \caption{Virtual knots up to five crossings with non-unique $1$-, $2$-, and $3$-cabled arrow polynomials.  Each cell consists of virtual knots with the same such polynomials. Pairs indicated by \dag{} are distinguishable by their Alexander polynomials.}
  \label{tab:arrow-non-unique-5-3}
  \def\arraystretch{1.2}
  \footnotesize
  \begin{tabular}{|l|l|l|}
    \hline
    5.196, 5.1662 &
    5.197, 5.1657 &
    5.204, 5.1670 \\
    \hline
    5.205, 5.1665 &
    5.287, 5.1168 \dag{} & % alex
    5.294, 5.1175 \dag{} \\ % alex
    \hline
    5.295, 5.1176 \dag{} & % alex
    5.302, 5.1183 \dag{} & % alex
    5.2322, 5.2411 \\
    \hline
  \end{tabular}
\end{table}

\begin{table}[tb]
  \centering
  \caption{Virtual knots up to four crossings with non-unique $1$- and $2$-cabled Jones polynomials.  Each cell consists of virtual knots with the same such polynomials.}
  \label{tab:jones-non-unique-4-2}
  \def\arraystretch{1.2}
  \footnotesize
  \begin{tabular}{|l|l|l|l|}
    \hline
    0.1, 4.55, 4.56, 4.76, 4.77 &
    2.1, 4.4, 4.5, 4.54, 4.74 &
    3.3, 4.63 &
    4.1, 4.3, 4.7, 4.53, 4.73 \\
    \hline
    4.2, 4.6, 4.8, 4.12, 4.75 &
    4.13, 4.59, 4.107 &
    4.19, 4.42 &
    4.26, 4.97 \\
    \hline
    4.28, 4.83 &
    4.95, 4.101 \\
    \cline{1-2}
  \end{tabular}
\end{table}

\begin{table}[tb]
  \centering
  \caption{Virtual knots up to five crossings with non-unique $1$-, $2$-, and $3$-cabled Jones polynomials.  Each cell consists of virtual knots with the same such polynomials.}
  \label{tab:jones-non-unique-5-3}
  \def\arraystretch{1.2}
  \footnotesize
  \begin{tabular}{|l|l|l|l|l|}
    \hline
    5.15, 5.116 &
    5.16, 5.117 &
    5.23, 5.71 &
    5.24, 5.73 &
    5.25, 5.72 \\
    \hline
    5.26, 5.74 &
    5.58, 5.94 &
    5.59, 5.95 &
    5.60, 5.96 &
    5.61, 5.97 \\
    \hline
    5.196, 5.1662 &
    5.197, 5.1657 &
    5.204, 5.1670 &
    5.205, 5.1665 &
    5.287, 5.1168 \\
    \hline
    5.294, 5.1175 &
    5.295, 5.1176 &
    5.302, 5.1183 &
    5.661, 5.662 &
    5.754, 5.763 \\
    \hline
    5.757, 5.760 &
    5.807, 5.1672 &
    5.808, 5.1674 &
    5.809, 5.1673 &
    5.810, 5.1675 \\
    \hline
    \multicolumn{2}{|l|}{5.811, 5.814, 5.1676, 5.1679} &
    \multicolumn{2}{|l|}{5.812, 5.813, 5.1677, 5.1678} &
    5.1113, 5.1124 \\
    \hline
    5.1116, 5.1121 &
    5.1184, 5.1187 &
    5.1186, 5.1189 &
    \multicolumn{2}{|l|}{5.1190, 5.1192, 5.1193, 5.1195} \\
    \hline
    5.1191, 5.1194 &
    5.2322, 5.2411 \\
    \cline{1-2}
  \end{tabular}
\end{table}

\section*{Acknowledgements}
I would like to thank Louis Kauffman for suggesting that I implement the arrow polynomial for KnotFolio\cite{KnotFolio} and for sharing Mathematica code with me, Allison Henrich for some discussion about virtual knots and for encouraging me to write this paper, and Ian Agol for his usual helpful advice.
I would also like to thank Hans Boden for valuable feedback on an early draft of this paper.
{Furthermore, I thank the anonymous reviewer for their careful reading and for their suggested improvements.}

\appendix
\section{Computing the arrow polynomial}
\label{sec:computing-arrow}

This section contains a Mathematica program for computing the non-writhe-normalized arrow polynomial of virtual links.
It serves as a reference implementation and is not optimized for speed of execution.

The input to the algorithm are framed virtual knots given in oriented PD notation
Recall that PD notation represents the combinatorial data of a knot diagram (a type of ribbon graph), with a node per crossing (a vertex of the ribbon graph).
The edges of the ribbon graph, called \emph{arcs}, are given arbitrary unique integer labels, and, for oriented virtual knots, each crossing is read off using the convention in \Cref{fig:pd-codes}.
The \texttt{P} node is useful for either representing the unknot (e.g. \texttt{PD[P[1, 1]]}) or for whenever it might be convenient to be able to represent subdivided arcs.
We are following the convention of the \texttt{KnotTheory\`} Mathematica package for oriented PD notation, where oriented crossing nodes are represented as \texttt{Xp} and \texttt{Xm} expressions.\footnote{See \url{http://katlas.org/wiki/Planar_Diagrams}.}
Note that PD stands for ``planar diagram,'' but the notation works for nonplanar diagrams (i.e., virtual knots) perfectly well.
An example PD code for the virtual knot $2.1$ is given in \Cref{fig:pd-2-1}.
We choose not to represent virtual crossings in PD notation, but one is free to extend the notation to include them.

\begin{figure}[tb]
  \centering
  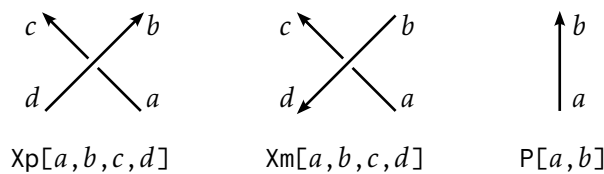
  \caption{Node types for oriented PD codes, following the convention for the \texttt{KnotTheory\`} Mathematica package. For unoriented PD codes, both crossing types are represented as \texttt{X[$a$,$b$,$c$,$d$]}.}
  \label{fig:pd-codes}
\end{figure}

\begin{figure}[tb]
  \centering
  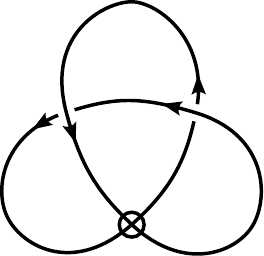
  \caption{The ``virtual trefoil'' $2.1$ with arc labeled 1--4, with labels displayed for arcs around each crossing. This has PD code \texttt{PD[Xm[1,2,3,4],Xm[4,3,1,2]]}. We do not represent virtual crossings in PD codes.}
  \label{fig:pd-2-1}
\end{figure}

In the following listing, the rules in \texttt{arrowRules} give the expansion of oriented crossings using whiskers as described in the proof of \Cref{thm:h-evenness}.
Since we wish to normalize the arrow polynomial such that the arrow polynomial of the zero-framed unknot is $1$, the \texttt{arrow} function cuts open the virtual link to form a $1$-$1$ tangle since this saves having to divide by $-A^2-A^{-2}$ in the end.

\begin{lstlisting}
ClearAll[whiskers, K, arrow];

(* K[n] corresponds to the variable K_n *)
K[0] = 1;

(* whiskers[n, a, b] represents a path from arc id a to arc id b with
   n whiskers to the left if n is positive, and -n whiskers to the right
   if n is negative *)

whiskers /: whiskers[n_, a_, b_] whiskers[m_, b_, c_] := whiskers[n + m, a, c];
whiskers /: whiskers[n_, a_, b_] whiskers[m_, c_, b_] := whiskers[n - m, a, c];
whiskers /: whiskers[n_, a_, b_] whiskers[m_, a_, c_] := whiskers[n - m, c, b];
whiskers /: x_whiskers^2 := x x; (* activate above rules *)
whiskers[n_, a_, a_] := (-A^2 - A^-2) K[Abs[n]/2];
whiskers[n_, a_, b_] /; b < a := whiskers[-n, b, a]; (* a normalization *)

arrowRules = {
  Xp[a_, b_, c_, d_] :> 
    A whiskers[0, a, b] whiskers[0, c, d]
    + A^-1 whiskers[-1, a, d] whiskers[1, c, b],
  Xm[a_, b_, c_, d_] :> 
    A^-1 whiskers[0, b, c] whiskers[0, d, a]
    + A whiskers[-1, b, a] whiskers[1, d, c],
  P[a_, b_] :> whiskers[0, a, b]
};

arrow[pd_PD] := With[{max = Max[List @@@ (List @@ pd)]},
  With[{newpd = ReplacePart[pd, FirstPosition[pd, max] -> max + 1]},
    (* Now newpd is a 1-1 tangle *)
    With[{exp = Expand[Times @@ (newpd /. arrowRules)] /.
                  whiskers[n_, max, max + 1] :> K[Abs[n]/2]
         },
      exp // Collect[#, A, Simplify] &]]];
\end{lstlisting}

\noindent For example,

\begin{lstlisting}
(* 2.1 - Virtual trefoil *)
In[1]:= green2n1 = PD[Xm[1, 2, 3, 4], Xm[4, 3, 1, 2]];
In[2]:= arrow[green2n1]
Out[2]= 1/A^2 + K[1] - A^4 K[1]

(* 3.7 - Virtualized trefoil *)
In[3]:= green3n7 = PD[Xm[2, 5, 1, 4], Xp[4, 6, 3, 1], Xp[6, 2, 5, 3]];
In[4]:= arrow[green3n7]
Out[4]= -A^3 K[1]^2 + (-1 + K[1]^2)/A^5

(* Virtual Hopf link *)
In[5]:= vhopf = PD[Xm[1, 2, 1, 2]];
In[6]:= arrow[vhopf]
Out[6]= 1/A + A K[1]

(* 4.105 *)
In[7]:= green4n105 = PD[Xp[8, 4, 7, 5], Xp[4, 8, 3, 1],
                        Xp[2, 6, 1, 7], Xp[6, 2, 5, 3]];
In[7]:= arrow[green4n105]
Out[7]= 1 - 1/A^4 + A^8
\end{lstlisting}

\section{Computing the homological arrow polynomial}
\label{sec:computing-harrow}

This section gives a reference implementation in Mathematica for the homological arrow polynomial $\mathcal{A}(L)$ of a framed virtual link $L$ given in oriented PD notation (see \Cref{sec:computing-arrow} for a review).

This implementation supports components labeled from the set $\{1,\dots,10\}$ by encoding the component label of an arc by its label modulo $10$, so for example arc $5$ component $1$ would be labeled by $51$ in the PD code.
It is up to the user to verify that all arc ids in a given component are the same modulo $10$.
Different components may be given the same component label.

Note that, in contrast to \texttt{arrow} in \Cref{sec:computing-arrow}, the \texttt{harrow} function does not normalize the unknot to be $1$, and instead \texttt{harrow[PD[P[11, 11]]]} is $-A^2-A^{-2}$.
For simplicity of implementation, the variables $X_{\pm(i_1,\dots,i_n)}$ and $X_{\pm(i_1,\dots,i_n,0)}$ are equal.

\begin{lstlisting}
ClearAll[X, hwhiskers, V, mkV, harrow];

(* X[i1,i2,...] represents X_{\pm(i1,i2,...)} *)
X[cs___, 0] := X[cs];
X[] = 1;
(* Normalize tuple such that first nonzero index is positive. *)
X[zeros : (0 ...), c_, cs___] /; c < 0 := X[zeros, -c, Sequence @@ (-{cs})];

(* Vectors in Z^infinity *)
V[cs___, 0] := V[cs];
V /: V[cs1___] + V[cs2___] := With[{l1 = {cs1}, l2 = {cs2}},
   With[{l1x = Join[l1, Table[0, Length[l2] - Length[l1]]],
         l2x = Join[l2, Table[0, Length[l1] - Length[l2]]]},
    V @@ (l1x + l2x)]];
V /: n_Integer V[cs___] := V @@ (n {cs});

(* Creates a unit vector using the arc id modulo 10 *)
mkV[id_] := With[{label = Mod[id, 10, 1]}, V @@ UnitVector[label, label]];

(* hwhiskers[vec, a, b] is an arc from a to b with leftward whiskers described by vec *)
hwhiskers /: hwhiskers[n_, a_, b_] hwhiskers[m_, b_, c_] := hwhiskers[n + m, a, c];
hwhiskers /: hwhiskers[n_, a_, b_] hwhiskers[m_, c_, b_] := hwhiskers[n - m, a, c];
hwhiskers /: hwhiskers[n_, a_, b_] hwhiskers[m_, a_, c_] := hwhiskers[n - m, c, b];
hwhiskers /: x_hwhiskers^2 := x x;
hwhiskers[n_, a_, a_] := (-A^2 - A^-2) X @@ n;
hwhiskers[n_, a_, b_] /; b < a := hwhiskers[-n, b, a]; (* a normalization *)

hArrowRules = {
  Xp[a_, b_, c_, d_] :> 
   A hwhiskers[V[], a, b] hwhiskers[mkV[c] - mkV[d], d, c]
   + A^-1 hwhiskers[mkV[a], d, a] hwhiskers[mkV[b], c, b],
  Xm[a_, b_, c_, d_] :> 
   A^-1 hwhiskers[V[], b, c] hwhiskers[mkV[d] - mkV[a], a, d]
   + A hwhiskers[mkV[b], a, b] hwhiskers[mkV[c], d, c],
  P[a_, b_] :> hwhiskers[V[], a, b]
};

harrow[pd_PD] :=
  Expand[Times @@ (pd /. hArrowRules)] // Collect[#, A, Collect[#, _X, Simplify] &] &;
\end{lstlisting}

\noindent{}Examples:

\begin{lstlisting}[escapechar=@]
(* 2.1 - Virtual trefoil *)
In[1]:= harrow[PD[Xm[11, 21, 31, 41], Xm[41, 31, 11, 21]]]
Out[1]= -1 - 1/A^4 + (-(1/A^2) + A^6) X[2]

(* Virtual Hopf link *)
In[2]:= harrow[PD[Xm[11, 22, 11, 22]]]
Out[2]= (-(1/A^3) - A) X[1, -1] + (-(1/A) - A^3) X[1, 1]

(* Virtual Hopf link, both components with same label *)
In[3]:= harrow[PD[Xm[11, 21, 11, 21]]]
Out[3]= -(1/A^3) - A + (-(1/A) - A^3) X[2]

(* The link from @\Cref{fig:vlink1}@ *)
In[4]:= harrow[PD[Xm[82, 31, 72, 21], Xm[72, 41, 62, 31],
                  Xm[21, 52, 11, 82], Xm[11, 62, 41, 52]]]
Out[4]= 1 + 2/A^4 + A^8 + (1/A^12 - 2/A^4 + A^4) X[1, -1]^2
\end{lstlisting}

\let\MRhref\undefined
\bibliographystyle{hamsalpha}
\bibliography{refs.bib}

\end{document}

%% file: preamble.tex
\usepackage[T1]{fontenc}
\usepackage[utf8]{inputenc}
\usepackage[english]{babel}
\usepackage[autostyle]{csquotes}
\usepackage[full]{textcomp}

\usepackage[margin = 1in]{geometry}

\usepackage{kpfonts}
\usepackage{inconsolata}

\usepackage{amsfonts, amsmath, amsthm, thmtools, amssymb, mathtools}
\usepackage{tikz-cd}
\allowdisplaybreaks

\usepackage{booktabs}
\usepackage{adjustbox}

\usepackage{cite}
\usepackage{hyperref}
\usepackage{cleveref}
\usepackage{xcolor}
\hypersetup{
    colorlinks,
    linkcolor={red!50!black},
    citecolor={blue!50!black},
    urlcolor={blue!80!black}
}

\usepackage{listings}
\definecolor{mygreen}{rgb}{0,0.4,0}
\lstdefinestyle{custom}{
  xleftmargin=\parindent,
  language=mathematica,
  basicstyle=\footnotesize\ttfamily,
  keywordstyle=\bfseries,
  commentstyle=\color{mygreen},
  breaklines=false,
  deletekeywords={K},
  otherkeywords={FirstPosition}
}
\usepackage{graphicx}
\graphicspath{{figs/}}

\numberwithin{equation}{section}

\theoremstyle{plain}
\newtheorem{theorem}{Theorem}[section]
\newtheorem{proposition}[theorem]{Proposition}
\newtheorem{lemma}[theorem]{Lemma}
\newtheorem{corollary}[theorem]{Corollary}

\theoremstyle{definition}
\newtheorem{definition}[theorem]{Definition}
\newtheorem{example}[theorem]{Example}
\theoremstyle{remark}
\newtheorem{remark}[theorem]{Remark}

\newcommand{\incl}{\hookrightarrow}
\newcommand{\surj}{\twoheadrightarrow}

\newcommand{\defeq}{:=}

\DeclareMathOperator{\Hom}{Hom}

\DeclareMathOperator{\id}{id}

\renewcommand{\phi}{\varphi}
\DeclarePairedDelimiter\abs{\lvert}{\rvert}

\newcommand{\N}{\mathbb{N}}
\newcommand{\Z}{\mathbb{Z}}
\newcommand{\Q}{\mathbb{Q}}
\newcommand{\R}{\mathbb{R}}

\DeclareMathOperator{\breadth}{breadth}
\DeclareMathOperator{\maxdeg}{maxdeg}
\DeclareMathOperator{\mindeg}{mindeg}

\newcommand{\csum}{\mathbin{\#}}
\DeclareMathOperator{\Sk}{Sk}
\DeclareMathOperator{\writhe}{wr}

%% file: figs/reg-reid.pdf_tex
%% Creator: Inkscape 1.2 (dc2aedaf03, 2022-05-15), www.inkscape.org
%% Accompanies image file 'reg-reid.pdf' (pdf, eps, ps)
\begingroup%
  \makeatletter%
  \providecommand\color[2][]{%
    \errmessage{(Inkscape) Color is used for the text in Inkscape, but the package 'color.sty' is not loaded}%
    \renewcommand\color[2][]{}%
  }%
  \providecommand\transparent[1]{%
    \errmessage{(Inkscape) Transparency is used (non-zero) for the text in Inkscape, but the package 'transparent.sty' is not loaded}%
    \renewcommand\transparent[1]{}%
  }%
  \providecommand\rotatebox[2]{#2}%
  \newcommand*\fsize{\dimexpr\f@size pt\relax}%
  \newcommand*\lineheight[1]{\fontsize{\fsize}{#1\fsize}\selectfont}%
  \ifx\svgwidth\undefined%
    \setlength{\unitlength}{159.12000275bp}%
    \ifx\svgscale\undefined%
      \relax%
    \else%
      \setlength{\unitlength}{\unitlength * \real{\svgscale}}%
    \fi%
  \else%
    \setlength{\unitlength}{\svgwidth}%
  \fi%
  \global\let\svgwidth\undefined%
  \global\let\svgscale\undefined%
  \makeatother%
  \begin{picture}(1,0.88195869)%
    \lineheight{1}%
    \setlength\tabcolsep{0pt}%
    \put(0,0){\includegraphics[width=\unitlength,page=1]{reg-reid.pdf}}%
    \put(0.57083078,0.76825535){\makebox(0,0)[t]{\lineheight{1.25}\smash{\begin{tabular}[t]{c}$\sim$\end{tabular}}}}%
    \put(0.07898348,0.76825535){\makebox(0,0)[rt]{\lineheight{1.25}\smash{\begin{tabular}[t]{r}I\textquotesingle)\end{tabular}}}}%
    \put(0,0){\includegraphics[width=\unitlength,page=2]{reg-reid.pdf}}%
    \put(0.57083078,0.44774223){\makebox(0,0)[t]{\lineheight{1.25}\smash{\begin{tabular}[t]{c}$\sim$\end{tabular}}}}%
    \put(0.07898349,0.44774223){\makebox(0,0)[rt]{\lineheight{1.25}\smash{\begin{tabular}[t]{r}II)\end{tabular}}}}%
    \put(0,0){\includegraphics[width=\unitlength,page=3]{reg-reid.pdf}}%
    \put(0.57083078,0.1272291){\makebox(0,0)[t]{\lineheight{1.25}\smash{\begin{tabular}[t]{c}$\sim$\end{tabular}}}}%
    \put(0.07898349,0.1272291){\makebox(0,0)[rt]{\lineheight{1.25}\smash{\begin{tabular}[t]{r}III)\end{tabular}}}}%
  \end{picture}%
\endgroup%

%% file: figs/arrow-cusps-small.pdf_tex
%% Creator: Inkscape 1.2 (dc2aedaf03, 2022-05-15), www.inkscape.org
%% Accompanies image file 'arrow-cusps-small.pdf' (pdf, eps, ps)
\begingroup%
  \makeatletter%
  \providecommand\color[2][]{%
    \errmessage{(Inkscape) Color is used for the text in Inkscape, but the package 'color.sty' is not loaded}%
    \renewcommand\color[2][]{}%
  }%
  \providecommand\transparent[1]{%
    \errmessage{(Inkscape) Transparency is used (non-zero) for the text in Inkscape, but the package 'transparent.sty' is not loaded}%
    \renewcommand\transparent[1]{}%
  }%
  \providecommand\rotatebox[2]{#2}%
  \newcommand*\fsize{\dimexpr\f@size pt\relax}%
  \newcommand*\lineheight[1]{\fontsize{\fsize}{#1\fsize}\selectfont}%
  \ifx\svgwidth\undefined%
    \setlength{\unitlength}{168.88078952bp}%
    \ifx\svgscale\undefined%
      \relax%
    \else%
      \setlength{\unitlength}{\unitlength * \real{\svgscale}}%
    \fi%
  \else%
    \setlength{\unitlength}{\svgwidth}%
  \fi%
  \global\let\svgwidth\undefined%
  \global\let\svgscale\undefined%
  \makeatother%
  \begin{picture}(1,0.53245942)%
    \lineheight{1}%
    \setlength\tabcolsep{0pt}%
    \put(0,0){\includegraphics[width=\unitlength,page=1]{arrow-cusps-small.pdf}}%
    \put(0.17841355,0.41179943){\makebox(0,0)[lt]{\lineheight{1.25}\smash{\begin{tabular}[t]{l}${}=A$\end{tabular}}}}%
    \put(0,0){\includegraphics[width=\unitlength,page=2]{arrow-cusps-small.pdf}}%
    \put(0.49929265,0.41179943){\makebox(0,0)[lt]{\lineheight{1.25}\smash{\begin{tabular}[t]{l}${}+A^{-1}$\end{tabular}}}}%
    \put(0.17841355,0.09338052){\makebox(0,0)[lt]{\lineheight{1.25}\smash{\begin{tabular}[t]{l}${}=A$\end{tabular}}}}%
    \put(0,0){\includegraphics[width=\unitlength,page=3]{arrow-cusps-small.pdf}}%
    \put(0.49929265,0.09338052){\makebox(0,0)[lt]{\lineheight{1.25}\smash{\begin{tabular}[t]{l}${}+A^{-1}$\end{tabular}}}}%
    \put(0,0){\includegraphics[width=\unitlength,page=4]{arrow-cusps-small.pdf}}%
  \end{picture}%
\endgroup%

%% file: figs/cusp-magnetic-combined.pdf_tex
%% Creator: Inkscape 1.2 (dc2aedaf03, 2022-05-15), www.inkscape.org
%% Accompanies image file 'cusp-magnetic-combined.pdf' (pdf, eps, ps)
\begingroup%
  \makeatletter%
  \providecommand\color[2][]{%
    \errmessage{(Inkscape) Color is used for the text in Inkscape, but the package 'color.sty' is not loaded}%
    \renewcommand\color[2][]{}%
  }%
  \providecommand\transparent[1]{%
    \errmessage{(Inkscape) Transparency is used (non-zero) for the text in Inkscape, but the package 'transparent.sty' is not loaded}%
    \renewcommand\transparent[1]{}%
  }%
  \providecommand\rotatebox[2]{#2}%
  \newcommand*\fsize{\dimexpr\f@size pt\relax}%
  \newcommand*\lineheight[1]{\fontsize{\fsize}{#1\fsize}\selectfont}%
  \ifx\svgwidth\undefined%
    \setlength{\unitlength}{288.22812989bp}%
    \ifx\svgscale\undefined%
      \relax%
    \else%
      \setlength{\unitlength}{\unitlength * \real{\svgscale}}%
    \fi%
  \else%
    \setlength{\unitlength}{\svgwidth}%
  \fi%
  \global\let\svgwidth\undefined%
  \global\let\svgscale\undefined%
  \makeatother%
  \begin{picture}(1,0.29840849)%
    \lineheight{1}%
    \setlength\tabcolsep{0pt}%
    \put(0,0){\includegraphics[width=\unitlength,page=1]{cusp-magnetic-combined.pdf}}%
    \put(0.16962652,0.22345456){\makebox(0,0)[lt]{\lineheight{1.25}\smash{\begin{tabular}[t]{l}$\longleftrightarrow$\end{tabular}}}}%
    \put(0,0){\includegraphics[width=\unitlength,page=2]{cusp-magnetic-combined.pdf}}%
    \put(0.32333859,0.21889271){\makebox(0,0)[lt]{\lineheight{1.25}\smash{\begin{tabular}[t]{l}$1$\end{tabular}}}}%
    \put(0,0){\includegraphics[width=\unitlength,page=3]{cusp-magnetic-combined.pdf}}%
    \put(0.16962656,0.05557532){\makebox(0,0)[lt]{\lineheight{1.25}\smash{\begin{tabular}[t]{l}$\longleftrightarrow$\end{tabular}}}}%
    \put(0,0){\includegraphics[width=\unitlength,page=4]{cusp-magnetic-combined.pdf}}%
    \put(0.79358905,0.23696764){\makebox(0,0)[lt]{\lineheight{1.25}\smash{\begin{tabular}[t]{l}$=$\end{tabular}}}}%
    \put(0,0){\includegraphics[width=\unitlength,page=5]{cusp-magnetic-combined.pdf}}%
    \put(0.6118156,0.27062475){\makebox(0,0)[lt]{\lineheight{1.25}\smash{\begin{tabular}[t]{l}$m$\end{tabular}}}}%
    \put(0,0){\includegraphics[width=\unitlength,page=6]{cusp-magnetic-combined.pdf}}%
    \put(0.63263247,0.156132){\makebox(0,0)[lt]{\lineheight{1.25}\smash{\begin{tabular}[t]{l}$m$\end{tabular}}}}%
    \put(0,0){\includegraphics[width=\unitlength,page=7]{cusp-magnetic-combined.pdf}}%
    \put(0.66385769,0.27062475){\makebox(0,0)[lt]{\lineheight{1.25}\smash{\begin{tabular}[t]{l}$n$\end{tabular}}}}%
    \put(0.8876392,0.27062475){\makebox(0,0)[lt]{\lineheight{1.25}\smash{\begin{tabular}[t]{l}$m+n$\end{tabular}}}}%
    \put(0.79358905,0.12247489){\makebox(0,0)[lt]{\lineheight{1.25}\smash{\begin{tabular}[t]{l}$=$\end{tabular}}}}%
    \put(0,0){\includegraphics[width=\unitlength,page=8]{cusp-magnetic-combined.pdf}}%
    \put(0.63263247,0.04684347){\makebox(0,0)[lt]{\lineheight{1.25}\smash{\begin{tabular}[t]{l}$0$\end{tabular}}}}%
    \put(0,0){\includegraphics[width=\unitlength,page=9]{cusp-magnetic-combined.pdf}}%
    \put(0.79358905,0.01318635){\makebox(0,0)[lt]{\lineheight{1.25}\smash{\begin{tabular}[t]{l}$=$\end{tabular}}}}%
    \put(0,0){\includegraphics[width=\unitlength,page=10]{cusp-magnetic-combined.pdf}}%
    \put(0.90534092,0.156132){\makebox(0,0)[lt]{\lineheight{1.25}\smash{\begin{tabular}[t]{l}$-m$\end{tabular}}}}%
    \put(0,0){\includegraphics[width=\unitlength,page=11]{cusp-magnetic-combined.pdf}}%
    \put(0.32333862,0.05101347){\makebox(0,0)[lt]{\lineheight{1.25}\smash{\begin{tabular}[t]{l}$1$\end{tabular}}}}%
    \put(0,0){\includegraphics[width=\unitlength,page=12]{cusp-magnetic-combined.pdf}}%
    \put(-0.00166535,0.26676691){\makebox(0,0)[lt]{\lineheight{1.25}\smash{\begin{tabular}[t]{l}(a)\end{tabular}}}}%
    \put(0.46150985,0.26676691){\makebox(0,0)[lt]{\lineheight{1.25}\smash{\begin{tabular}[t]{l}(b)\end{tabular}}}}%
  \end{picture}%
\endgroup%

%% file: figs/whisker-expansion.pdf_tex
%% Creator: Inkscape 1.2 (dc2aedaf03, 2022-05-15), www.inkscape.org
%% Accompanies image file 'whisker-expansion.pdf' (pdf, eps, ps)
\begingroup%
  \makeatletter%
  \providecommand\color[2][]{%
    \errmessage{(Inkscape) Color is used for the text in Inkscape, but the package 'color.sty' is not loaded}%
    \renewcommand\color[2][]{}%
  }%
  \providecommand\transparent[1]{%
    \errmessage{(Inkscape) Transparency is used (non-zero) for the text in Inkscape, but the package 'transparent.sty' is not loaded}%
    \renewcommand\transparent[1]{}%
  }%
  \providecommand\rotatebox[2]{#2}%
  \newcommand*\fsize{\dimexpr\f@size pt\relax}%
  \newcommand*\lineheight[1]{\fontsize{\fsize}{#1\fsize}\selectfont}%
  \ifx\svgwidth\undefined%
    \setlength{\unitlength}{394.73152375bp}%
    \ifx\svgscale\undefined%
      \relax%
    \else%
      \setlength{\unitlength}{\unitlength * \real{\svgscale}}%
    \fi%
  \else%
    \setlength{\unitlength}{\svgwidth}%
  \fi%
  \global\let\svgwidth\undefined%
  \global\let\svgscale\undefined%
  \makeatother%
  \begin{picture}(1,0.24883544)%
    \lineheight{1}%
    \setlength\tabcolsep{0pt}%
    \put(0,0){\includegraphics[width=\unitlength,page=1]{whisker-expansion.pdf}}%
    \put(0.15728571,0.16042022){\makebox(0,0)[lt]{\lineheight{1.25}\smash{\begin{tabular}[t]{l}$a$\end{tabular}}}}%
    \put(0.05664963,0.16042022){\makebox(0,0)[lt]{\lineheight{1.25}\smash{\begin{tabular}[t]{l}$b$\end{tabular}}}}%
    \put(0,0){\includegraphics[width=\unitlength,page=2]{whisker-expansion.pdf}}%
    \put(0.27622405,0.20878645){\makebox(0,0)[lt]{\lineheight{1.25}\smash{\begin{tabular}[t]{l}$a$\end{tabular}}}}%
    \put(0.31481118,0.16926781){\makebox(0,0)[lt]{\lineheight{1.25}\smash{\begin{tabular}[t]{l}$b$\end{tabular}}}}%
    \put(0.52018287,0.23520213){\makebox(0,0)[lt]{\lineheight{1.25}\smash{\begin{tabular}[t]{l}$b$\end{tabular}}}}%
    \put(0.51937474,0.18047205){\makebox(0,0)[lt]{\lineheight{1.25}\smash{\begin{tabular}[t]{l}$a$\end{tabular}}}}%
    \put(0.1917073,0.18795719){\makebox(0,0)[lt]{\lineheight{1.25}\smash{\begin{tabular}[t]{l}${}=A$\end{tabular}}}}%
    \put(0.38674182,0.18795719){\makebox(0,0)[lt]{\lineheight{1.25}\smash{\begin{tabular}[t]{l}${}+A^{-1}$\end{tabular}}}}%
    \put(0,0){\includegraphics[width=\unitlength,page=3]{whisker-expansion.pdf}}%
    \put(0.15728571,0.01221809){\makebox(0,0)[lt]{\lineheight{1.25}\smash{\begin{tabular}[t]{l}$a$\end{tabular}}}}%
    \put(0.05664963,0.01221809){\makebox(0,0)[lt]{\lineheight{1.25}\smash{\begin{tabular}[t]{l}$b$\end{tabular}}}}%
    \put(0,0){\includegraphics[width=\unitlength,page=4]{whisker-expansion.pdf}}%
    \put(0.27622405,0.06058432){\makebox(0,0)[lt]{\lineheight{1.25}\smash{\begin{tabular}[t]{l}$a$\end{tabular}}}}%
    \put(0.31481118,0.02106567){\makebox(0,0)[lt]{\lineheight{1.25}\smash{\begin{tabular}[t]{l}$b$\end{tabular}}}}%
    \put(0.52018287,0.08700001){\makebox(0,0)[lt]{\lineheight{1.25}\smash{\begin{tabular}[t]{l}$b$\end{tabular}}}}%
    \put(0.51937474,0.03226991){\makebox(0,0)[lt]{\lineheight{1.25}\smash{\begin{tabular}[t]{l}$a$\end{tabular}}}}%
    \put(0.1917073,0.03975505){\makebox(0,0)[lt]{\lineheight{1.25}\smash{\begin{tabular}[t]{l}${}=A^{-1}$\end{tabular}}}}%
    \put(0.38674182,0.03975505){\makebox(0,0)[lt]{\lineheight{1.25}\smash{\begin{tabular}[t]{l}${}+A$\end{tabular}}}}%
    \put(-0.00121602,0.22573114){\makebox(0,0)[lt]{\lineheight{1.25}\smash{\begin{tabular}[t]{l}(a)\end{tabular}}}}%
    \put(0.62021637,0.22573114){\makebox(0,0)[lt]{\lineheight{1.25}\smash{\begin{tabular}[t]{l}(b)\end{tabular}}}}%
    \put(0.84928124,0.19752736){\makebox(0,0)[lt]{\lineheight{1.25}\smash{\begin{tabular}[t]{l}$=$\end{tabular}}}}%
    \put(0,0){\includegraphics[width=\unitlength,page=5]{whisker-expansion.pdf}}%
    \put(0.72000355,0.22210337){\makebox(0,0)[lt]{\lineheight{1.25}\smash{\begin{tabular}[t]{l}$a$\end{tabular}}}}%
    \put(0,0){\includegraphics[width=\unitlength,page=6]{whisker-expansion.pdf}}%
    \put(0.73434085,0.13850212){\makebox(0,0)[lt]{\lineheight{1.25}\smash{\begin{tabular}[t]{l}$a$\end{tabular}}}}%
    \put(0.75627848,0.22210337){\makebox(0,0)[lt]{\lineheight{1.25}\smash{\begin{tabular}[t]{l}$b$\end{tabular}}}}%
    \put(0.92555606,0.22210337){\makebox(0,0)[lt]{\lineheight{1.25}\smash{\begin{tabular}[t]{l}$a+b$\end{tabular}}}}%
    \put(0.84928124,0.11392611){\makebox(0,0)[lt]{\lineheight{1.25}\smash{\begin{tabular}[t]{l}$=$\end{tabular}}}}%
    \put(0,0){\includegraphics[width=\unitlength,page=7]{whisker-expansion.pdf}}%
    \put(0.7334781,0.05870097){\makebox(0,0)[lt]{\lineheight{1.25}\smash{\begin{tabular}[t]{l}$0$\end{tabular}}}}%
    \put(0,0){\includegraphics[width=\unitlength,page=8]{whisker-expansion.pdf}}%
    \put(0.84928124,0.03412496){\makebox(0,0)[lt]{\lineheight{1.25}\smash{\begin{tabular}[t]{l}$=$\end{tabular}}}}%
    \put(0,0){\includegraphics[width=\unitlength,page=9]{whisker-expansion.pdf}}%
    \put(0.93468147,0.08659556){\makebox(0,0)[lt]{\lineheight{1.25}\smash{\begin{tabular}[t]{l}$-a$\end{tabular}}}}%
    \put(0,0){\includegraphics[width=\unitlength,page=10]{whisker-expansion.pdf}}%
  \end{picture}%
\endgroup%

%% file: figs/vknot_3_5.pdf_tex
%% Creator: Inkscape 1.2 (dc2aedaf03, 2022-05-15), www.inkscape.org
%% Accompanies image file 'vknot_3_5.pdf' (pdf, eps, ps)
\begingroup%
  \makeatletter%
  \providecommand\color[2][]{%
    \errmessage{(Inkscape) Color is used for the text in Inkscape, but the package 'color.sty' is not loaded}%
    \renewcommand\color[2][]{}%
  }%
  \providecommand\transparent[1]{%
    \errmessage{(Inkscape) Transparency is used (non-zero) for the text in Inkscape, but the package 'transparent.sty' is not loaded}%
    \renewcommand\transparent[1]{}%
  }%
  \providecommand\rotatebox[2]{#2}%
  \newcommand*\fsize{\dimexpr\f@size pt\relax}%
  \newcommand*\lineheight[1]{\fontsize{\fsize}{#1\fsize}\selectfont}%
  \ifx\svgwidth\undefined%
    \setlength{\unitlength}{115.65088019bp}%
    \ifx\svgscale\undefined%
      \relax%
    \else%
      \setlength{\unitlength}{\unitlength * \real{\svgscale}}%
    \fi%
  \else%
    \setlength{\unitlength}{\svgwidth}%
  \fi%
  \global\let\svgwidth\undefined%
  \global\let\svgscale\undefined%
  \makeatother%
  \begin{picture}(1,0.99942094)%
    \lineheight{1}%
    \setlength\tabcolsep{0pt}%
    \put(0,0){\includegraphics[width=\unitlength,page=1]{vknot_3_5.pdf}}%
  \end{picture}%
\endgroup%

%% file: figs/kbsm.pdf_tex
%% Creator: Inkscape 1.2 (dc2aedaf03, 2022-05-15), www.inkscape.org
%% Accompanies image file 'kbsm.pdf' (pdf, eps, ps)
\begingroup%
  \makeatletter%
  \providecommand\color[2][]{%
    \errmessage{(Inkscape) Color is used for the text in Inkscape, but the package 'color.sty' is not loaded}%
    \renewcommand\color[2][]{}%
  }%
  \providecommand\transparent[1]{%
    \errmessage{(Inkscape) Transparency is used (non-zero) for the text in Inkscape, but the package 'transparent.sty' is not loaded}%
    \renewcommand\transparent[1]{}%
  }%
  \providecommand\rotatebox[2]{#2}%
  \newcommand*\fsize{\dimexpr\f@size pt\relax}%
  \newcommand*\lineheight[1]{\fontsize{\fsize}{#1\fsize}\selectfont}%
  \ifx\svgwidth\undefined%
    \setlength{\unitlength}{179.81279613bp}%
    \ifx\svgscale\undefined%
      \relax%
    \else%
      \setlength{\unitlength}{\unitlength * \real{\svgscale}}%
    \fi%
  \else%
    \setlength{\unitlength}{\svgwidth}%
  \fi%
  \global\let\svgwidth\undefined%
  \global\let\svgscale\undefined%
  \makeatother%
  \begin{picture}(1,0.52229538)%
    \lineheight{1}%
    \setlength\tabcolsep{0pt}%
    \put(0,0){\includegraphics[width=\unitlength,page=1]{kbsm.pdf}}%
    \put(0.27153498,0.40540645){\makebox(0,0)[lt]{\lineheight{1.25}\smash{\begin{tabular}[t]{l}${}=A$\end{tabular}}}}%
    \put(0.63294564,0.40540645){\makebox(0,0)[lt]{\lineheight{1.25}\smash{\begin{tabular}[t]{l}${}+A^{-1}$\end{tabular}}}}%
    \put(0.27153498,0.08006778){\makebox(0,0)[lt]{\lineheight{1.25}\smash{\begin{tabular}[t]{l}${}=-A^2-A^{-2}$\end{tabular}}}}%
    \put(0,0){\includegraphics[width=\unitlength,page=2]{kbsm.pdf}}%
  \end{picture}%
\endgroup%

%% file: figs/arrow-whisker.pdf_tex
%% Creator: Inkscape 1.2 (dc2aedaf03, 2022-05-15), www.inkscape.org
%% Accompanies image file 'arrow-whisker.pdf' (pdf, eps, ps)
\begingroup%
  \makeatletter%
  \providecommand\color[2][]{%
    \errmessage{(Inkscape) Color is used for the text in Inkscape, but the package 'color.sty' is not loaded}%
    \renewcommand\color[2][]{}%
  }%
  \providecommand\transparent[1]{%
    \errmessage{(Inkscape) Transparency is used (non-zero) for the text in Inkscape, but the package 'transparent.sty' is not loaded}%
    \renewcommand\transparent[1]{}%
  }%
  \providecommand\rotatebox[2]{#2}%
  \newcommand*\fsize{\dimexpr\f@size pt\relax}%
  \newcommand*\lineheight[1]{\fontsize{\fsize}{#1\fsize}\selectfont}%
  \ifx\svgwidth\undefined%
    \setlength{\unitlength}{195.67901015bp}%
    \ifx\svgscale\undefined%
      \relax%
    \else%
      \setlength{\unitlength}{\unitlength * \real{\svgscale}}%
    \fi%
  \else%
    \setlength{\unitlength}{\svgwidth}%
  \fi%
  \global\let\svgwidth\undefined%
  \global\let\svgscale\undefined%
  \makeatother%
  \begin{picture}(1,0.46823443)%
    \lineheight{1}%
    \setlength\tabcolsep{0pt}%
    \put(0,0){\includegraphics[width=\unitlength,page=1]{arrow-whisker.pdf}}%
    \put(0.25394481,0.35640579){\makebox(0,0)[lt]{\lineheight{1.25}\smash{\begin{tabular}[t]{l}${}=A$\end{tabular}}}}%
    \put(0.64737624,0.35640579){\makebox(0,0)[lt]{\lineheight{1.25}\smash{\begin{tabular}[t]{l}${}+A^{-1}$\end{tabular}}}}%
    \put(0,0){\includegraphics[width=\unitlength,page=2]{arrow-whisker.pdf}}%
    \put(0.25394481,0.08019556){\makebox(0,0)[lt]{\lineheight{1.25}\smash{\begin{tabular}[t]{l}${}=A^{-1}$\end{tabular}}}}%
    \put(0.64737624,0.08019556){\makebox(0,0)[lt]{\lineheight{1.25}\smash{\begin{tabular}[t]{l}${}+A$\end{tabular}}}}%
    \put(0,0){\includegraphics[width=\unitlength,page=3]{arrow-whisker.pdf}}%
  \end{picture}%
\endgroup%

%% file: figs/arrow-whisker-disori.pdf_tex
%% Creator: Inkscape 1.2 (dc2aedaf03, 2022-05-15), www.inkscape.org
%% Accompanies image file 'arrow-whisker-disori.pdf' (pdf, eps, ps)
\begingroup%
  \makeatletter%
  \providecommand\color[2][]{%
    \errmessage{(Inkscape) Color is used for the text in Inkscape, but the package 'color.sty' is not loaded}%
    \renewcommand\color[2][]{}%
  }%
  \providecommand\transparent[1]{%
    \errmessage{(Inkscape) Transparency is used (non-zero) for the text in Inkscape, but the package 'transparent.sty' is not loaded}%
    \renewcommand\transparent[1]{}%
  }%
  \providecommand\rotatebox[2]{#2}%
  \newcommand*\fsize{\dimexpr\f@size pt\relax}%
  \newcommand*\lineheight[1]{\fontsize{\fsize}{#1\fsize}\selectfont}%
  \ifx\svgwidth\undefined%
    \setlength{\unitlength}{192.67900826bp}%
    \ifx\svgscale\undefined%
      \relax%
    \else%
      \setlength{\unitlength}{\unitlength * \real{\svgscale}}%
    \fi%
  \else%
    \setlength{\unitlength}{\svgwidth}%
  \fi%
  \global\let\svgwidth\undefined%
  \global\let\svgscale\undefined%
  \makeatother%
  \begin{picture}(1,0.47557075)%
    \lineheight{1}%
    \setlength\tabcolsep{0pt}%
    \put(0,0){\includegraphics[width=\unitlength,page=1]{arrow-whisker-disori.pdf}}%
    \put(0.25789872,0.36200094){\makebox(0,0)[lt]{\lineheight{1.25}\smash{\begin{tabular}[t]{l}${}=A$\end{tabular}}}}%
    \put(0.64967088,0.36200094){\makebox(0,0)[lt]{\lineheight{1.25}\smash{\begin{tabular}[t]{l}${}+A^{-1}$\end{tabular}}}}%
    \put(0,0){\includegraphics[width=\unitlength,page=2]{arrow-whisker-disori.pdf}}%
    \put(0.25789872,0.08149013){\makebox(0,0)[lt]{\lineheight{1.25}\smash{\begin{tabular}[t]{l}${}=A^{-1}$\end{tabular}}}}%
    \put(0.64967088,0.08149013){\makebox(0,0)[lt]{\lineheight{1.25}\smash{\begin{tabular}[t]{l}${}+A$\end{tabular}}}}%
    \put(0,0){\includegraphics[width=\unitlength,page=3]{arrow-whisker-disori.pdf}}%
  \end{picture}%
\endgroup%

%% file: figs/cusp-cancel.pdf_tex
%% Creator: Inkscape 1.2 (dc2aedaf03, 2022-05-15), www.inkscape.org
%% Accompanies image file 'cusp-cancel.pdf' (pdf, eps, ps)
\begingroup%
  \makeatletter%
  \providecommand\color[2][]{%
    \errmessage{(Inkscape) Color is used for the text in Inkscape, but the package 'color.sty' is not loaded}%
    \renewcommand\color[2][]{}%
  }%
  \providecommand\transparent[1]{%
    \errmessage{(Inkscape) Transparency is used (non-zero) for the text in Inkscape, but the package 'transparent.sty' is not loaded}%
    \renewcommand\transparent[1]{}%
  }%
  \providecommand\rotatebox[2]{#2}%
  \newcommand*\fsize{\dimexpr\f@size pt\relax}%
  \newcommand*\lineheight[1]{\fontsize{\fsize}{#1\fsize}\selectfont}%
  \ifx\svgwidth\undefined%
    \setlength{\unitlength}{243.78809386bp}%
    \ifx\svgscale\undefined%
      \relax%
    \else%
      \setlength{\unitlength}{\unitlength * \real{\svgscale}}%
    \fi%
  \else%
    \setlength{\unitlength}{\svgwidth}%
  \fi%
  \global\let\svgwidth\undefined%
  \global\let\svgscale\undefined%
  \makeatother%
  \begin{picture}(1,0.07920342)%
    \lineheight{1}%
    \setlength\tabcolsep{0pt}%
    \put(0,0){\includegraphics[width=\unitlength,page=1]{cusp-cancel.pdf}}%
    \put(0.33824592,0.02867412){\makebox(0,0)[lt]{\lineheight{1.25}\smash{\begin{tabular}[t]{l}$\leftrightarrow$\end{tabular}}}}%
    \put(0.6168614,0.02867412){\makebox(0,0)[lt]{\lineheight{1.25}\smash{\begin{tabular}[t]{l}$\leftrightarrow$\end{tabular}}}}%
  \end{picture}%
\endgroup%

%% file: figs/cusp-whisker.pdf_tex
%% Creator: Inkscape 1.2 (dc2aedaf03, 2022-05-15), www.inkscape.org
%% Accompanies image file 'cusp-whisker.pdf' (pdf, eps, ps)
\begingroup%
  \makeatletter%
  \providecommand\color[2][]{%
    \errmessage{(Inkscape) Color is used for the text in Inkscape, but the package 'color.sty' is not loaded}%
    \renewcommand\color[2][]{}%
  }%
  \providecommand\transparent[1]{%
    \errmessage{(Inkscape) Transparency is used (non-zero) for the text in Inkscape, but the package 'transparent.sty' is not loaded}%
    \renewcommand\transparent[1]{}%
  }%
  \providecommand\rotatebox[2]{#2}%
  \newcommand*\fsize{\dimexpr\f@size pt\relax}%
  \newcommand*\lineheight[1]{\fontsize{\fsize}{#1\fsize}\selectfont}%
  \ifx\svgwidth\undefined%
    \setlength{\unitlength}{112.44733795bp}%
    \ifx\svgscale\undefined%
      \relax%
    \else%
      \setlength{\unitlength}{\unitlength * \real{\svgscale}}%
    \fi%
  \else%
    \setlength{\unitlength}{\svgwidth}%
  \fi%
  \global\let\svgwidth\undefined%
  \global\let\svgscale\undefined%
  \makeatother%
  \begin{picture}(1,0.35006012)%
    \lineheight{1}%
    \setlength\tabcolsep{0pt}%
    \put(0,0){\includegraphics[width=\unitlength,page=1]{cusp-whisker.pdf}}%
    \put(0.42824234,0.23752362){\makebox(0,0)[lt]{\lineheight{1.25}\smash{\begin{tabular}[t]{l}$\longrightarrow$\end{tabular}}}}%
    \put(0,0){\includegraphics[width=\unitlength,page=2]{cusp-whisker.pdf}}%
    \put(0.42824234,0.05917807){\makebox(0,0)[lt]{\lineheight{1.25}\smash{\begin{tabular}[t]{l}$\longrightarrow$\end{tabular}}}}%
    \put(0,0){\includegraphics[width=\unitlength,page=3]{cusp-whisker.pdf}}%
  \end{picture}%
\endgroup%

%% file: figs/vhopf.pdf_tex
%% Creator: Inkscape 1.2 (dc2aedaf03, 2022-05-15), www.inkscape.org
%% Accompanies image file 'vhopf.pdf' (pdf, eps, ps)
\begingroup%
  \makeatletter%
  \providecommand\color[2][]{%
    \errmessage{(Inkscape) Color is used for the text in Inkscape, but the package 'color.sty' is not loaded}%
    \renewcommand\color[2][]{}%
  }%
  \providecommand\transparent[1]{%
    \errmessage{(Inkscape) Transparency is used (non-zero) for the text in Inkscape, but the package 'transparent.sty' is not loaded}%
    \renewcommand\transparent[1]{}%
  }%
  \providecommand\rotatebox[2]{#2}%
  \newcommand*\fsize{\dimexpr\f@size pt\relax}%
  \newcommand*\lineheight[1]{\fontsize{\fsize}{#1\fsize}\selectfont}%
  \ifx\svgwidth\undefined%
    \setlength{\unitlength}{236.34833664bp}%
    \ifx\svgscale\undefined%
      \relax%
    \else%
      \setlength{\unitlength}{\unitlength * \real{\svgscale}}%
    \fi%
  \else%
    \setlength{\unitlength}{\svgwidth}%
  \fi%
  \global\let\svgwidth\undefined%
  \global\let\svgscale\undefined%
  \makeatother%
  \begin{picture}(1,0.27272311)%
    \lineheight{1}%
    \setlength\tabcolsep{0pt}%
    \put(0,0){\includegraphics[width=\unitlength,page=1]{vhopf.pdf}}%
    \put(-0.0010493,0.15831374){\makebox(0,0)[lt]{\lineheight{1.25}\smash{\begin{tabular}[t]{l}$1$\end{tabular}}}}%
    \put(0.25489818,0.15951737){\makebox(0,0)[lt]{\lineheight{1.25}\smash{\begin{tabular}[t]{l}$2$\end{tabular}}}}%
    \put(0,0){\includegraphics[width=\unitlength,page=2]{vhopf.pdf}}%
    \put(0.55166517,0.21087236){\makebox(0,0)[lt]{\lineheight{1.25}\smash{\begin{tabular}[t]{l}$1$\end{tabular}}}}%
    \put(0.50619952,0.17646962){\makebox(0,0)[lt]{\lineheight{1.25}\smash{\begin{tabular}[t]{l}$2$\end{tabular}}}}%
    \put(0.30884618,0.17403143){\makebox(0,0)[lt]{\lineheight{1.25}\smash{\begin{tabular}[t]{l}${}=A$\end{tabular}}}}%
    \put(0,0){\includegraphics[width=\unitlength,page=3]{vhopf.pdf}}%
    \put(0.8323604,0.20297366){\makebox(0,0)[lt]{\lineheight{1.25}\smash{\begin{tabular}[t]{l}$1$\end{tabular}}}}%
    \put(0.88252401,0.25987768){\makebox(0,0)[lt]{\lineheight{1.25}\smash{\begin{tabular}[t]{l}$2$\end{tabular}}}}%
    \put(0.65464306,0.17403143){\makebox(0,0)[lt]{\lineheight{1.25}\smash{\begin{tabular}[t]{l}${}+A^{-1}$\end{tabular}}}}%
    \put(0.30884618,0.00267402){\makebox(0,0)[lt]{\lineheight{1.25}\smash{\begin{tabular}[t]{l}${}=A(-A^2-A^{-2})X_{1,-1}+A^{-1}(-A^2-A^{-2})X_{1,1}.$\end{tabular}}}}%
  \end{picture}%
\endgroup%

%% file: figs/high-genus-pants.pdf_tex
%% Creator: Inkscape 1.2 (dc2aedaf03, 2022-05-15), www.inkscape.org
%% Accompanies image file 'high-genus-pants.pdf' (pdf, eps, ps)
\begingroup%
  \makeatletter%
  \providecommand\color[2][]{%
    \errmessage{(Inkscape) Color is used for the text in Inkscape, but the package 'color.sty' is not loaded}%
    \renewcommand\color[2][]{}%
  }%
  \providecommand\transparent[1]{%
    \errmessage{(Inkscape) Transparency is used (non-zero) for the text in Inkscape, but the package 'transparent.sty' is not loaded}%
    \renewcommand\transparent[1]{}%
  }%
  \providecommand\rotatebox[2]{#2}%
  \newcommand*\fsize{\dimexpr\f@size pt\relax}%
  \newcommand*\lineheight[1]{\fontsize{\fsize}{#1\fsize}\selectfont}%
  \ifx\svgwidth\undefined%
    \setlength{\unitlength}{148.12833374bp}%
    \ifx\svgscale\undefined%
      \relax%
    \else%
      \setlength{\unitlength}{\unitlength * \real{\svgscale}}%
    \fi%
  \else%
    \setlength{\unitlength}{\svgwidth}%
  \fi%
  \global\let\svgwidth\undefined%
  \global\let\svgscale\undefined%
  \makeatother%
  \begin{picture}(1,0.18986017)%
    \lineheight{1}%
    \setlength\tabcolsep{0pt}%
    \put(0,0){\includegraphics[width=\unitlength,page=1]{high-genus-pants.pdf}}%
  \end{picture}%
\endgroup%

%% file: figs/vknot_4_55.pdf_tex
%% Creator: Inkscape 1.2 (dc2aedaf03, 2022-05-15), www.inkscape.org
%% Accompanies image file 'vknot_4_55.pdf' (pdf, eps, ps)
\begingroup%
  \makeatletter%
  \providecommand\color[2][]{%
    \errmessage{(Inkscape) Color is used for the text in Inkscape, but the package 'color.sty' is not loaded}%
    \renewcommand\color[2][]{}%
  }%
  \providecommand\transparent[1]{%
    \errmessage{(Inkscape) Transparency is used (non-zero) for the text in Inkscape, but the package 'transparent.sty' is not loaded}%
    \renewcommand\transparent[1]{}%
  }%
  \providecommand\rotatebox[2]{#2}%
  \newcommand*\fsize{\dimexpr\f@size pt\relax}%
  \newcommand*\lineheight[1]{\fontsize{\fsize}{#1\fsize}\selectfont}%
  \ifx\svgwidth\undefined%
    \setlength{\unitlength}{107.6446804bp}%
    \ifx\svgscale\undefined%
      \relax%
    \else%
      \setlength{\unitlength}{\unitlength * \real{\svgscale}}%
    \fi%
  \else%
    \setlength{\unitlength}{\svgwidth}%
  \fi%
  \global\let\svgwidth\undefined%
  \global\let\svgscale\undefined%
  \makeatother%
  \begin{picture}(1,0.59358427)%
    \lineheight{1}%
    \setlength\tabcolsep{0pt}%
    \put(0,0){\includegraphics[width=\unitlength,page=1]{vknot_4_55.pdf}}%
  \end{picture}%
\endgroup%

%% file: figs/uparrow.pdf_tex
%% Creator: Inkscape 1.2 (dc2aedaf03, 2022-05-15), www.inkscape.org
%% Accompanies image file 'uparrow.pdf' (pdf, eps, ps)
\begingroup%
  \makeatletter%
  \providecommand\color[2][]{%
    \errmessage{(Inkscape) Color is used for the text in Inkscape, but the package 'color.sty' is not loaded}%
    \renewcommand\color[2][]{}%
  }%
  \providecommand\transparent[1]{%
    \errmessage{(Inkscape) Transparency is used (non-zero) for the text in Inkscape, but the package 'transparent.sty' is not loaded}%
    \renewcommand\transparent[1]{}%
  }%
  \providecommand\rotatebox[2]{#2}%
  \newcommand*\fsize{\dimexpr\f@size pt\relax}%
  \newcommand*\lineheight[1]{\fontsize{\fsize}{#1\fsize}\selectfont}%
  \ifx\svgwidth\undefined%
    \setlength{\unitlength}{5.59998772bp}%
    \ifx\svgscale\undefined%
      \relax%
    \else%
      \setlength{\unitlength}{\unitlength * \real{\svgscale}}%
    \fi%
  \else%
    \setlength{\unitlength}{\svgwidth}%
  \fi%
  \global\let\svgwidth\undefined%
  \global\let\svgscale\undefined%
  \makeatother%
  \begin{picture}(1,6.86113171)%
    \lineheight{1}%
    \setlength\tabcolsep{0pt}%
    \put(0,0){\includegraphics[width=\unitlength,page=1]{uparrow.pdf}}%
  \end{picture}%
\endgroup%

%% file: figs/vknot_4_105.pdf_tex
%% Creator: Inkscape 1.2 (dc2aedaf03, 2022-05-15), www.inkscape.org
%% Accompanies image file 'vknot_4_105.pdf' (pdf, eps, ps)
\begingroup%
  \makeatletter%
  \providecommand\color[2][]{%
    \errmessage{(Inkscape) Color is used for the text in Inkscape, but the package 'color.sty' is not loaded}%
    \renewcommand\color[2][]{}%
  }%
  \providecommand\transparent[1]{%
    \errmessage{(Inkscape) Transparency is used (non-zero) for the text in Inkscape, but the package 'transparent.sty' is not loaded}%
    \renewcommand\transparent[1]{}%
  }%
  \providecommand\rotatebox[2]{#2}%
  \newcommand*\fsize{\dimexpr\f@size pt\relax}%
  \newcommand*\lineheight[1]{\fontsize{\fsize}{#1\fsize}\selectfont}%
  \ifx\svgwidth\undefined%
    \setlength{\unitlength}{115.65088019bp}%
    \ifx\svgscale\undefined%
      \relax%
    \else%
      \setlength{\unitlength}{\unitlength * \real{\svgscale}}%
    \fi%
  \else%
    \setlength{\unitlength}{\svgwidth}%
  \fi%
  \global\let\svgwidth\undefined%
  \global\let\svgscale\undefined%
  \makeatother%
  \begin{picture}(1,0.99942094)%
    \lineheight{1}%
    \setlength\tabcolsep{0pt}%
    \put(0,0){\includegraphics[width=\unitlength,page=1]{vknot_4_105.pdf}}%
    \put(0.78199942,0.16299552){\makebox(0,0)[lt]{\lineheight{1.25}\smash{\begin{tabular}[t]{l}$0$\end{tabular}}}}%
    \put(0.13469865,0.14491157){\makebox(0,0)[lt]{\lineheight{1.25}\smash{\begin{tabular}[t]{l}$0$\end{tabular}}}}%
    \put(0.11930561,0.84411446){\makebox(0,0)[lt]{\lineheight{1.25}\smash{\begin{tabular}[t]{l}$0$\end{tabular}}}}%
    \put(0.76856624,0.8181743){\makebox(0,0)[lt]{\lineheight{1.25}\smash{\begin{tabular}[t]{l}$0$\end{tabular}}}}%
    \put(0.37191208,0.58321589){\makebox(0,0)[lt]{\lineheight{1.25}\smash{\begin{tabular}[t]{l}$1$\end{tabular}}}}%
    \put(0.80314792,0.54569561){\makebox(0,0)[lt]{\lineheight{1.25}\smash{\begin{tabular}[t]{l}$1$\end{tabular}}}}%
    \put(0.64831234,0.52427147){\makebox(0,0)[lt]{\lineheight{1.25}\smash{\begin{tabular}[t]{l}$2$\end{tabular}}}}%
    \put(0.45290801,0.28310179){\makebox(0,0)[lt]{\lineheight{1.25}\smash{\begin{tabular}[t]{l}$2$\end{tabular}}}}%
    \put(0.45454117,0.12524336){\makebox(0,0)[lt]{\lineheight{1.25}\smash{\begin{tabular}[t]{l}$1$\end{tabular}}}}%
  \end{picture}%
\endgroup%

%% file: figs/vlink_ex.pdf_tex
%% Creator: Inkscape 1.2 (dc2aedaf03, 2022-05-15), www.inkscape.org
%% Accompanies image file 'vlink_ex.pdf' (pdf, eps, ps)
\begingroup%
  \makeatletter%
  \providecommand\color[2][]{%
    \errmessage{(Inkscape) Color is used for the text in Inkscape, but the package 'color.sty' is not loaded}%
    \renewcommand\color[2][]{}%
  }%
  \providecommand\transparent[1]{%
    \errmessage{(Inkscape) Transparency is used (non-zero) for the text in Inkscape, but the package 'transparent.sty' is not loaded}%
    \renewcommand\transparent[1]{}%
  }%
  \providecommand\rotatebox[2]{#2}%
  \newcommand*\fsize{\dimexpr\f@size pt\relax}%
  \newcommand*\lineheight[1]{\fontsize{\fsize}{#1\fsize}\selectfont}%
  \ifx\svgwidth\undefined%
    \setlength{\unitlength}{115.65088019bp}%
    \ifx\svgscale\undefined%
      \relax%
    \else%
      \setlength{\unitlength}{\unitlength * \real{\svgscale}}%
    \fi%
  \else%
    \setlength{\unitlength}{\svgwidth}%
  \fi%
  \global\let\svgwidth\undefined%
  \global\let\svgscale\undefined%
  \makeatother%
  \begin{picture}(1,0.99942094)%
    \lineheight{1}%
    \setlength\tabcolsep{0pt}%
    \put(0,0){\includegraphics[width=\unitlength,page=1]{vlink_ex.pdf}}%
    \put(0.67985369,0.7985937){\makebox(0,0)[lt]{\lineheight{1.25}\smash{\begin{tabular}[t]{l}$0$\end{tabular}}}}%
    \put(0.70049635,0.1910168){\makebox(0,0)[lt]{\lineheight{1.25}\smash{\begin{tabular}[t]{l}$0$\end{tabular}}}}%
    \put(0.09979598,0.83539609){\makebox(0,0)[lt]{\lineheight{1.25}\smash{\begin{tabular}[t]{l}$0$\end{tabular}}}}%
    \put(0.0936697,0.13315278){\makebox(0,0)[lt]{\lineheight{1.25}\smash{\begin{tabular}[t]{l}$0$\end{tabular}}}}%
    \put(0.28245126,0.556567){\makebox(0,0)[lt]{\lineheight{1.25}\smash{\begin{tabular}[t]{l}$1$\end{tabular}}}}%
    \put(0.86656984,0.54396418){\makebox(0,0)[lt]{\lineheight{1.25}\smash{\begin{tabular}[t]{l}$1$\end{tabular}}}}%
    \put(0.62272373,0.52805138){\makebox(0,0)[lt]{\lineheight{1.25}\smash{\begin{tabular}[t]{l}$-1$\end{tabular}}}}%
    \put(0.34495876,0.28369308){\makebox(0,0)[lt]{\lineheight{1.25}\smash{\begin{tabular}[t]{l}$2$\end{tabular}}}}%
    \put(0.34391988,0.1164731){\makebox(0,0)[lt]{\lineheight{1.25}\smash{\begin{tabular}[t]{l}$1$\end{tabular}}}}%
  \end{picture}%
\endgroup%

%% file: figs/altern_A_smooth.pdf_tex
%% Creator: Inkscape 1.2 (dc2aedaf03, 2022-05-15), www.inkscape.org
%% Accompanies image file 'altern_A_smooth.pdf' (pdf, eps, ps)
\begingroup%
  \makeatletter%
  \providecommand\color[2][]{%
    \errmessage{(Inkscape) Color is used for the text in Inkscape, but the package 'color.sty' is not loaded}%
    \renewcommand\color[2][]{}%
  }%
  \providecommand\transparent[1]{%
    \errmessage{(Inkscape) Transparency is used (non-zero) for the text in Inkscape, but the package 'transparent.sty' is not loaded}%
    \renewcommand\transparent[1]{}%
  }%
  \providecommand\rotatebox[2]{#2}%
  \newcommand*\fsize{\dimexpr\f@size pt\relax}%
  \newcommand*\lineheight[1]{\fontsize{\fsize}{#1\fsize}\selectfont}%
  \ifx\svgwidth\undefined%
    \setlength{\unitlength}{135.40418646bp}%
    \ifx\svgscale\undefined%
      \relax%
    \else%
      \setlength{\unitlength}{\unitlength * \real{\svgscale}}%
    \fi%
  \else%
    \setlength{\unitlength}{\svgwidth}%
  \fi%
  \global\let\svgwidth\undefined%
  \global\let\svgscale\undefined%
  \makeatother%
  \begin{picture}(1,0.33131298)%
    \lineheight{1}%
    \setlength\tabcolsep{0pt}%
    \put(0,0){\includegraphics[width=\unitlength,page=1]{altern_A_smooth.pdf}}%
    \put(0.44308395,0.15068266){\makebox(0,0)[lt]{\lineheight{1.25}\smash{\begin{tabular}[t]{l}$\longrightarrow$\end{tabular}}}}%
  \end{picture}%
\endgroup%

%% file: figs/checkerboard-orient.pdf_tex
%% Creator: Inkscape 1.2 (dc2aedaf03, 2022-05-15), www.inkscape.org
%% Accompanies image file 'checkerboard-orient.pdf' (pdf, eps, ps)
\begingroup%
  \makeatletter%
  \providecommand\color[2][]{%
    \errmessage{(Inkscape) Color is used for the text in Inkscape, but the package 'color.sty' is not loaded}%
    \renewcommand\color[2][]{}%
  }%
  \providecommand\transparent[1]{%
    \errmessage{(Inkscape) Transparency is used (non-zero) for the text in Inkscape, but the package 'transparent.sty' is not loaded}%
    \renewcommand\transparent[1]{}%
  }%
  \providecommand\rotatebox[2]{#2}%
  \newcommand*\fsize{\dimexpr\f@size pt\relax}%
  \newcommand*\lineheight[1]{\fontsize{\fsize}{#1\fsize}\selectfont}%
  \ifx\svgwidth\undefined%
    \setlength{\unitlength}{297.23869864bp}%
    \ifx\svgscale\undefined%
      \relax%
    \else%
      \setlength{\unitlength}{\unitlength * \real{\svgscale}}%
    \fi%
  \else%
    \setlength{\unitlength}{\svgwidth}%
  \fi%
  \global\let\svgwidth\undefined%
  \global\let\svgscale\undefined%
  \makeatother%
  \begin{picture}(1,0.13078079)%
    \lineheight{1}%
    \setlength\tabcolsep{0pt}%
    \put(0,0){\includegraphics[width=\unitlength,page=1]{checkerboard-orient.pdf}}%
    \put(0.47890662,0.05854444){\makebox(0,0)[lt]{\lineheight{1.25}\smash{\begin{tabular}[t]{l}$\longrightarrow$\end{tabular}}}}%
    \put(0,0){\includegraphics[width=\unitlength,page=2]{checkerboard-orient.pdf}}%
    \put(0.17107292,0.05854444){\makebox(0,0)[lt]{\lineheight{1.25}\smash{\begin{tabular}[t]{l}and\end{tabular}}}}%
    \put(0.79178641,0.05854444){\makebox(0,0)[lt]{\lineheight{1.25}\smash{\begin{tabular}[t]{l}or\end{tabular}}}}%
  \end{picture}%
\endgroup%

%% file: figs/checkerboard-cabling.pdf_tex
%% Creator: Inkscape 1.2 (dc2aedaf03, 2022-05-15), www.inkscape.org
%% Accompanies image file 'checkerboard-cabling.pdf' (pdf, eps, ps)
\begingroup%
  \makeatletter%
  \providecommand\color[2][]{%
    \errmessage{(Inkscape) Color is used for the text in Inkscape, but the package 'color.sty' is not loaded}%
    \renewcommand\color[2][]{}%
  }%
  \providecommand\transparent[1]{%
    \errmessage{(Inkscape) Transparency is used (non-zero) for the text in Inkscape, but the package 'transparent.sty' is not loaded}%
    \renewcommand\transparent[1]{}%
  }%
  \providecommand\rotatebox[2]{#2}%
  \newcommand*\fsize{\dimexpr\f@size pt\relax}%
  \newcommand*\lineheight[1]{\fontsize{\fsize}{#1\fsize}\selectfont}%
  \ifx\svgwidth\undefined%
    \setlength{\unitlength}{240.0003978bp}%
    \ifx\svgscale\undefined%
      \relax%
    \else%
      \setlength{\unitlength}{\unitlength * \real{\svgscale}}%
    \fi%
  \else%
    \setlength{\unitlength}{\svgwidth}%
  \fi%
  \global\let\svgwidth\undefined%
  \global\let\svgscale\undefined%
  \makeatother%
  \begin{picture}(1,0.3749994)%
    \lineheight{1}%
    \setlength\tabcolsep{0pt}%
    \put(0,0){\includegraphics[width=\unitlength,page=1]{checkerboard-cabling.pdf}}%
    \put(0.45975063,0.17950513){\makebox(0,0)[lt]{\lineheight{1.25}\smash{\begin{tabular}[t]{l}$\longleftrightarrow$\end{tabular}}}}%
    \put(0,0){\includegraphics[width=\unitlength,page=2]{checkerboard-cabling.pdf}}%
  \end{picture}%
\endgroup%

%% file: figs/nonreduced.pdf_tex
%% Creator: Inkscape 1.2 (dc2aedaf03, 2022-05-15), www.inkscape.org
%% Accompanies image file 'nonreduced.pdf' (pdf, eps, ps)
\begingroup%
  \makeatletter%
  \providecommand\color[2][]{%
    \errmessage{(Inkscape) Color is used for the text in Inkscape, but the package 'color.sty' is not loaded}%
    \renewcommand\color[2][]{}%
  }%
  \providecommand\transparent[1]{%
    \errmessage{(Inkscape) Transparency is used (non-zero) for the text in Inkscape, but the package 'transparent.sty' is not loaded}%
    \renewcommand\transparent[1]{}%
  }%
  \providecommand\rotatebox[2]{#2}%
  \newcommand*\fsize{\dimexpr\f@size pt\relax}%
  \newcommand*\lineheight[1]{\fontsize{\fsize}{#1\fsize}\selectfont}%
  \ifx\svgwidth\undefined%
    \setlength{\unitlength}{201.47668999bp}%
    \ifx\svgscale\undefined%
      \relax%
    \else%
      \setlength{\unitlength}{\unitlength * \real{\svgscale}}%
    \fi%
  \else%
    \setlength{\unitlength}{\svgwidth}%
  \fi%
  \global\let\svgwidth\undefined%
  \global\let\svgscale\undefined%
  \makeatother%
  \begin{picture}(1,0.30456771)%
    \lineheight{1}%
    \setlength\tabcolsep{0pt}%
    \put(0,0){\includegraphics[width=\unitlength,page=1]{nonreduced.pdf}}%
  \end{picture}%
\endgroup%

%% file: figs/vknot_3_7_not_reduced.pdf_tex
%% Creator: Inkscape 1.2 (dc2aedaf03, 2022-05-15), www.inkscape.org
%% Accompanies image file 'vknot_3_7_not_reduced.pdf' (pdf, eps, ps)
\begingroup%
  \makeatletter%
  \providecommand\color[2][]{%
    \errmessage{(Inkscape) Color is used for the text in Inkscape, but the package 'color.sty' is not loaded}%
    \renewcommand\color[2][]{}%
  }%
  \providecommand\transparent[1]{%
    \errmessage{(Inkscape) Transparency is used (non-zero) for the text in Inkscape, but the package 'transparent.sty' is not loaded}%
    \renewcommand\transparent[1]{}%
  }%
  \providecommand\rotatebox[2]{#2}%
  \newcommand*\fsize{\dimexpr\f@size pt\relax}%
  \newcommand*\lineheight[1]{\fontsize{\fsize}{#1\fsize}\selectfont}%
  \ifx\svgwidth\undefined%
    \setlength{\unitlength}{121.67412562bp}%
    \ifx\svgscale\undefined%
      \relax%
    \else%
      \setlength{\unitlength}{\unitlength * \real{\svgscale}}%
    \fi%
  \else%
    \setlength{\unitlength}{\svgwidth}%
  \fi%
  \global\let\svgwidth\undefined%
  \global\let\svgscale\undefined%
  \makeatother%
  \begin{picture}(1,0.9685712)%
    \lineheight{1}%
    \setlength\tabcolsep{0pt}%
    \put(0,0){\includegraphics[width=\unitlength,page=1]{vknot_3_7_not_reduced.pdf}}%
    \put(0.94401442,0.63901953){\makebox(0,0)[lt]{\lineheight{1.25}\smash{\begin{tabular}[t]{l}$b$\end{tabular}}}}%
    \put(0.2345645,0.94361932){\makebox(0,0)[lt]{\lineheight{1.25}\smash{\begin{tabular}[t]{l}$r$\end{tabular}}}}%
  \end{picture}%
\endgroup%

%% file: figs/loop-split.pdf_tex
%% Creator: Inkscape 1.2 (dc2aedaf03, 2022-05-15), www.inkscape.org
%% Accompanies image file 'loop-split.pdf' (pdf, eps, ps)
\begingroup%
  \makeatletter%
  \providecommand\color[2][]{%
    \errmessage{(Inkscape) Color is used for the text in Inkscape, but the package 'color.sty' is not loaded}%
    \renewcommand\color[2][]{}%
  }%
  \providecommand\transparent[1]{%
    \errmessage{(Inkscape) Transparency is used (non-zero) for the text in Inkscape, but the package 'transparent.sty' is not loaded}%
    \renewcommand\transparent[1]{}%
  }%
  \providecommand\rotatebox[2]{#2}%
  \newcommand*\fsize{\dimexpr\f@size pt\relax}%
  \newcommand*\lineheight[1]{\fontsize{\fsize}{#1\fsize}\selectfont}%
  \ifx\svgwidth\undefined%
    \setlength{\unitlength}{223.50329011bp}%
    \ifx\svgscale\undefined%
      \relax%
    \else%
      \setlength{\unitlength}{\unitlength * \real{\svgscale}}%
    \fi%
  \else%
    \setlength{\unitlength}{\svgwidth}%
  \fi%
  \global\let\svgwidth\undefined%
  \global\let\svgscale\undefined%
  \makeatother%
  \begin{picture}(1,0.51675193)%
    \lineheight{1}%
    \setlength\tabcolsep{0pt}%
    \put(0,0){\includegraphics[width=\unitlength,page=1]{loop-split.pdf}}%
    \put(0.41192652,0.25054016){\makebox(0,0)[lt]{\lineheight{1.25}\smash{\begin{tabular}[t]{l}$A\longleftrightarrow B$\end{tabular}}}}%
  \end{picture}%
\endgroup%

%% file: figs/loop-split-cases.pdf_tex
%% Creator: Inkscape 1.2 (dc2aedaf03, 2022-05-15), www.inkscape.org
%% Accompanies image file 'loop-split-cases.pdf' (pdf, eps, ps)
\begingroup%
  \makeatletter%
  \providecommand\color[2][]{%
    \errmessage{(Inkscape) Color is used for the text in Inkscape, but the package 'color.sty' is not loaded}%
    \renewcommand\color[2][]{}%
  }%
  \providecommand\transparent[1]{%
    \errmessage{(Inkscape) Transparency is used (non-zero) for the text in Inkscape, but the package 'transparent.sty' is not loaded}%
    \renewcommand\transparent[1]{}%
  }%
  \providecommand\rotatebox[2]{#2}%
  \newcommand*\fsize{\dimexpr\f@size pt\relax}%
  \newcommand*\lineheight[1]{\fontsize{\fsize}{#1\fsize}\selectfont}%
  \ifx\svgwidth\undefined%
    \setlength{\unitlength}{215.17900006bp}%
    \ifx\svgscale\undefined%
      \relax%
    \else%
      \setlength{\unitlength}{\unitlength * \real{\svgscale}}%
    \fi%
  \else%
    \setlength{\unitlength}{\svgwidth}%
  \fi%
  \global\let\svgwidth\undefined%
  \global\let\svgscale\undefined%
  \makeatother%
  \begin{picture}(1,0.43034949)%
    \lineheight{1}%
    \setlength\tabcolsep{0pt}%
    \put(0,0){\includegraphics[width=\unitlength,page=1]{loop-split-cases.pdf}}%
    \put(0.87000363,0.00592018){\makebox(0,0)[lt]{\lineheight{1.25}\smash{\begin{tabular}[t]{l}$a$\end{tabular}}}}%
    \put(0.93864694,0.00592018){\makebox(0,0)[lt]{\lineheight{1.25}\smash{\begin{tabular}[t]{l}$b$\end{tabular}}}}%
    \put(0,0){\includegraphics[width=\unitlength,page=2]{loop-split-cases.pdf}}%
    \put(0.86965809,0.14752869){\makebox(0,0)[lt]{\lineheight{1.25}\smash{\begin{tabular}[t]{l}$b$\end{tabular}}}}%
    \put(0.93909188,0.14752004){\makebox(0,0)[lt]{\lineheight{1.25}\smash{\begin{tabular}[t]{l}$a$\end{tabular}}}}%
    \put(0,0){\includegraphics[width=\unitlength,page=3]{loop-split-cases.pdf}}%
    \put(0.10737795,0.1347121){\makebox(0,0)[lt]{\lineheight{1.25}\smash{\begin{tabular}[t]{l}$a$\end{tabular}}}}%
    \put(0.10635556,0.00990874){\makebox(0,0)[lt]{\lineheight{1.25}\smash{\begin{tabular}[t]{l}$b$\end{tabular}}}}%
    \put(0,0){\includegraphics[width=\unitlength,page=4]{loop-split-cases.pdf}}%
    \put(0.14382082,0.29208214){\makebox(0,0)[lt]{\lineheight{1.25}\smash{\begin{tabular}[t]{l}$a$\end{tabular}}}}%
    \put(0.01106128,0.29208214){\makebox(0,0)[lt]{\lineheight{1.25}\smash{\begin{tabular}[t]{l}$b$\end{tabular}}}}%
    \put(0.27972841,0.32865497){\makebox(0,0)[lt]{\lineheight{1.25}\smash{\begin{tabular}[t]{l}$:$\end{tabular}}}}%
    \put(0.60962243,0.32865497){\makebox(0,0)[lt]{\lineheight{1.25}\smash{\begin{tabular}[t]{l}$A\longleftrightarrow B$\end{tabular}}}}%
    \put(0.27972841,0.07694698){\makebox(0,0)[lt]{\lineheight{1.25}\smash{\begin{tabular}[t]{l}$:$\end{tabular}}}}%
    \put(0.60962243,0.07694698){\makebox(0,0)[lt]{\lineheight{1.25}\smash{\begin{tabular}[t]{l}$A\longleftrightarrow B$\end{tabular}}}}%
    \put(0,0){\includegraphics[width=\unitlength,page=5]{loop-split-cases.pdf}}%
    \put(0.86485374,0.40230824){\makebox(0,0)[lt]{\lineheight{1.25}\smash{\begin{tabular}[t]{l}$a$\end{tabular}}}}%
    \put(0.93719434,0.40222726){\makebox(0,0)[lt]{\lineheight{1.25}\smash{\begin{tabular}[t]{l}$b$\end{tabular}}}}%
    \put(0,0){\includegraphics[width=\unitlength,page=6]{loop-split-cases.pdf}}%
    \put(0.8722855,0.26007462){\makebox(0,0)[lt]{\lineheight{1.25}\smash{\begin{tabular}[t]{l}$b$\end{tabular}}}}%
    \put(0.93512423,0.25986894){\makebox(0,0)[lt]{\lineheight{1.25}\smash{\begin{tabular}[t]{l}$a$\end{tabular}}}}%
    \put(0,0){\includegraphics[width=\unitlength,page=7]{loop-split-cases.pdf}}%
  \end{picture}%
\endgroup%

%% file: figs/altern-whisker-expansion.pdf_tex
%% Creator: Inkscape 1.2 (dc2aedaf03, 2022-05-15), www.inkscape.org
%% Accompanies image file 'altern-whisker-expansion.pdf' (pdf, eps, ps)
\begingroup%
  \makeatletter%
  \providecommand\color[2][]{%
    \errmessage{(Inkscape) Color is used for the text in Inkscape, but the package 'color.sty' is not loaded}%
    \renewcommand\color[2][]{}%
  }%
  \providecommand\transparent[1]{%
    \errmessage{(Inkscape) Transparency is used (non-zero) for the text in Inkscape, but the package 'transparent.sty' is not loaded}%
    \renewcommand\transparent[1]{}%
  }%
  \providecommand\rotatebox[2]{#2}%
  \newcommand*\fsize{\dimexpr\f@size pt\relax}%
  \newcommand*\lineheight[1]{\fontsize{\fsize}{#1\fsize}\selectfont}%
  \ifx\svgwidth\undefined%
    \setlength{\unitlength}{192.67900886bp}%
    \ifx\svgscale\undefined%
      \relax%
    \else%
      \setlength{\unitlength}{\unitlength * \real{\svgscale}}%
    \fi%
  \else%
    \setlength{\unitlength}{\svgwidth}%
  \fi%
  \global\let\svgwidth\undefined%
  \global\let\svgscale\undefined%
  \makeatother%
  \begin{picture}(1,1.0355026)%
    \lineheight{1}%
    \setlength\tabcolsep{0pt}%
    \put(0,0){\includegraphics[width=\unitlength,page=1]{altern-whisker-expansion.pdf}}%
    \put(0.84907206,0.16368897){\makebox(0,0)[lt]{\lineheight{1.25}\smash{\begin{tabular}[t]{l}$a$\end{tabular}}}}%
    \put(0.92986019,0.1635985){\makebox(0,0)[lt]{\lineheight{1.25}\smash{\begin{tabular}[t]{l}$b$\end{tabular}}}}%
    \put(0,0){\includegraphics[width=\unitlength,page=2]{altern-whisker-expansion.pdf}}%
    \put(0.85737166,0.0061717){\makebox(0,0)[lt]{\lineheight{1.25}\smash{\begin{tabular}[t]{l}$b$\end{tabular}}}}%
    \put(0.92754835,0.00594205){\makebox(0,0)[lt]{\lineheight{1.25}\smash{\begin{tabular}[t]{l}$a$\end{tabular}}}}%
    \put(0,0){\includegraphics[width=\unitlength,page=3]{altern-whisker-expansion.pdf}}%
    \put(0.44148015,0.39896997){\makebox(0,0)[lt]{\lineheight{1.25}\smash{\begin{tabular}[t]{l}$a$\end{tabular}}}}%
    \put(0.45772866,0.3217399){\makebox(0,0)[lt]{\lineheight{1.25}\smash{\begin{tabular}[t]{l}$b$\end{tabular}}}}%
    \put(0,0){\includegraphics[width=\unitlength,page=4]{altern-whisker-expansion.pdf}}%
    \put(0.60896191,0.39878488){\makebox(0,0)[lt]{\lineheight{1.25}\smash{\begin{tabular}[t]{l}$b$\end{tabular}}}}%
    \put(0.60297976,0.3218142){\makebox(0,0)[lt]{\lineheight{1.25}\smash{\begin{tabular}[t]{l}$a$\end{tabular}}}}%
    \put(0,0){\includegraphics[width=\unitlength,page=5]{altern-whisker-expansion.pdf}}%
    \put(0.16061542,0.04059131){\makebox(0,0)[lt]{\lineheight{1.25}\smash{\begin{tabular}[t]{l}$a$\end{tabular}}}}%
    \put(0.01235296,0.04059131){\makebox(0,0)[lt]{\lineheight{1.25}\smash{\begin{tabular}[t]{l}$b$\end{tabular}}}}%
    \put(0,0){\includegraphics[width=\unitlength,page=6]{altern-whisker-expansion.pdf}}%
    \put(0.16061542,0.31943153){\makebox(0,0)[lt]{\lineheight{1.25}\smash{\begin{tabular}[t]{l}$a$\end{tabular}}}}%
    \put(0.01235296,0.31943153){\makebox(0,0)[lt]{\lineheight{1.25}\smash{\begin{tabular}[t]{l}$b$\end{tabular}}}}%
    \put(0,0){\includegraphics[width=\unitlength,page=7]{altern-whisker-expansion.pdf}}%
    \put(0.16061542,0.60053242){\makebox(0,0)[lt]{\lineheight{1.25}\smash{\begin{tabular}[t]{l}$a$\end{tabular}}}}%
    \put(0.01235296,0.60053242){\makebox(0,0)[lt]{\lineheight{1.25}\smash{\begin{tabular}[t]{l}$b$\end{tabular}}}}%
    \put(0.25789872,0.64137602){\makebox(0,0)[lt]{\lineheight{1.25}\smash{\begin{tabular}[t]{l}${}=A$\end{tabular}}}}%
    \put(0.65745585,0.64137602){\makebox(0,0)[lt]{\lineheight{1.25}\smash{\begin{tabular}[t]{l}${}+A^{-1}$\end{tabular}}}}%
    \put(0.25789872,0.36027496){\makebox(0,0)[lt]{\lineheight{1.25}\smash{\begin{tabular}[t]{l}${}=A^{-1}$\end{tabular}}}}%
    \put(0.65745585,0.36027496){\makebox(0,0)[lt]{\lineheight{1.25}\smash{\begin{tabular}[t]{l}${}+A$\end{tabular}}}}%
    \put(0,0){\includegraphics[width=\unitlength,page=8]{altern-whisker-expansion.pdf}}%
    \put(0.16061542,0.88108928){\makebox(0,0)[lt]{\lineheight{1.25}\smash{\begin{tabular}[t]{l}$a$\end{tabular}}}}%
    \put(0.01235296,0.88108928){\makebox(0,0)[lt]{\lineheight{1.25}\smash{\begin{tabular}[t]{l}$b$\end{tabular}}}}%
    \put(0,0){\includegraphics[width=\unitlength,page=9]{altern-whisker-expansion.pdf}}%
    \put(0.44148015,0.96062763){\makebox(0,0)[lt]{\lineheight{1.25}\smash{\begin{tabular}[t]{l}$a$\end{tabular}}}}%
    \put(0.45772866,0.88339753){\makebox(0,0)[lt]{\lineheight{1.25}\smash{\begin{tabular}[t]{l}$b$\end{tabular}}}}%
    \put(0.25789872,0.92193279){\makebox(0,0)[lt]{\lineheight{1.25}\smash{\begin{tabular}[t]{l}${}=A$\end{tabular}}}}%
    \put(0.65745585,0.92193279){\makebox(0,0)[lt]{\lineheight{1.25}\smash{\begin{tabular}[t]{l}${}+A^{-1}$\end{tabular}}}}%
    \put(0.25789872,0.08143473){\makebox(0,0)[lt]{\lineheight{1.25}\smash{\begin{tabular}[t]{l}${}=A^{-1}$\end{tabular}}}}%
    \put(0.65745585,0.08143473){\makebox(0,0)[lt]{\lineheight{1.25}\smash{\begin{tabular}[t]{l}${}+A$\end{tabular}}}}%
    \put(0,0){\includegraphics[width=\unitlength,page=10]{altern-whisker-expansion.pdf}}%
    \put(0.60896191,0.96044256){\makebox(0,0)[lt]{\lineheight{1.25}\smash{\begin{tabular}[t]{l}$b$\end{tabular}}}}%
    \put(0.60297976,0.88347189){\makebox(0,0)[lt]{\lineheight{1.25}\smash{\begin{tabular}[t]{l}$a$\end{tabular}}}}%
    \put(0,0){\includegraphics[width=\unitlength,page=11]{altern-whisker-expansion.pdf}}%
    \put(0.84907194,0.72363012){\makebox(0,0)[lt]{\lineheight{1.25}\smash{\begin{tabular}[t]{l}$a$\end{tabular}}}}%
    \put(0.92986008,0.72353968){\makebox(0,0)[lt]{\lineheight{1.25}\smash{\begin{tabular}[t]{l}$b$\end{tabular}}}}%
    \put(0,0){\includegraphics[width=\unitlength,page=12]{altern-whisker-expansion.pdf}}%
    \put(0.85737155,0.56611293){\makebox(0,0)[lt]{\lineheight{1.25}\smash{\begin{tabular}[t]{l}$b$\end{tabular}}}}%
    \put(0.92754824,0.56588323){\makebox(0,0)[lt]{\lineheight{1.25}\smash{\begin{tabular}[t]{l}$a$\end{tabular}}}}%
  \end{picture}%
\endgroup%

%% file: figs/mutation.pdf_tex
%% Creator: Inkscape 1.2 (dc2aedaf03, 2022-05-15), www.inkscape.org
%% Accompanies image file 'mutation.pdf' (pdf, eps, ps)
\begingroup%
  \makeatletter%
  \providecommand\color[2][]{%
    \errmessage{(Inkscape) Color is used for the text in Inkscape, but the package 'color.sty' is not loaded}%
    \renewcommand\color[2][]{}%
  }%
  \providecommand\transparent[1]{%
    \errmessage{(Inkscape) Transparency is used (non-zero) for the text in Inkscape, but the package 'transparent.sty' is not loaded}%
    \renewcommand\transparent[1]{}%
  }%
  \providecommand\rotatebox[2]{#2}%
  \newcommand*\fsize{\dimexpr\f@size pt\relax}%
  \newcommand*\lineheight[1]{\fontsize{\fsize}{#1\fsize}\selectfont}%
  \ifx\svgwidth\undefined%
    \setlength{\unitlength}{190.61342227bp}%
    \ifx\svgscale\undefined%
      \relax%
    \else%
      \setlength{\unitlength}{\unitlength * \real{\svgscale}}%
    \fi%
  \else%
    \setlength{\unitlength}{\svgwidth}%
  \fi%
  \global\let\svgwidth\undefined%
  \global\let\svgscale\undefined%
  \makeatother%
  \begin{picture}(1,0.29881173)%
    \lineheight{1}%
    \setlength\tabcolsep{0pt}%
    \put(0,0){\includegraphics[width=\unitlength,page=1]{mutation.pdf}}%
    \put(0.54476664,0.13476508){\makebox(0,0)[lt]{\lineheight{1.25}\smash{\begin{tabular}[t]{l}$\longrightarrow$\end{tabular}}}}%
  \end{picture}%
\endgroup%

%% file: figs/mutation-ex.pdf_tex
%% Creator: Inkscape 1.2 (dc2aedaf03, 2022-05-15), www.inkscape.org
%% Accompanies image file 'mutation-ex.pdf' (pdf, eps, ps)
\begingroup%
  \makeatletter%
  \providecommand\color[2][]{%
    \errmessage{(Inkscape) Color is used for the text in Inkscape, but the package 'color.sty' is not loaded}%
    \renewcommand\color[2][]{}%
  }%
  \providecommand\transparent[1]{%
    \errmessage{(Inkscape) Transparency is used (non-zero) for the text in Inkscape, but the package 'transparent.sty' is not loaded}%
    \renewcommand\transparent[1]{}%
  }%
  \providecommand\rotatebox[2]{#2}%
  \newcommand*\fsize{\dimexpr\f@size pt\relax}%
  \newcommand*\lineheight[1]{\fontsize{\fsize}{#1\fsize}\selectfont}%
  \ifx\svgwidth\undefined%
    \setlength{\unitlength}{250.39011591bp}%
    \ifx\svgscale\undefined%
      \relax%
    \else%
      \setlength{\unitlength}{\unitlength * \real{\svgscale}}%
    \fi%
  \else%
    \setlength{\unitlength}{\svgwidth}%
  \fi%
  \global\let\svgwidth\undefined%
  \global\let\svgscale\undefined%
  \makeatother%
  \begin{picture}(1,0.24226101)%
    \lineheight{1}%
    \setlength\tabcolsep{0pt}%
    \put(0,0){\includegraphics[width=\unitlength,page=1]{mutation-ex.pdf}}%
  \end{picture}%
\endgroup%

%% file: figs/mutation-expan.pdf_tex
%% Creator: Inkscape 1.2 (dc2aedaf03, 2022-05-15), www.inkscape.org
%% Accompanies image file 'mutation-expan.pdf' (pdf, eps, ps)
\begingroup%
  \makeatletter%
  \providecommand\color[2][]{%
    \errmessage{(Inkscape) Color is used for the text in Inkscape, but the package 'color.sty' is not loaded}%
    \renewcommand\color[2][]{}%
  }%
  \providecommand\transparent[1]{%
    \errmessage{(Inkscape) Transparency is used (non-zero) for the text in Inkscape, but the package 'transparent.sty' is not loaded}%
    \renewcommand\transparent[1]{}%
  }%
  \providecommand\rotatebox[2]{#2}%
  \newcommand*\fsize{\dimexpr\f@size pt\relax}%
  \newcommand*\lineheight[1]{\fontsize{\fsize}{#1\fsize}\selectfont}%
  \ifx\svgwidth\undefined%
    \setlength{\unitlength}{118.04779124bp}%
    \ifx\svgscale\undefined%
      \relax%
    \else%
      \setlength{\unitlength}{\unitlength * \real{\svgscale}}%
    \fi%
  \else%
    \setlength{\unitlength}{\svgwidth}%
  \fi%
  \global\let\svgwidth\undefined%
  \global\let\svgscale\undefined%
  \makeatother%
  \begin{picture}(1,0.23138091)%
    \lineheight{1}%
    \setlength\tabcolsep{0pt}%
    \put(0,0){\includegraphics[width=\unitlength,page=1]{mutation-expan.pdf}}%
  \end{picture}%
\endgroup%

%% file: figs/pd.pdf_tex
%% Creator: Inkscape 1.2 (dc2aedaf03, 2022-05-15), www.inkscape.org
%% Accompanies image file 'pd.pdf' (pdf, eps, ps)
\begingroup%
  \makeatletter%
  \providecommand\color[2][]{%
    \errmessage{(Inkscape) Color is used for the text in Inkscape, but the package 'color.sty' is not loaded}%
    \renewcommand\color[2][]{}%
  }%
  \providecommand\transparent[1]{%
    \errmessage{(Inkscape) Transparency is used (non-zero) for the text in Inkscape, but the package 'transparent.sty' is not loaded}%
    \renewcommand\transparent[1]{}%
  }%
  \providecommand\rotatebox[2]{#2}%
  \newcommand*\fsize{\dimexpr\f@size pt\relax}%
  \newcommand*\lineheight[1]{\fontsize{\fsize}{#1\fsize}\selectfont}%
  \ifx\svgwidth\undefined%
    \setlength{\unitlength}{226.4240603bp}%
    \ifx\svgscale\undefined%
      \relax%
    \else%
      \setlength{\unitlength}{\unitlength * \real{\svgscale}}%
    \fi%
  \else%
    \setlength{\unitlength}{\svgwidth}%
  \fi%
  \global\let\svgwidth\undefined%
  \global\let\svgscale\undefined%
  \makeatother%
  \begin{picture}(1,0.26359842)%
    \lineheight{1}%
    \setlength\tabcolsep{0pt}%
    \put(0,0){\includegraphics[width=\unitlength,page=1]{pd.pdf}}%
    \put(0.22510532,0.10346758){\makebox(0,0)[lt]{\lineheight{1.25}\smash{\begin{tabular}[t]{l}$a$\end{tabular}}}}%
    \put(0.0231646,0.10346758){\makebox(0,0)[lt]{\lineheight{1.25}\smash{\begin{tabular}[t]{l}$d$\end{tabular}}}}%
    \put(0.22510532,0.22271297){\makebox(0,0)[lt]{\lineheight{1.25}\smash{\begin{tabular}[t]{l}$b$\end{tabular}}}}%
    \put(0.0231646,0.22271297){\makebox(0,0)[lt]{\lineheight{1.25}\smash{\begin{tabular}[t]{l}$c$\end{tabular}}}}%
    \put(-0.00017666,0.00423983){\makebox(0,0)[lt]{\lineheight{1.25}\smash{\begin{tabular}[t]{l}\texttt{Xp[$a$,$b$,$c$,$d$]}\end{tabular}}}}%
    \put(0,0){\includegraphics[width=\unitlength,page=2]{pd.pdf}}%
    \put(0.64908894,0.10346758){\makebox(0,0)[lt]{\lineheight{1.25}\smash{\begin{tabular}[t]{l}$a$\end{tabular}}}}%
    \put(0.44714819,0.10346758){\makebox(0,0)[lt]{\lineheight{1.25}\smash{\begin{tabular}[t]{l}$d$\end{tabular}}}}%
    \put(0.64908894,0.22271297){\makebox(0,0)[lt]{\lineheight{1.25}\smash{\begin{tabular}[t]{l}$b$\end{tabular}}}}%
    \put(0.44714819,0.22271297){\makebox(0,0)[lt]{\lineheight{1.25}\smash{\begin{tabular}[t]{l}$c$\end{tabular}}}}%
    \put(0.42380667,0.00423983){\makebox(0,0)[lt]{\lineheight{1.25}\smash{\begin{tabular}[t]{l}\texttt{Xm[$a$,$b$,$c$,$d$]}\end{tabular}}}}%
    \put(0,0){\includegraphics[width=\unitlength,page=3]{pd.pdf}}%
    \put(0.93395293,0.10346758){\makebox(0,0)[lt]{\lineheight{1.25}\smash{\begin{tabular}[t]{l}$a$\end{tabular}}}}%
    \put(0.93395293,0.22271297){\makebox(0,0)[lt]{\lineheight{1.25}\smash{\begin{tabular}[t]{l}$b$\end{tabular}}}}%
    \put(0.84779015,0.00423983){\makebox(0,0)[lt]{\lineheight{1.25}\smash{\begin{tabular}[t]{l}\texttt{P[$a$,$b$]}\end{tabular}}}}%
  \end{picture}%
\endgroup%

%% file: figs/pd-2-1.pdf_tex
%% Creator: Inkscape 1.2 (dc2aedaf03, 2022-05-15), www.inkscape.org
%% Accompanies image file 'pd-2-1.pdf' (pdf, eps, ps)
\begingroup%
  \makeatletter%
  \providecommand\color[2][]{%
    \errmessage{(Inkscape) Color is used for the text in Inkscape, but the package 'color.sty' is not loaded}%
    \renewcommand\color[2][]{}%
  }%
  \providecommand\transparent[1]{%
    \errmessage{(Inkscape) Transparency is used (non-zero) for the text in Inkscape, but the package 'transparent.sty' is not loaded}%
    \renewcommand\transparent[1]{}%
  }%
  \providecommand\rotatebox[2]{#2}%
  \newcommand*\fsize{\dimexpr\f@size pt\relax}%
  \newcommand*\lineheight[1]{\fontsize{\fsize}{#1\fsize}\selectfont}%
  \ifx\svgwidth\undefined%
    \setlength{\unitlength}{75.76282644bp}%
    \ifx\svgscale\undefined%
      \relax%
    \else%
      \setlength{\unitlength}{\unitlength * \real{\svgscale}}%
    \fi%
  \else%
    \setlength{\unitlength}{\svgwidth}%
  \fi%
  \global\let\svgwidth\undefined%
  \global\let\svgscale\undefined%
  \makeatother%
  \begin{picture}(1,0.97188897)%
    \lineheight{1}%
    \setlength\tabcolsep{0pt}%
    \put(0,0){\includegraphics[width=\unitlength,page=1]{pd-2-1.pdf}}%
    \put(0.726968,0.3025923){\makebox(0,0)[lt]{\lineheight{1.25}\smash{\begin{tabular}[t]{l}1\end{tabular}}}}%
    \put(0.88271161,0.50874669){\makebox(0,0)[lt]{\lineheight{1.25}\smash{\begin{tabular}[t]{l}2\end{tabular}}}}%
    \put(0.79297129,0.66765767){\makebox(0,0)[lt]{\lineheight{1.25}\smash{\begin{tabular}[t]{l}3\end{tabular}}}}%
    \put(0.55667192,0.61977713){\makebox(0,0)[lt]{\lineheight{1.25}\smash{\begin{tabular}[t]{l}4\end{tabular}}}}%
    \put(0.35584549,0.62145074){\makebox(0,0)[lt]{\lineheight{1.25}\smash{\begin{tabular}[t]{l}4\end{tabular}}}}%
    \put(0.137383,0.66929616){\makebox(0,0)[lt]{\lineheight{1.25}\smash{\begin{tabular}[t]{l}3\end{tabular}}}}%
    \put(0.02342724,0.50775446){\makebox(0,0)[lt]{\lineheight{1.25}\smash{\begin{tabular}[t]{l}1\end{tabular}}}}%
    \put(0.20805485,0.30685353){\makebox(0,0)[lt]{\lineheight{1.25}\smash{\begin{tabular}[t]{l}2\end{tabular}}}}%
  \end{picture}%
\endgroup%